\documentclass[11pt]{article}

\usepackage{amsmath,amsthm,amsfonts,amssymb,color, hyperref}

\newtheorem{theorem}{Theorem}[section]
\newtheorem{corollary}[theorem]{Corollary}

\newtheorem{proposition}[theorem]{Proposition}
\newtheorem{lemma}[theorem]{Lemma}
\newtheorem{definition}[theorem]{Definition}
\newtheorem{remark}[theorem]{Remark}

\def\cF{\mathcal{F}}
\def\cH{\mathcal{H}}
\def\cP{\mathcal{P}}

\def\bD{\mathbb{D}}
\def\bE{\mathbb{E}}
\def\bP{\mathbb{P}}
\def\bR{\mathbb{R}}
\def\R{\mathbb{R}}

\def\e{\varepsilon}

\numberwithin{equation}{section}

\begin{document}

\title{The hyperbolic Anderson model: Moment estimates of the Malliavin derivatives and  applications\footnote{Dedicated to Professor Istv\'an Gy\"ongy on the occasion of his seventieth birthday}}

\author{Raluca M. Balan\footnote{University of Ottawa, Department of Mathematics and Statistics, STEM Building, 150 Louis-Pasteur Private,
Ottawa, ON, K1N 6N5, Canada. E-mail: rbalan@uottawa.ca. Research
supported by a grant from the Natural Sciences and Engineering
Research Council of Canada.},   ~ David Nualart\footnote{Department
of Mathematics, University of Kansas, 405 Snow Hall, Lawrence, KS,
66045, USA. Email: nualart@ku.edu. Supported by NSF Grant DMS
1811181.},  ~ Llu\'is Quer-Sardanyons\footnote{Departament de
Matem\`atiques, Universitat Aut\`onoma de Barcelona, 08193,
Cerdanyola del Vall\`es, Catalonia, Spain. E-mail: quer@mat.uab.cat.
Supported by
the grant PGC2018-097848-B-I00 (Ministerio de Econom\'ia y Competitividad).}, ~ \\
Guangqu Zheng\footnote{Corresponding author. School of Mathematics, The University of
Edinburgh,  James Clerk Maxwell Building,
Peter Guthrie Tait Road, Edinburgh, EH9 3FD, United Kingdom. Email:
zhengguangqu@gmail.com}}

\maketitle

 \vspace{-0.8cm}

 \begin{abstract} In this article, we study the hyperbolic Anderson model
 driven by a space-time \emph{colored}  Gaussian homogeneous noise with spatial
 dimension $d=1,2$. Under mild assumptions, we provide $L^p$-estimates of  the
 iterated Malliavin derivative of the solution in terms of the fundamental
 solution of the wave solution.  To achieve this goal, we rely heavily on
 the \emph{Wiener chaos expansion} of the solution.

  Our first  application are  \emph{quantitative central limit theorems} for  spatial averages of the solution to the  hyperbolic Anderson model, where the  rates of convergence are described by  the total variation distance.
 These quantitative results  have been elusive so far due to the temporal correlation of the noise blocking us from using the It\^o calculus. A  \emph{novel} ingredient to overcome this difficulty is  the   \emph{second-order Gaussian Poincar\'e inequality}   coupled with the  application of the aforementioned $L^p$-estimates of  the first two   Malliavin derivatives. Besides, we provide  the corresponding functional central limit theorems.

 As  a second  application,  we establish the  absolute continuity of the law for the
 hyperbolic Anderson model.  The $L^p$-estimates of Malliavin derivatives
 are crucial ingredients to verify a local version of Bouleau-Hirsch criterion for  absolute
 continuity. Our approach  substantially simplifies the  arguments for the one-dimensional case,
 which has been studied in the recent   work by  Balan, Quer-Sardanyons and Song (2019).
 \end{abstract}



\bigskip
\bigskip

\allowdisplaybreaks

\pagebreak

\section{Introduction}

One of the main tools of modern stochastic analysis is Malliavin calculus. To put it short, this is a  differential calculus on a   Gaussian space that  represents an infinite dimensional generalization of the usual analytical concepts  on an Euclidean space. The Malliavin calculus (also known as the stochastic calculus of variations) was initiated by Paul Malliavin \cite{M76} to give  a probabilistic proof of H\"ormander's ``sum of squares'' theorem. It has been further developed by Stroock, Bismut, Watanabe and others.   One of the main applications of Malliavin calculus is the study of  regularity properties of probability laws, for example, the laws of the solutions to certain stochastic differential equations and stochastic partial differential equations (SPDEs), see \emph{e.g.} \cite[Chapter 2]{Nualart06}. The Malliavin calculus is also useful in formulating and  interpreting  stochastic  (partial) differential equations when the solution is not adapted to a Brownian filtration, which is the case of  SPDEs driven by a Gaussian noise that is colored in time.

Recently, the  Malliavin calculus has found another important
application in the work of   Nualart and Ortiz-Latorre \cite{NOL08},
which paved the road for \emph{Stein to meet Malliavin}. The authors
of  \cite{NOL08} applied the Malliavin calculus (notably the
integration by parts formula) to characterize the convergence in law
of a sequence of multiple Wiener integrals, and they were able to
give new proofs for the fourth moment theorems of Nualart, Peccati
and Tudor \cite{FMT05, PT05}. Soon after the work \cite{NOL08},
Nourdin and Peccati  combined  Malliavin   calculus and   Stein's
method of normal approximation to quantify the fourth moment
theorem. Their work \cite{NP09} marked the birth of the so-called
Malliavin-Stein approach.  This combination works admirably well,
partially because one of  the  fundamental ingredients in Stein's
method---the so-called Stein's lemma \eqref{S_lem}---that
characterizes the normal distribution,
 is nothing else but a particular case of the integration by parts formula \eqref{IbP} in Malliavin calculus. We refer interested readers to \cite[Section 1.2]{GZ18} for a friendly introduction to this approach.

\medskip
The central object   of study in this paper is
 the stochastic wave equation with  \emph{linear}  Gaussian multiplicative noise (in \emph{Skorokhod sense}):
\begin{align}
\label{wave}
\begin{cases}
\dfrac{\partial^2 u}{\partial t^2} =\Delta u+u\dot{W} \\
u(0,x) = 1, \quad \dfrac{\partial u}{\partial t}(0,x)=0\end{cases} ~\text{on $\bR_{+} \times \bR^d$ for $d\in\{1,2\}$,}
\end{align}
 where $\Delta$ is the Laplacian in space variables and the Gaussian noise $\dot{W}$ has the following correlation structure
\[
\bE\big[ \dot{W}(t,x) \dot{W}(s,y)     \big] =\gamma_0(t-s) \gamma(x-y),
\]
with the following standing assumptions:\label{page2}
\begin{itemize}
\item[(i)] $\gamma_0:\bR \to [0,\infty]$ is  locally integrable and non-negative definite;

\item[(ii)] $\gamma$ is a non-negative  and  non-negative definite measure on $\R^d$
whose spectral measure $\mu$\footnote{The spectral measure $\mu$  of $\gamma$ is a tempered measure on $\R^d$ such that    $\gamma=\cF \mu$, that is, $\gamma$ is the Fourier transform of $\mu$,  and its existence is guaranteed by the Bochner-Schwarz theorem.}
satisfies
  \emph{Dalang's condition}:
\begin{align}
\qquad\qquad \quad   \int_{\bR^d}\frac{1}{1+|\xi|^2}\mu(d\xi)<\infty,
\label{DC}
\end{align}
where  $|\xi|$ denotes the Euclidean norm of $\xi\in\bR^d$.
\end{itemize}
An important example of the temporal correlation is the Riesz kernel $\gamma_0(t)=|t|^{-\alpha_0}$ for some $\alpha_0\in (0,1)$ (with $\gamma_0(0)=\infty$).

Equation \eqref{wave}  is also known in the literature as the \emph{hyperbolic Anderson model},
by analogy with the parabolic Anderson model
 in which the wave operator is replaced by
  the heat operator.
The noise $\dot{W}$ can be formally realized as an isonormal Gaussian process
$W=\{W(\phi): \phi \in \mathcal{H}\}$ and   here $\mathcal{H}$ is a Hilbert space that is the  completion of the set $C^\infty_c\big(\bR_+\times\bR^d)$ of infinitely differentiable  functions with compact support  under the inner product
\begin{align}
\langle \phi, \psi \rangle_{\mathcal{H}} &= \int_{\bR_{+}^2 \times \bR^{2d}}\phi(t,x)\psi(s,y)\gamma_0(t-s)\gamma(x-y)dtdxdsdy\label{defcH}  \\
&=  \int_{\bR_{+}^2 } dt ds \gamma_0(t-s)  \int_{\bR^d} dx \phi(t,x)   \big[\psi(s,\bullet) \ast\gamma\big](x)   \label{defcH2},
\end{align}
where we  write $\gamma(x)$ for the density of $\gamma$ if it exists and we shall use the definition \eqref{defcH2} instead of \eqref{defcH} when $\gamma$ is a measure. In \eqref{defcH2},   $\ast$ denotes the convolution in the space variable and $\gamma_0(t)= \gamma_0(-t)$ for $t<0$.
We denote by $\mathcal{H}^{\otimes p}$ the $p$th tensor product of $\cH$ for $p\in\mathbb{N}^\ast$, see Section \ref{sec2} for more details.

As mentioned before, the existence of a temporal correlation $\gamma_0$ prevents us from defining  equation \eqref{wave} in the It\^o sense due to a  lack of the  martingale structure. In the recent work \cite{BS17} by Balan and Song, the following   results are established using Malliavin calculus. Let $G_t$ denote the fundamental solution to the corresponding deterministic wave equation, that is, for $(t,z)\in (0,\infty)\times\bR^d$,
\begin{align}\label{Green}
G_t(z) := \begin{cases} \dfrac{1}{2} \mathbf{1}_{\{ | z| < t\}} \quad &\text{if $d=1$}; \\
\dfrac{1}{2\pi  \sqrt{ t^2 - | z|^2}}  \mathbf{1}_{\{ | z| < t\}} \quad &\text{if $d=2$}.
\end{cases}
\end{align}
To ease the notation, we will stick to the convention that
 \begin{align}
 \text{$G_t(z) =0$ when $t\leq 0$.} \label{rule1}
 \end{align}

\begin{definition} {\rm{Fix} $d\in\{1,2\}$.  We say that a square-integrable process $u  = \{ u(t,x): (t,x)\in\bR_+\times\bR^d\}$     is a \emph{mild Skorokhod solution} to the hyperbolic Anderson model \eqref{wave} if  $u$ has a jointly measurable  modification $($still denoted by $u$$)$ such that $\sup\{ \bE[u(t,x)^2 ]: (t,x)\in[0,T]\times\bR^d\} < \infty$ for any finite $T$; and for any $t>0$ and $x\in\bR^d$, the following equality holds in $L^2(\Omega)$:
\[
u(t,x)=1 + \int_0^t \int_{\bR^d} G_{t-s}(x-y) u(s,y) W(ds,dy),
\]
where the above stochastic integral is understood in the \emph{Skorokhod sense} and the process  $(s,y)\in\bR_+\times\bR^d\longmapsto   \mathbf{1}_{(0,t)}(s) G_{t-s}(x-y)u(s,y)$ is Skorokhod integrable.  See Definition 5.1 in \cite{BS17} and Definition 1.1 in \cite{BQS}.}
\end{definition}

It has been proved in \cite[Section 5]{BS17} that  equation
\eqref{wave} admits a unique mild  Skorokhod solution $u$ with the
following Wiener chaos expansion:
\begin{align}\label{WCE}
u(t,x) =  1 + \sum_{n\geq 1}I_n\big( \widetilde{f}_{t,x,n} \big),
\end{align}
 where $I_n$ denotes the $n$th multiple Wiener integral associated to the isonormal Gaussian process $W$ (see Section \ref{sec2} for more details),
 $f_{t,x,n}\in\mathcal{H}^{\otimes n}$ is defined by (with the convention \eqref{rule1} in mind)
 \begin{equation}
f_{t,x,n}(t_1,x_1,\dots   ,t_n, x_n):=G_{t-t_{1}}(x-x_{1})  G_{t_1-t_2}(x_1-x_2) \cdots G_{t_{n-1}-t_n}(x_{n-1}-x_n),   \label{eq:3}
\end{equation}
  and $\widetilde{f}_{t,x,n}$ is the canonical symmetrization of $f_{t,x,n}\in\mathcal{H}^{\otimes n}$ given by
  \begin{equation}
\widetilde{f}_{t,x,n}(t_1,x_1,\dots ,t_n, x_n):=\frac{1}{n!}
\sum_{\sigma \in \mathfrak{S}_n}f_{t,x,n}(t_{\sigma(1)}, x_{\sigma(1)}, \dots,t_{\sigma(n)}, x_{\sigma(n)}), \label{eq:3wt}
\end{equation}
where the sum in \eqref{eq:3wt} runs  over    $\mathfrak{S}_n$, the set of permutations on $\{1,2,\dots, n\}$. For example, $f_{t,x,1}(t_1,x_1) =G_{t-t_1}(x-x_1)$ and
\[
\widetilde{f}_{t,x,2}(t_1,x_1, t_2, x_2) = \frac{1}{2} \Big( G_{t-t_1}(x-x_1) G_{t_1-t_2}(x_1-x_2) +  G_{t-t_2}(x-x_2) G_{t_2-t_1}(x_2-x_1)   \Big).
\]
We would like to point out that in the presence of temporal correlation, there is no developed solution theory for the nonlinear wave equation (replacing $u \dot{W}$ in \eqref{wave} by $\sigma(u) \dot{W}$ for some deterministic Lipschitz function $\sigma:\R\to\R$). We regard this as a totally different   problem. \\

Now let us introduce the following hypothesis when $d=2$:
\begin{align*}
{\bf (H1)} \begin{cases}
  &\text{(\texttt{a})   $\gamma\in L^\ell(\bR^2)$ for some $\ell\in(1,\infty)$,}
  \\
& \text{(\texttt{b})   $\gamma(x)= |x|^{-\beta}$ for some $\beta\in(0,2)$,}
\\
   &\text{(\texttt{c}) $\gamma(x_1,x_2) =\gamma_1(x_1)\gamma_2(x_2)$, where  $\gamma_i(x_i) = |x_i|^{-\beta_i}$ or   $\gamma_i\in L^{\ell_i}(\bR)$   } \\
 &\qquad   \text{for some $0 < \beta_i < 1  < \ell_i <+\infty$, $i=1,2$. }
\end{cases}
\end{align*}

 \begin{remark}  {\rm
 (i)
 Note that  condition (\texttt{a}) for $d=2$ is slightly stronger than  Dalang's condition \eqref{DC}. In fact, when $d=2$,  the paper \cite{KZ99} pointed out that  Dalang's condition \eqref{DC} is equivalent to
 \begin{align}\label{DCeq}
 \int_{|x|\leq 1} \ln( |x|^{-1}  ) \gamma(x)dx < \infty;
 \end{align}
 let $\ell^\star = \frac{\ell}{\ell-1}$ and $0< \varepsilon <  1/\ell^\star$, then there is some $\delta\in(0,1)$ and a constant $C_\e$ such that $\ln ( |x|^{-1})  \leq C_\e |x|^{-\e}$ for any $|x|\leq \delta$, from which we deduce that
  \begin{align*}
 \int_{|x|\leq 1} \ln( |x|^{-1}  ) \gamma(x)dx  & \leq  \ln(\delta^{-1})  \int_{ \delta < |x|\leq 1}  \gamma(x)dx + C_\e \int_{|x|\leq \delta} |x|^{-\e} \gamma(x)dx \\
 &\leq  \ln(\delta^{-1})  \int_{ \delta < |x|\leq 1}  \gamma(x)dx + C_\e  \| \gamma\|_{L^\ell(\bR^2)} \left( \int_{|x|\leq \delta} |x|^{-\e\ell^\star}dx \right)^{1/\ell^\star}<\infty.
 \end{align*}
(ii)  The case (\texttt{c}) in  Hypothesis  ${\bf (H1)}$ is a mixture of cases (\texttt{a}) and (\texttt{b}).  Accordingly,  more examples of the noise $\dot{W}$  arise. In the space variables,
 $W$ can behave like a fractional Brownian sheet with Hurst indices greater than $1/2$ in both directions, i.e. $\gamma(x_1,x_2)=|x_1|^{2H_1-2}|x_2|^{2H_2-2}$ for some  $H_1,H_2 \in (1/2,1)$.

 \noindent
 (iii) For $d=1$ we just assume that $\gamma$ is a non-negative and non-negative definite measure on $\R$. In this case
 (see, for instance, Remark 10 of \cite{Dalang99})  Dalang's condition is always satisfied.
 }
 \end{remark}

  Under  Hypothesis  $\bf (H1)$,  we will  state our first main result --- the $L^p(\Omega)$ estimates of the Malliavin derivatives of $u(t,x)$.
 The first Malliavin derivative $Du(t,x)$ is a random element in the Hilbert space $\cH$, the completion of $C^\infty_c\big(\bR_+\times\bR^d)$ under the inner product \eqref{defcH}; as the space $\mathcal{H}$ contains generalized functions,   it is not  clear at first sight whether $(s,y) \longmapsto D_{s,y}u(t,x)$ is a  (random) function. The higher-order Malliavin derivative  $D^{m} u(t,x)$  is a random element in $\cH^{\otimes m}$ for $m\geq 1$, see Section \ref{sec2} for more details.

 \medskip

Let us first fix some notation.

\medskip

\noindent\textbf{Notation A.}  (1)  We write  $a\lesssim b$ to mean $a\leq Kb$ for some immaterial constant $K>0$.

\noindent{(2)}  We write    $\| X\|_p =  \big(\bE [ |X | ^p  ]\big)^{1/p}$ to denote the $L^p(\Omega)$-norm of $X$ for  $p\in[1,\infty)$.

\noindent{(3)} When  $p$
is a positive integer, we often write $\pmb{z_p} = (z_1,
\dots, z_p)$ for points in $\bR_+^p$ or  $\bR^{dp}$,  and $d\pmb{z_p}=dz_1 \cdots dz_p$,
$\mu(d\pmb{z_p}) = \mu(dz_1)\cdots \mu(dz_p)$. For a  function $h:
(\R_+\times \R^d)^p\rightarrow \R$ with $p\geq 2$, we often write
\[
h(\pmb{s_p}, \pmb{y_p}) = h(s_1, \dots,  s_p, y_1,\dots , y_p) = h(s_1, y_1, \dots, s_p, y_p),
\]
which shall not cause any confusion.
   For $m\in\{1,\dots, p-1\}$ and $(\pmb{s_m}, \pmb{y_m})\in\bR_+^m\times\bR^{dm}$, the expression $h(\pmb{s_m}, \pmb{y_m};\bullet)$ stands for the function
\[
(t_1,x_1, \dots, t_{p-m}, x_{p-m}) \mapsto h(s_1, y_1, \dots, s_m, y_m, t_1, x_1, \dots , t_{p-m}, x_{p-m})=h(\pmb{s_m}, \pmb{y_m};\pmb{t_{p-m}}, \pmb{x_{p-m}}).
\]

Now, with the above notation in mind,  we are in the position to state the first main result\footnote{{In higher dimension $(d\geq 3)$, the fundamental wave solution is a uniform measure supported on certain surfaces, then
the Malliavin derivative $Du(t,x)$ is expected to be merely a random measure instead of being a random function. In this case, the expression $D_{s,y}u(t,x)$ does not make sense; see also the recent article \cite{NXZ21} for related discussions.    }}.

\begin{theorem}\label{MR1}
Let $d\in \{1,2\}$ and suppose that Hypothesis $\bf (H1)$ holds if $d=2$.
 Then,  for any $(t,x) \in \R_+\times \R^d$, the random variable $u(t,x)$ belongs to $\mathbb{D}^{\infty}$ $($see Section \ref{sec21}$)$. Moreover,  for any integer $m\geq 1$, the $m$th Malliavin derivative   $D^mu(t,x)$ is a random  symmetric function denoted by
\[
(\pmb{s_m}, \pmb{y_m})=(s_1,y_1, \dots, s_m, y_m)\longmapsto D_{s_1,y_1} D_{s_2, y_2}\ldots D_{s_m, y_m} u(t,x) = D^m_{\pmb{s_m}, \pmb{y_m}} u(t,x),
\]
and for any $p\in[2,\infty)$, we have, for almost all  $(\pmb{s_m}, \pmb{y_m}) \in  [0,t]^m \times \R^{md}$,
\begin{align}\label{goalz}
m! \widetilde{f}_{t,x,m}(\pmb{s_m}, \pmb{y_m})  \leq    \big\| D^m_{\pmb{s_m}, \pmb{y_m}} u(t,x) \big\|_p \lesssim \widetilde{f}_{t,x,m}(\pmb{s_m}, \pmb{y_m}),
\end{align}
where the   constant in the upper bound only depends on $(p,t,\gamma_0, \gamma,m)$ and is increasing in $t$. Moreover, $D^m u(t,x)$ has a measurable modification.
\end{theorem}

Throughout this paper, we will work with the measurable modifications of $D u(t,x)$ and $D^2 u(t,x)$ given by Theorem \ref{MR1}, which are still   denoted by $D u(t,x), D^2 u(t,x)$ respectively.

 \medskip

In this paper, we will present two  applications of Theorem
\ref{MR1}.  Our first  application are \emph{quantitative central
limit theorems}  (CLTs) for the spatial averages of the solution to
\eqref{wave}, which have been elusive so far due to the temporal
correlation of the noise  preventing the use of   It\^o calculus
approach. A  \emph{novel} ingredient to overcome this difficulty is
the so-called \emph{second-order Gaussian Poincar\'e inequality} in
an improved form.  We will address these CLT results in  Section
\ref{sec11}.  While in Section \ref{sec12},  as the second
application, we establish  the  absolute continuity of the law of the solution
to equation \eqref{wave} using the $L^p$-estimates of Malliavin
derivatives that are crucial to establish a local version of
Bouleau-Hirsch criterion \cite{BH86}.

\subsection{Gaussian fluctuation of spatial averages}  \label{sec11}

Spatial averages of SPDEs have recently  attracted considerable interest. It was  Huang, Nualart and Viitasaari who first studied the  fluctuation of  spatial statistics and established  a central limit theorem  for a nonlinear SPDE in   \cite{HNV20}. More precisely, they considered the following one-dimensional stochastic heat equation
\begin{align}\label{SHE}
\frac{\partial u}{\partial t}  = \frac{1}{2}\Delta u + \sigma(u) \dot{W}
\end{align}
on $\bR_+\times\bR$, where $\dot{W}$ is a space-time Gaussian white
noise, with constant initial condition $u(0,\bullet)=1$ and the
nonlinearity $\sigma:\bR\to\bR$  is a  Lipschitz function. In view
of the localization property of its mild formulation (in the Walsh
sense \cite{Walsh}),
\begin{equation} \label{eq1}
u(t,x)=1+\int_0^t \int_{\bR} p_{t-s}(x-y) \sigma\big( u(s,y) \big) W(ds, dy),
\end{equation}
with $p_t$ denoting the heat kernel\footnote{$p_t(x) =(2\pi t)^{-d/2} e^{-|x|^2/(2t)}  $ for $t>0$ and $x\in\bR^d$; in \eqref{eq1}, $d=1$.}, one can regard $u(t,x)$ and $u(t,y)$ as weakly dependent random variables for $x,y$ far apart so that the integral
 \[
  \int_{-R}^R \big[ u(t,x) -1 \big] dx
\]
can be roughly understood as a sum of weakly dependent random variables.  Therefore, it is very natural to expect  Gaussian fluctuations when $R$ tends to infinity.

\medskip

Let us stop now to briefly  fix some notation to facilitate our discussion.

\medskip

\noindent\textbf{Notation B.} (1)  For $t>0$, we define, with $B_R:=\{ x\in\bR^d: |x| \leq R\}$,
\begin{align}\label{FRT}
F_R(t) := \int_{B_R}  \big[ u(t,x) -1 \big] dx \quad{\rm and}\quad \sigma_R(t) = \sqrt{ \text{Var}\big(F_R(t) \big) }.
\end{align}
\noindent(2) We write  $f(R)\sim g(R)$  to mean that   $f(R)/ g(R)$ converges to some positive constant    as $R\to \infty$.

\noindent(3) For two real random variables $X, Y$ with distribution measures $\mu, \nu$ respectively,  the total variation distance between $X, Y$ (or $\mu,\nu$) is defined to be
\begin{align}\label{TVdef}
d_{\rm TV}( X, Y) = \sup_{B} \big\vert \mu(B) -\nu(B)\vert,
\end{align}
where the supremum runs over all Borel set $B\subset \bR$. The total
variation distance is well known to induce a stronger topology than
that of convergence in distribution, see \cite[Appendix C]{blue}.

\noindent(4) We define the following quantities for future reference:
\begin{align}\label{KBD}
\omega_1=2, \quad \omega_2=\pi, \quad {\rm and}\quad
\kappa_{\beta,d} := \int_{\bR^{2d}} dxdy |x-y|^{-\beta} \mathbf{1}_{B_1}(x) \mathbf{1}_{B_1}(y)~\text{for $\beta \in (0,d)$}.
\end{align}
\noindent(5)  For an integer $m\geq 1$ and $p\in[1,\infty)$, we say
$F\in\mathbb{D}^{m,p}$  if $F$ is $m$-times Malliavin differentiable
random variable in $L^p(\Omega)$ and $\bE\big[ \| D^j
F\|_{\cH^{\otimes j}}^p \big] <\infty $ for every $j=1,\dots, m$;
see Section \ref{sec21} for more details.

Now let us illustrate the strategy in  \cite{HNV20}: {(For this reference, $d=1$)}
\begin{itemize}
\item The authors first rewrite $F_R(t) = \delta(V_{t, R})$ with the random kernel
\[
V_{t,R}(s,y) = \sigma(u(s,y) ) \int_{B_R} p_{t-s}(x-y) dx,
\]
where $\delta$ denotes the
Skorokhod integral,
the adjoint of the Malliavin derivative $D$.
\item  By standard computations, they obtained $\sigma^2_R(t)\sim R$.
\item If $F=\delta(v)\in\mathbb{D}^{1,2}$ is a centered random variable with variance one, for some $v$ in the domain of $\delta$,  the (univariate) Malliavin-Stein bound (see \cite[Proposition 2.2]{HNV20}) ensures that
$d_{\rm TV}(  F, Z  )\leq 2\sqrt{ \text{Var}( \langle DF, v\rangle_\cH)}$ for $Z\sim N(0,1)$.
\item Combining the above points, one can see that the obtention of a quantitative CLT is reduced to the computation of $ \text{Var}( \langle DF_R(t), V_{t,R} \rangle_\cH)$.
\end{itemize}
Because the driving noise is white in time as considered in \cite{HNV20},   tools from It\^o calculus (Clark-Ocone formula, Burkholder's inequality, \emph{etc.}) are used to estimate the above variance term. It is proved in \cite{HNV20} that  $d_{\rm TV}(  F_R(t) /\sigma_R(t), Z   ) \lesssim R^{-1/2}$. Meanwhile, a multivariate  Malliavin-Stein bound and similar computations lead to the convergence  of the finite-dimensional distributions, which coupled with the tightness property gives a functional CLT for $\{ R^{-1/2}F_R(t): t\in\bR_+\}$.

\medskip

The above general strategy has been adapted to various settings, see
\cite{CKNP19-2,  CKNP20,    HNVZ19, KNP20, KY20, Pu20} for the study
of stochastic heat equations and see \cite{BNZ20, DNZ20, NZ20} for
the study of stochastic wave equations. All these references
consider a Gaussian noise that is white in time.   Nevertheless,
when the Gaussian noise is colored in time, the mild formulation
\eqref{eq1} cannot be interpreted in the Walsh-It\^o sense. In this
situation, only in  the case
  $\sigma(u)=u$   the stochastic heat equation \eqref{SHE} (also known as the \emph{parabolic Anderson model})
  can be properly solved using
  Wiener chaos expansions,
  so that $F_R(t)$, defined in \eqref{FRT}, can be expressed as an infinite sum of multiple Wiener integrals.
  With this well-known fact in mind,  Nualart and Zheng \cite{NZ19BM} considered the parabolic Anderson model
(\emph{i.e.} \eqref{SHE} with $\sigma(u)=u$) on $\bR_+\times\bR^d$
such that $d\geq 1$,  the initial condition is constant and  the
assumptions (i)-(ii)  hold (see page \pageref{page2}).  The main
result of \cite{NZ19BM} is the chaotic CLT that is based on the
fourth moment theorems \cite{FMT05, PT05}.
  When, additionally,  $\gamma$ is a finite measure, the authors of \cite{NZ19BM}  established $\sigma_R(t)\sim R^{d/2}$ and a
functional CLT for the process $R^{-d/2} F_R$; they also considered
the case where $\gamma(x)=|x|^{-\beta}$, for some
$\beta\in(0,2\wedge d)$, is the Riesz kernel, and obtain the
corresponding CLT results. As pointed out in the paper
\cite{NZ19BM}, due to the homogeneity of the underlying Gaussian
noise, the solution $u$ to \eqref{SHE} can be regarded as the
functional of a stationary Gaussian random field so that, with the
Breuer-Major theorem \cite{BM83} in mind, it is natural  to study
Gaussian fluctuations for the problems \eqref{SHE} and \eqref{wave}.
Note  that the constant initial condition  makes the solution
stationary in space and, in fact it is spatially ergodic (see
\cite{CKNP20, NZ20erg}). At last, let us mention the paper
\cite{NSZ20} in which  chaotic CLT was used to study the parabolic
Anderson model driven by a colored Gaussian noise that is rough in
space.  However,  let us point out that the aforementioned methods fail to provide the rate of convergence when
the noise is colored in time.

  \medskip
 In this paper, we bring in  a novel ingredient -- the \emph{second-order Gaussian Poincar\'e inequality}\footnote{{The use of second-order Gaussian Poincar\'e inequality for obtaining CLT on a Gaussian space is one of the central techniques in the Malliavin-Stein approach; for example, in the recent paper \cite{DGLZ20}, Dunlap \emph{et al.} have used this Poincar\'e inequality to investigate the Gaussian fluctuation of the KPZ in dimension three and higher.   We remark here  that we can not directly apply this inequality because of the complicated correlation structure of the underlying Gaussian homogeneous noise, while the underlying Gaussian noise in \cite{DGLZ20} is white in time and smooth in space so that they can directly apply the version from  \cite{NPR09}. In this article, we have established   a quite involved variant  of second-order Poincar\'e inequality, which is tailor-made for our applications.
}  } --  to reach  quantitative CLT results for the hyperbolic Anderson model \eqref{wave}.  Let us first state our main result.

\begin{theorem}\label{MR2} Let $u$ denote the solution to the hyperbolic Anderson model \eqref{wave} and recall the definition of $F_R(t)$ and $ \sigma_R(t)$ from \eqref{FRT}. Let $Z\sim N(0,1)$ be the standard normal random variable. We assume that $\gamma_0$ is not identically zero meaning
\begin{align}
\| \gamma_0\|_{L^1([0,\e])}>0 ~\text{ for any $\e\in(0,1)$.} \label{NDC}
\end{align}  Then the following statements hold true:

\medskip

{\rm  (1)} Suppose that $0<\gamma(\bR^d) <\infty$ if $d=1$ and $\gamma\in L^1(\bR^d) \cap L^\ell(\bR^d)$ for some $\ell >1$  if $d=2$. Then,
\begin{center}
$\sigma_R(t) \sim R^{d/2}$ and
$
d_{\rm TV}\big( F_R(t) / \sigma_R(t) ,  Z \big) \lesssim R^{-d/2}.
$
\end{center}
Moreover, as $R\to\infty$, the process $\big\{ R^{-d/2} F_R(t): t\in\bR_+\big\}$ converges weakly  in the space of continuous functions $C(\bR_+)$ to a centered Gaussian process $\mathcal{G}$ with covariance structure
\begin{align}\label{COV:G}
\bE\big[ \mathcal{G}(t) \mathcal{G}(s) \big] =  \omega_d \sum_{p\geq
1} p! \int_{\bR^d} \big\langle \widetilde{f}_{t,x,p},
\widetilde{f}_{s,0,p} \big\rangle_{\cH^{\otimes p}}dx,
\end{align}
for $t,s\in\bR_+$. Here $\omega_1=2$, $\omega_2=\pi$ and
$\widetilde{f}_{t,x,p}$ are introduced in \eqref{KBD} and
\eqref{eq:3wt}, respectively.  The convergence of the series in
\eqref{COV:G} is part of the conclusion.

\medskip

{\rm (2)} Suppose $d\in\{1,2\}$ and $\gamma(x) = | x|^{-\beta}$ for some $\beta\in(0,2\wedge d)$.   Then,
\begin{center}
$\sigma_R(t) \sim R^{d-\frac{\beta}{2}}$ and $
d_{\rm TV}\big( F_R(t) / \sigma_R(t) , Z  \big) \lesssim R^{-\beta/2}.
$
\end{center}
Moreover, as $R\to\infty$, the process $\big\{ R^{-d+\frac{\beta}{2}} F_R(t): t\in\bR_+\big\}$ converges weakly in the space $C(\bR_+)$ to a centered Gaussian process $\mathcal{G}_{\beta}$ with the covariance structure
\begin{align} \label{COV:Gbeta}
\bE\big[ \mathcal{G}_\beta(t) \mathcal{G}_\beta(s) \big] =
\kappa_{\beta, d} \int_0^t dr\int_0^s dr' \gamma_0(r-r')  (t-r)(s-r'),
\end{align}
for $t,s\in\bR_+$. Here the quantity $\kappa_{\beta, d}$ is introduced in \eqref{KBD}.

\medskip

{\rm  (3)} Suppose $d=2$ and $\gamma(x_1,x_2) = \gamma_1(x_1)
\gamma_2(x_2)$ such that one of the following two conditions holds:
\begin{align}
  \begin{cases}
&\text{\rm ($a'$)}  ~\gamma_i(x_i) =|x_i|^{-\beta_i}~\text{for some $\beta_i\in(0,1)$, $i=1,2$;} \\
&\text{\rm ($b'$)}  ~\gamma_1\in L^{\ell}(\bR) \cap L^1(\bR) ~\text{and  $\gamma_2(x_2)=|x_2|^{-\beta}$ for some $0< \beta< 1 < \ell<\infty$.}
\end{cases}  \label{mix}
\end{align}
Then,
 \begin{align*}
    \begin{cases}
 \sigma_R(t) \sim R^{2 - \frac{1}{2}( \beta_1 + \beta_2)} \quad \text{and} \quad d_{\rm TV}\big( F_R(t) / \sigma_R(t) , Z  \big) \lesssim R^{-(\beta_1+\beta_2)/2} ~ &   \text{in case {\rm $(a')$}}, \\
 \sigma_R(t) \sim R^{(3-\beta)/2 } \quad \text{and} \quad d_{\rm TV}\big( F_R(t) / \sigma_R(t) , Z  \big) \lesssim R^{-(\beta+1)/2} ~  & \text{in case {\rm $(b')$}}.
\end{cases}
 \end{align*}
Moreover, as $R\to\infty$,  in case {\rm $(a')$ },  the process $\big\{ R^{-2+\frac{\beta_1+\beta_2}{2}} F_R(t): t\in\bR_+\big\}$ converges weakly in the space $C(\bR_+)$ to a centered Gaussian process $\mathcal{G}_{\beta_1, \beta_2}$ with the covariance structure
\begin{align} \label{COV:Gbeta12}
\bE\big[ \mathcal{G}_{\beta_1, \beta_2}(t) \mathcal{G}_{\beta_1,
\beta_2 } (s) \big] = K_{\beta_1, \beta_2}   \int_0^t dr\int_0^s dr'
\gamma_0(r-r') (t-r)(s-r'),
\end{align}
for $t,s\in\bR_+$, where
\begin{align}
  K_{\beta_1, \beta_2} :&= \int_{\bR^4}  \mathbf{1}_{\{ x_1^2+x_2^2\leq 1 \}} \mathbf{1}_{\{ y_1^2+y_2^2\leq 1 \}}  |x_1 - y_1|^{-\beta_1} |x_2 - y_2|^{-\beta_2} dx_1dx_2dy_1dy_2;  \label{Kbeta12}
 \end{align}
and
  in case {\rm $(b')$ },  the process $\big\{ R^{\frac{\beta-3}{2}} F_R(t): t\in\bR_+\big\}$ converges weakly in the space $C(\bR_+)$ to a centered Gaussian process $\widehat{\mathcal{G}}_{\beta}$ with the covariance structure
\begin{align} \label{COV:GbetaH}
\bE\big[ \widehat{\mathcal{G}}_{\beta}(t) \widehat{\mathcal{G}}_{\beta} (s) \big] = \gamma_1(\bR) \mathcal{L}_\beta  \int_0^t dr\int_0^s dr' \gamma_0(r-r') (t-r)(s-r')
\end{align}
for $t,s\in\bR_+$, where
  \begin{align}\label{def_L1B}
   \mathcal{L}_\beta : = \int_{\bR^3} dx_1dx_2 dx_3 \mathbf{1}_{\{  x_1^2 + x_2^2\leq 1 \}} \mathbf{1}_{\{  x_1^2 + x_3^2\leq 1 \}} |x_2-x_3|^{-\beta}.
  \end{align}

 For the above functional convergences, we specify that  the space $C(\bR_+)$ is equipped with the topology of uniform convergence on compact sets.

\end{theorem}

\begin{remark} {\rm (i) Note that the case when
$\gamma(x) =\gamma_1(x_1)\gamma_2(x_2)$ with
$\gamma_i\in
L^{\ell_i}(\bR)\cap L^1(\bR)$ for some $\ell_i>1$, $i=1,2$, is covered in part (1). Indeed, suppose that $\ell_1\geq \ell_2$, then  by H\"older's
inequality, $\gamma_1\in L^{\ell_1}(\bR)\cap L^{1}(\bR)$ implies $\gamma_1\in L^{\ell_2}(\bR)\cap L^{1}(\bR)
$ and hence   $\gamma\in
L^{\ell_2}(\bR^2) \cap L^1(\bR^2)$.

 (ii) The rate of convergence can also be described using other common  distances such as the Wasserstein distance and the Kolmogorov distance; see   \cite[Appendix C]{blue}.

(iii) The variance orders and the rates  in parts  (1) and (2) of
Theorem \ref{MR2} are consistent with previous work on stochastic
wave equations, see \cite{BNZ20, DNZ20, NZ20}. The setting in part
(3)  is new.
 As we will see shortly, our strategy is quite different from that in these papers.
}

\end{remark}

Now, let us briefly explain our strategy and begin with the Gaussian
Poincar\'e inequality. For $F\in\mathbb{D}^{1,2}$, the  Gaussian
Poincar\'e inequality (see  \emph{e.g.} \cite{HP95} or \eqref{GPI})
ensures that
\[
\text{Var}(F) \leq \bE\big[ \| DF \|_\cH^2 \big]~\text{with equality if and only if $F$ is  Gaussian},
\]
that is, if $DF$ is small, then the random variable $F$ has necessarily small fluctuations.  In the paper \cite{CHA09},   Chatterjee  pointed out that for $F=f(X_1, \dots , X_d)$ with $X_1, \dots, X_d$ i.i.d. $N(0,1)$ and $f$ twice differentiable,  $F$ is close in total variation distance to a normal distribution with matched mean and variance if {the Hessian matrix $\text{Hess}f(X_1, \dots, X_d)$ is negligible, roughly speaking}. This is known as the second-order Gaussian Poincar\'e inequality.  In what follows, we state the infinite-dimensional version of this inequality due to Nourdin, Peccati and Reinert; see the  paper  \cite{NPR09} as well as the book \cite{blue}\footnote{Note that there is a typo in equation (5.3.2) of \cite{blue}: We have $E[\|DF\|_{\cH}^4]^{1/4}$ instead of $E[\|D^2F\|_{\cH}^4]^{1/4}$.}.

\begin{proposition} Let $F$ be a centered element of $\mathbb{D}^{2,4}$ such that $\bE[ F^2] = \sigma^2 > 0$ and let $Z\sim N(0,\sigma^2)$.  Then,
\begin{align} \label{2ndP}
d_{\rm TV}(F,Z) \leq \frac{3}{\sigma^2}  \left( \bE\Big[  \big\|D^2 F \otimes_1 D^2F \big\|^2_{ \cH^{\otimes 2}} \Big] \right)^{1/4} \left(\bE\big[ \| DF \|_{\cH}^4 \big]\right)^{1/4},
\end{align}
where $D^2 F \otimes_1 D^2F$ denotes the 1-contraction between $D^2F$ and itself $($see \eqref{contrac}$)$.
\end{proposition}
It has been known that this inequality usually gives sub-optimal rate.
 In the recent work \cite{Anna} by Vidotto, she provided an improved version of the above inequality,
where she considered an $L^2$-based Hilbert space $\cH = L^2(A,
\nu)$ with $\nu$ a diffusive measure (nonnegative, $\sigma$-finite and
non-atomic) on some measurable space $A$. Let us state this result
for the convenience of readers.

\begin{theorem}[Theorem 2.1 in \cite{Anna}] Let $F\in\mathbb{D}^{2,4}$ with  mean zero and variance $\sigma^2>0$ and let $Z\sim N(0,\sigma^2)$. Suppose $\cH = L^2(A,\nu)$ with $\nu$ a diffusive measure on some measurable space $A$.  Then,
\[
d_{\rm TV}\big(F, Z\big) \leq \frac{4}{\sigma^2} \left[  \int_{A\times A}  \sqrt{ \bE\big[ \big( D^2F\otimes_1 D^2F\big)^2(x,y)  \big]  \times  \bE\big[  (DF)^2(x)  (DF)^2(y)  \big]     } \nu(dx)\nu(dy)   \right]^{\frac12}.
\]
\end{theorem}

The proof of the above inequality follows from the general Malliavin-Stein bound
\begin{align}\label{MSbdd}
   d_{\rm TV}\big(F, Z\big) \leq \frac{2}{\sigma^2} \bE \left(\big\vert \sigma^2 - \langle DF, - DL^{-1}F \rangle_{\cH} \big\vert  \right)
\end{align}
(see  \cite[equation (5.1.4)]{blue}\footnote{Unlike in  \cite{blue}, we do not assume   $F$   to have a density; in fact,   it suffices to use \cite[Proposition 2.1.1]{GZ18} and \cite[(5.1.1)]{blue} to establish   \cite[equation (5.1.4)]{blue}.  }) and Vidotto's new bound of
\[
\qquad\qquad \bE\big[ ( \text{Cov}(F,G) - \langle DF, - DL^{-1}G \rangle_{\cH} )^2 \big]~\text{for centered $F, G\in\mathbb{D}^{2,4}$}
\]
(see \cite[Proposition 3.2]{Anna}), where $L^{-1}$ is the pseudo-inverse of the Ornstein-Uhlenbeck operator $L$; see Section \ref{sec21} for the definitions.

Recall that our  Hilbert space $\mathcal{H}$    is the  completion of $C^\infty_c(\bR_+\times\bR^d)$ under the inner product
\eqref{defcH}.
The Hilbert space $\mathcal{H}$ contains generalized functions, but fortunately the objects $D^2u(t,x)$,  $Du(t,x)$  are random functions in view of Theorem \ref{MR1}. By adapting    Vidotto's proof to our setting, we have the following   version of second-order Gaussian Poincar\'e inequality. Note we write $f\in | \cH^{\otimes p }|$ to mean $f$ is a real valued function and $\bullet\mapsto | f(\bullet) |$ belongs to $\cH^{\otimes p }$.

\begin{proposition}\label{2nd-tool}
   If $F\in\mathbb{D}^{2,4}$  has mean zero and variance $\sigma^2\in(0,\infty)$ such that with probability 1,  $DF\in| \cH|$ and $D^2F\in|\cH^{\otimes 2}|$, then
\[
d_{\rm TV}\big( F,  Z\big) \leq \frac{4}{\sigma^2} \sqrt{\mathcal{A}},
\]
where $Z\sim N(0,\sigma^2)$ and
\begin{align*}
\mathcal{A}:&= \int_{\bR_+^6\times\bR^{6d}} drdr' dsds' d\theta d\theta' dzdz' dydy' dwdw' \gamma_0(\theta - \theta') \gamma_0(s-s') \gamma_0(r-r') \\
&\quad\times  \gamma(z-z') \gamma(w-w') \gamma(y-y')  \| D_{r,z}D_{\theta,w}F \|_4  \| D_{s,y}D_{\theta',w'}F \|_4  \| D_{r',z'}F\|_4 \| D_{s', y' }F \|_4.
\end{align*}
\end{proposition}
As mentioned before, Proposition \ref{2nd-tool} will follow from the
Malliavin-Stein bound  \eqref{MSbdd} and {Cauchy-Schwarz} inequality,
taking into account that, by the duality relation \eqref{IbP}, we
have that $\bE \left( \langle DF, - DL^{-1}F \rangle_{\cH}  \right)=
\bE[ F^2]=\sigma^2$. Indeed, we can write
\begin{align*}
d_{\rm TV}(F, Z) &  \leq \frac{2}{\sigma^2} \bE \left(\big\vert \sigma^2 - \langle DF, - DL^{-1}F \rangle_{\cH} \big\vert  \right)   \leq \frac{2}{\sigma^2}  \sqrt{   \text{Var}\big(  \langle DF, - DL^{-1}F \rangle_{\cH} \big) }   \\
&\leq  \frac{4}{\sigma^2}  \sqrt{\mathcal{A}} \quad\text{by Proposition \ref{propAV} below.}
\end{align*}

\begin{proposition} \label{propAV}     If $F, G\in\mathbb{D}^{2,4}$ have mean zero such that with probability one, $DF, DG\in | \cH|$ and $D^2F, D^2G\in| \cH^{\otimes 2}|$,
 then
\begin{align}
{\rm Var}\Big( \langle DF, - DL^{-1}G \rangle_\mathcal{H} \Big) =    \bE\big[ ( \text{\rm Cov}(F,G) - \langle DF, - DL^{-1}G \rangle_\mathcal{H} )^2 \big]  \leq 2 A_1 + 2A_2, \label{propA}
\end{align}
where
\begin{align*}
A_1:&= \int_{\bR_+^6\times\bR^{6d}} drdr' dsds' d\theta d\theta' dzdz' dydy' dwdw' \gamma_0(\theta - \theta') \gamma_0(s-s') \gamma_0(r-r') \\
&\quad\times  \gamma(z-z') \gamma(w-w') \gamma(y-y')  \| D_{r,z}D_{\theta,w}F \|_4  \| D_{s,y}D_{\theta',w'}F \|_4  \| D_{r',z'}G\|_4 \| D_{s', y' }G \|_4
\end{align*}
and $A_2$ is defined by switching the positions of $F, G$ in the definition of $A_1$.
 \end{proposition}
 For the sake of completeness, we sketch the  proof of Proposition \ref{propAV} in  Appendix \ref{pfAV}.
Once we have the information on the growth order of $\sigma_R(t)$,
we can apply   Theorem \ref{MR1}  and Proposition \ref{propAV} to
obtain the error bounds in Theorem \ref{MR2}.   The proof of Theorem
\ref{MR2} will be given in Section \ref{sec4}: In Section
\ref{sec41}, we will establish the limiting covariance structure,
which will be used to obtain the quantitative CLTs in Section
\ref{sec42};   Proposition \ref{propAV}, combined with a
multivariate Malliavin-Stein bound (see \emph{e.g.} \cite[Theorem
6.1.2]{blue}),
 also gives us easy access to the   convergence of finite-dimensional distributions (\emph{f.d.d. convergence}) for part (1),
 while in the other parts, the \emph{f.d.d.} convergence follows easily from the dominance of the first chaotic component of $F_R(t)$; finally in Section \ref{sec43},  we establish the functional CLT by showing the required tightness, which will follow by verifying
the well-known criterion of Kolmogorov-Chentsov (see \emph{e.g.} \cite[Corollary 16.9]{OK}).

\subsection{Absolute continuity of the law of  the solution to equation \eqref{wave}}
\label{sec12}

In this part, we fix the following extra hypothesis on the correlation kernels $\gamma_0,\gamma$.
\begin{align*}
{\bf (H2)}
\begin{cases}
\text{ $\gamma_0 =\mathcal{F}\mu_0$ and $\gamma = \mathcal{F} \mu$,  where   $\mu_0, \mu$ are   nonnegative tempered measures} \\
\text{  and have   strictly positive densities with respect to the Lebesgue measure. }
\end{cases}
\end{align*}

The following is  the main result of this section.
\begin{theorem}
\label{MR3}
 Let  $d\in \{1,2\}$ and assume that Hypothesis ${\bf (H2)}$  holds. In addition, assume that Hypothesis ${\bf (H1)}$ holds  if $d=2$. Let $u$ be the solution to
\eqref{wave}. For any $t>0$ and $x \in \bR^d$,    the law of $u(t,x)$ restricted to the set $\bR
\verb2\2 \{0\}$ is absolutely continuous with respect to the
Lebesgue measure on $\bR \verb2\2 \{0\}$.
\end{theorem}

Let us sketch the proof of Theorem \ref{MR3}.
In view of the Bouleau-Hirsch criterion for absolute continuity (see \cite{BH86}),  it suffices to prove that for each $m\ge 1$,
\begin{equation} \label{EQ1}
\|D u(t,x)\|_{\cH}>0 \quad \mbox{a.s. on} \ \Omega_m,
\end{equation}
where $\Omega_m =\{ |u(t,x) | \geq   1/m\}$.  Notice that
\[
\|D u(t,x)\|^2_{\cH} = \int_0^t \int_0^t  \gamma_0(r-s) \langle
D_{r,\bullet}u(t,x) , D_{s,\bullet}u(t,x) \rangle_{0}
  drds,
\]
where $\mathcal{P}_0$
is the completion of $C^\infty_c(\bR^d)$ with respect to the inner
product  $\langle \cdot, \cdot \rangle_0 $ introduced in
\eqref{H0}. The usual approach to show the positivity of
this norm  is to get a lower bound for this integral by integrating
on a small interval $[ t-\delta, t]^2$ and use that, for $r$ close
to $t$, $D_{r,y}u(t,x)$ behaves as $G_{t-r}(x-y) u(s,y)$ (see, e.g.,
\cite{NQ}). However, for $r\not =s$,  the inner product $ \langle
D_{r,\bullet}u(t,x) , D_{s,\bullet}u(t,x) \rangle_{0} $ is not
necessarily non-negative. Our strategy to overcome this difficulty
consists in making use of  Hypothesis ${\bf (H2)}$ in order to show
that
\[
\int_0^t   \| D_{r,\bullet}u(t,x)  \| _{0} ^2 dr >0 ~ ~ \text{implies ~ $\|D u(t,x)\|_{\cH}>0 $  (see Lemma  \ref{pos-norm}).}
\]
 This
allows us to reduce the problem to the non-degeneracy of
$\int_{t-\delta} ^t   \| D_{r,\bullet}u(t,x)  \| _{0} ^2 dr$ for
$\delta$ small enough, which  can be handled by the usual arguments.
At this point, we will make use of the estimates provided in Theorem
\ref{MR1}.

For $d=1$,  Theorem \ref{MR3} was proved in \cite{BQS} under
stronger assumptions on the covariance structure. The result in Theorem \ref{MR3} for
$d=2$ is new. Indeed, the study of the existence (and smoothness) of the density for the stochastic wave equation
has been extensively revisited over the last three decades. We refer the readers to \cite{CN,MS,MMS,QS1,QS2,NQ,SS}. In all these articles,
the authors considered a stochastic wave equation of the form
\[
 \frac{\partial^2 u}{\partial t^2}(t,x)=\Delta u(t,x)+b(u(t,x))+\sigma(u(t,x)) \dot{\mathfrak{X}}(t,x),
\]
on $\bR_+\times \bR^d$, with $d\geq 1$. Here, $\dot{\mathfrak{X}}$ denotes a space-time white noise in the case $d=1$, or a
Gaussian noise that is white in time and has a spatially homogeneous   correlation (slightly more general than that of $W$) in the
case $d\geq 2$. The functions $b,\sigma$ are usually assumed to be globally Lipschitz, and such that the following non-degeneracy condition
is fulfilled: $|\sigma(z)|\geq C>0$, for all $z\in \bR$. The temporal nature of the noise $\dot{\mathfrak{X}}$ made possible to interpret
the solution in the classical Dalang-Walsh sense, making use of all needed martingale techniques. The first attempt to consider
a Gaussian noise that is colored in time was in the paper \cite{BQS}, where the hyperbolic Anderson model with spatial dimension one was considered. As mentioned above, in that paper the existence of density was proved under a slightly stronger assumption than
Hypothesis ${\bf (H2)}$.

\medskip

The rest of this paper is organized as follows.  Section 2 contains preliminary results and the proofs of our main results -- Theorems \ref{MR1}, \ref{MR2} and \ref{MR3} --  are given in Sections \ref{sec3}, \ref{sec4} and \ref{sec5}, respectively.

\medskip

\noindent{\bf Acknowledgement.} The authors would like to thank Wangjun Yuan for  carefully proofreading the manuscript and providing a list of typos.

\section{Preliminary results}\label{sec2}

 This section is devoted to presenting some basic elements of the Malliavin calculus and
 collecting some preliminary results that will be needed in the sequel.

 \subsection{Basic   Malliavin calculus} \label{sec21}

  Recall  that the Hilbert space $\mathcal{H}$ is the completion of $C^\infty_c(\bR_+\times\bR^d)$ under the inner product \eqref{defcH} that can be written as
  \begin{align*}
  \big\langle \psi, \phi \big\rangle_{\cH} = \int_{\bR_+^2}dsdt  \gamma_0(t-s) \big\langle \psi(t, \bullet), \phi(s, \bullet) \big\rangle_{0} \quad \text{for $\psi, \phi\in  C^\infty_c(\bR_+\times\bR^d)$,}
    \end{align*}
where
\begin{equation} \label{H0}
\langle h, g\rangle_0= \int_{\bR^{2d}}dzdz' \gamma(z-z') h(z) g(z').
\end{equation}
As defined in Section \ref{sec12}, we denote by $\mathcal{P}_0$  the
completion of $C^\infty_c(\bR^d)$ with respect to the inner product
$\langle h, g\rangle_0$. Let  $ | \cP_0|$ be the set of  measurable
  functions  $h:\R^d \to \R$ such that
\begin{equation} \label{|cP_0|}
\int_{\bR^{2d}}dzdz' \gamma(z-z') |h|(z) |h|(z') <\infty.
\end{equation}
Then    $|\cP_0| \subset \cP_0$ and for  $h\in | \cP_0|$,    $\| h \|^2_0= \int_{\bR^{2d}}dzdz' \gamma(z-z') h(z) h(z')$.   We define the space $|\cH|$ in a similar way.
For $h,g\in C^\infty_c(\bR^d)$ we can express \eqref{H0}    using  the Fourier transform:
\begin{equation} \label{parseval}
\langle h, g\rangle_0 =\int_{\bR^{d}} \mu(d\xi) \cF h(\xi) \overline{ \cF g(\xi)}.
\end{equation}
  The Parseval-type relation \eqref{parseval} also holds for functions $h,g \in L^1(\R^d) \cap |\mathcal{P}_0|$.

For every integer $p\geq 1$, $\cH^{\otimes p}$ and $\cH^{\odot p}$    denote the $p$th tensor product of $\cH$ and its symmetric subspace, respectively. For example, $f_{t,x,n}$ in \eqref{eq:3} belongs to $\cH^{\otimes n}$ and  $\widetilde{f}_{t,x,n}\in\cH^{\odot n}$;  we also have     $f\otimes  g\in\cH^{\otimes (n+m)}$, provided $f\in\cH^{\otimes m}$ and  $g\in\cH^{\otimes n}$; see \cite[Appendix B]{blue} for more details.

Fix a probability space $(\Omega, \mathcal{B}, \mathbb{P})$, on which we can construct  the isonormal Gaussian process associated to the Gaussian noise $\dot{W}$ in \eqref{wave} that we denote  by  $\{W(\phi): \phi\in\cH\}$. That is,  $\{W(\phi): \phi\in\cH\}$ is a \emph{centered Gaussian family} of real-valued random variables defined on $(\Omega, \mathcal{B}, \mathbb{P})$  such that  $\bE[ W(\psi) W(\phi) ] = \langle \psi, \phi\rangle_{\cH}$ for any $\psi, \phi\in\cH$. We will take $\mathcal{B}$ to be the $\sigma$-algebra $\sigma\{W\}$ generated by the family of random variables $\{ W(h): h\in C^\infty_c(\bR_+\times\bR^d)\}$.

In the sequel, we recall  some basics on Malliavin calculus
from the books \cite{blue, Nualart06}.

Let $C^\infty_\text{poly}(\bR^n)$ denote the space of smooth functions with all their partial derivatives having at most polynomial growth at infinity and let $\mathcal{S}$ denote the set of simple smooth functionals of the form
\begin{center}
$F = f\big(W(h_1), \dots,  W(h_n) \big)$ for $f\in C^\infty_\text{poly}(\bR^n)$ and $h_i\in\cH$, $1\leq i\leq n$.
\end{center}
For such a random variable $F$,  its Malliavin derivative $DF$ is the $\cH$-valued random variable given by
\[
DF = \sum_{i=1}^n \frac{\partial f}{\partial x_i} \big(W(h_1), \dots, W(h_n) \big) h_i.
\]
And similarly its $m$th Malliavin derivative $D^mF$ is the $\cH^{\otimes m}$-valued random variable given by
\begin{align}
D^mF = \sum_{i_1, \dots, i_m=1}^n \frac{\partial^m f }{\partial x_{i_1} \cdots \partial x_{i_m}} \big(W(h_1), \dots , W(h_n) \big) h_{i_1}\otimes \cdots  \otimes h_{i_m}, \label{D^mF}
\end{align}
which is an element in $L^p(\Omega; \cH^{\odot m})$ for any $p\in[1,\infty)$. It is known that the space $\mathcal{S}$ is dense in $L^p(\Omega, \sigma\{W\}, \mathbb{P})$ and
\[
D^m: \mathcal{S} \longrightarrow L^p(\Omega; \cH^{\odot m})
\]
is closable for any $p\in[1,\infty)$; see \emph{e.g.} Lemma 2.3.1 and Proposition 2.3.4 in \cite{blue}. Let $\mathbb{D}^{m,p}$ be the closure of $\mathcal{S}$ under the norm
\[
\big\| F \big\|_{\mathbb{D}^{m,p}} = \Big( \bE\big[ | F |^p \big] +  \bE\big[ \|D F \|^p_{\cH} \big] + \cdots +  \bE\big[ \| D^mF \|^p_{\cH^{\otimes m}} \big] \Big)^{1/p}~\text{and let $\mathbb{D}^{\infty} := \bigcap_{m,p\geq 1}\mathbb{D}^{m,p}. $}
\]
Now, let us introduce the adjoint  of the derivative operator $D^m$. Let $\text{Dom}(\delta^m)$ be the set of   random variables $v\in L^2 ( \Omega; \cH^{\otimes m}  )$ such that there is a constant $C_v>0$ for which
\[
\Big\vert \bE\big[  \langle D^m F, v \rangle_{\cH^{\otimes m}} \big] \Big\vert \leq C_v \| F \|_2  \quad   \text{for all $F\in\mathcal{S}$}.
\]
By \emph{Riesz representation theorem}, there is a unique random variable, denoted by $\delta^m(v)$, such that the following duality relationship  holds:
\begin{align}\label{IbP}
\bE\big[ F \delta^m(v) \big] = \bE\big[  \langle D^m F, v \rangle_{\cH^{\otimes m}} \big].
\end{align}
 Equality \eqref{IbP} holds for all $v\in\text{Dom}(\delta^m)$ and all $F\in\mathbb{D}^{m,2}$. In the simplest case when $F = f( W(h))$ with $h\in\cH$ and $f\in C^1_\text{poly}(\bR)$, we have $\delta(h) = W(h)\sim N(0, \| h\|_\cH^2)$ and equality \eqref{IbP} reduces to
\[
 \bE\big[ f(W(h)) W(h) \big] = \bE\big[ f'(W(h) )  \big] \| h \|_{\cH}^2,
\]
which is exactly part of the Stein's lemma recalled below: For $\sigma\in(0,\infty)$ and an integrable random variable $Z$, Stein's lemma (see \emph{e.g.} \cite[Lemma 3.1.2]{blue}) asserts that
\begin{align} \label{S_lem}
Z\sim N(0, \sigma^2) ~\text{if and only if} ~ \bE[ Z f(Z) ] = \sigma^2 \bE[ f'(Z) ],
\end{align}
for any differentiable function $f:\bR\to\bR$ such that the above expectations are finite. The operator $\delta$ is often called the \emph{Skorokhod integral} since  in the case of the Brownian motion, it coincides with an extension of the
It\^o integral introduced by Skorokhod, see \emph{e.g.} \cite{NP88}. Then we can say $\text{Dom}(\delta^m)$ is the space of Skorokhod integrable random variables with values in $\cH^{\otimes m}$.

 The Wiener-It\^o chaos decomposition theorem asserts that  $L^2(\Omega, \sigma\{W\}, \mathbb{P})$ can be written   as a direct sum of mutually orthogonal subspaces:
\[
L^2(\Omega, \sigma\{W\}, \mathbb{P}) =  \bigoplus_{n\geq 0} \mathbb{C}_n^W,
\]
where $\mathbb{C}_0^W$, identified as $\bR$, is the space of constant random variables and $\mathbb{C}_n^W = \{ \delta^n( h):  h \in\cH^{\otimes n} ~\text{is deterministic}\}$, for $n\geq 1$, is called the $n$th \emph{Wiener chaos} associated to $W$. Note that the first Wiener chaos consists of centered Gaussian random variables.  When $h  \in\cH^{\otimes n}$ is deterministic, we write $I_n(h) =  \delta^n( h )$ and we call it the $n$th multiple integral of $h$ with respect to $W$.  By the symmetry in \eqref{D^mF} and the duality relation \eqref{IbP}, $\delta^n( h ) =  \delta^n( \widetilde{h} ) $ with $ \widetilde{h}$ the canonical symmetrization of $h$, so that we have  $I_n(h) = I_n(\widetilde{h})$ for any $h  \in\cH^{\otimes n}$. The above decomposition can be rephrased as follows. For any $F\in L^2(\Omega, \sigma\{W\}, \mathbb{P})$,
\begin{align}\label{chaos_F}
F = \bE[F] + \sum_{n\geq 1} I_n(f_n),
\end{align}
with $f_n \in    \cH^{\odot n}$ uniquely determined for each $n\geq 1$. Moreover, the (modified) isometry property holds
 \begin{align}\label{miso}
 \bE\big[ I_p(f) I_q(g) \big] = p! \mathbf{1}_{\{ p=q\}}  \big\langle \widetilde{f},  \widetilde{g} \big\rangle_{\cH^{\otimes p}},
  \end{align}
  for any $f\in\cH^{\otimes p}$ and $g\in\cH^{\otimes q}$. We have the following \emph{product formula}:
  For $f\in\cH^{\odot p}$ and $g\in \cH^{\odot q}$,
 \begin{align}\label{prod_f}
 I_p(f) I_q(g) =  \sum_{r=0}^{p\wedge q} r! \binom{p}{r}  \binom{q}{r}  I_{p+q-2r}( f\otimes_r g ),
 \end{align}
where $f\otimes_r g$  is the $r$-contraction between $f$ and $g$, which is an element in $\cH^{\otimes (p+q-2r)}$ defined as follows.
Fix an orthonormal basis  $\{e_i, i\in \mathcal{O}\}$ of $\cH$. Then,  for $1\leq r \leq p\wedge q$,
\begin{align}     \notag
 f\otimes_r g  &:= \sum_{i_1,\dots,i_p, j_1,\dots ,j_q\in\mathcal{O}} \langle f, e_{i_1} \otimes \cdots \otimes e_{i_p} \rangle_{\cH^{\otimes p}}\langle g, e_{j_1} \otimes \cdots \otimes e_{j_q}\rangle_{\cH^{\otimes p}} \mathbf{1}_{\{ i_k=j_k, \forall k=1,\dots,r \}}\\
 &\qquad \times e_{i_{r+1}}  \otimes \cdots \otimes e_{i_p}  \otimes e_{j_{r+1}}  \otimes \cdots \otimes e_{j_q}.   \label{contrac}
 \end{align}
  In the particular case when $f,g$ are real-valued functions, we can write
 \begin{align*}
  (f\otimes_r g)( \pmb{t_{p-r}},   \pmb{x_{p-r}} , \pmb{t'_{q-r}} , \pmb{x'_{q-r}} )&= \int_{\bR_+^{2r} \times\bR^{2rd}} d\pmb{s_r}d\pmb{s'_r} d\pmb{y_r}d\pmb{y'_r} \left( \prod_{j=1}^r \gamma_0(s_j-s'_j)  \gamma(y_j-y'_j)  \right) \\
  &\quad\times f(  \pmb{s_r},\pmb{t_{p-r}},  \pmb{y_r},  \pmb{x_{p-r}}  ) g(  \pmb{s'_r},\pmb{t'_{q-r}},  \pmb{y'_r},  \pmb{x'_{q-r}}  ),
\end{align*}
provided the above integral exists.
For $F\in\mathbb{D}^{m,2}$ with the representation \eqref{chaos_F} and $m\geq 1$, we have
\begin{align}\label{chaos_Dm}
D^m_{\bullet} F = \sum_{n\geq m} \frac{n!}{(n-m)!} I_{n-m}\big( f_n(\bullet,\ast)\big) ~\text{with convergence in $L^2(\Omega; \cH^{\otimes m})$},
\end{align}
where  $I_{n-m}\big( f_n(\bullet,\ast)\big)$ is understood as the $(n-m)$th multiple integral of $ f_n(\bullet,\ast)\in \cH^{\otimes (n-m)}$ for fixed $\bullet$. We can write
\begin{align*}
D^m_{\pmb{s_m}, \pmb{y_m}} F = \sum_{n\geq m} \frac{n!}{(n-m)!} I_{n-m}\big( f_n(\pmb{s_m}, \pmb{y_m};\ast)\big),
\end{align*}
whenever the above series makes sense  and converges in $L^2(\Omega)$. With the decomposition \eqref{chaos_Dm} in mind, we have the following Gaussian Poincar\'e inequality:  For $F\in\mathbb{D}^{1,2}$, it holds that
\begin{align}\label{GPI}
\text{Var}(F) \leq \bE\big[ \| DF \|_\cH^2 \big].
\end{align}
In fact, if $F$ has the representation \eqref{chaos_F}, then
 \[
 \text{Var}(F) = \sum_{n\geq 1} n! \| f_n \|_{\cH^{\otimes n}}^2 \quad \mbox{and} \quad \bE\big[ \| DF \|_\cH^2 \big] = \sum_{n\geq 1} n n!  \| f_n \|_{\cH^{\otimes n}}^2,
\]
which gives us \eqref{GPI} and, moreover, indicates that the equality in \eqref{GPI} holds only when $F\in\mathbb{C}^W_0 \oplus \mathbb{C}^W_1$, that is, only when $F$ is a real Gaussian random variable.

Now let us mention the particular case when the Gaussian noise is white in time, which is used in the reduction step in Section \ref{sec32}.
First, let us   denote   $$\cH_0:=L^2\big(\bR_{+};\cP_0\big)$$ and point out that
the following inequality reduces many calculations to the case of the white noise in time. For any nonnegative function $f\in  \cH_0^{\otimes n}$  that vanishes outside $([0,t] \times \bR^d)^n$,
\begin{equation}
\label{white-ineq}
\|f\|_{\cH^{\otimes n}}^2 \leq \Gamma_t^n \|f\|_{\cH_0^{\otimes n}}^2,
\end{equation}
 where\footnote{For the sake of completeness, we sketch a proof of \eqref{white-ineq} here: Given such a function $f  \in \cH_0^{\otimes n}$,
 \begin{align*}
\|f\|_{\cH^{\otimes n}}^2 &= \int_{[0,t]^{2n}} d\pmb{s_n} d\pmb{t_n} \big\langle f(\pmb{s_n}, \bullet), f(\pmb{t_n}, \bullet) \big\rangle_{\cP_0^{}\otimes n} \prod_{j=1}^n \gamma_0(s_j-t_j)  \\
& \leq  \int_{[0,t]^{2n}} d\pmb{s_n} d\pmb{t_n} \frac{1}{2} \Big(  \big\| f(\pmb{s_n}, \bullet)\big\| _{\cP_0^{\otimes n} }^2  + \big\|   f(\pmb{t_n}, \bullet) \big\|_{\cP_0^{\otimes n} } ^2 \Big)\prod_{j=1}^n \gamma_0(s_j-t_j)  \leq \Gamma_t^n \|f\|_{\cH_0^{\otimes n}}^2.
 \end{align*}
  }
\[
\Gamma_t=2\int_{0}^t \gamma_0(s)ds \quad {\rm and} \quad \|f\|_{\cH_0^{\otimes n}}^2=\int_{[0,t]^n}\|f(t_1,\cdot,\ldots,t_n,\cdot)\|_{\cP_0^{\otimes n}}^2 dt_1 \cdots dt_n;
\]
 whenever no ambiguity arises, we   write $\| f\|_0:=\| f\|_{\cP_0^{\otimes n}}$ so that
$
\|f\|_{\cH_0^{\otimes n}}^2=\int_{[0,t]^n}\|f( \pmb{t_n} ,\bullet)\|_{0}^2 d\pmb{t_n}.
$

 Let $\dot{\mathfrak{X}}$ denote the Gaussian noise that is white in time and has the same spatial correlation as $W$. More precisely, $\{\mathfrak{X}(f): f\in\cH_0\}$ is   a centered Gaussian family with covariance $$\bE[ \mathfrak{X}(f) \mathfrak{X}(g) ]  =\langle f, g \rangle_{\cH_0}, \quad \mbox{for any $f,g\in\cH_0$}.$$
 Denote by $I^{\mathfrak{X}}_p$ the $p$-th multiple stochastic  integral with respect to $\mathfrak{X}$. The product formula \eqref{prod_f} still holds with  $W$ replaced  by the noise $\mathfrak{X}$. Moreover,
 if $f\in\cH^{\otimes p}$ and $g\in \cH^{\otimes q}$ have disjoint temporal supports\footnote{This means $f = 0$ outside $(J \times\bR^{d})^p$ and  $g= 0$ outside $(J^c\times\bR^{d})^q$ for some set $J\subset\R_+$.
 We will apply this formula to functions $f=f_{t,x,j}^{(j)}(r,z;\bullet)$ and $g=f_{r,z,n-j}$ given in Section \ref{sec31}, in which case $J=(r,t)$.},
 then we have $f\otimes_r g =0$ for $r=1,\dots, p\wedge q$  and the product formula \eqref{prod_f} reduces to
 \begin{align}\label{prod}
  I^{\mathfrak{X}}_p(f) I^{\mathfrak{X}}_q(g) = I^{\mathfrak{X}}_{p+q}(f\otimes g).
 \end{align}
In this case, the random variables  $I^{\mathfrak{X}}_p(f)$ and $I^{\mathfrak{X}}_q(g)$ are independent by the \"Ust\"unel-Zakai-Kallenberg criterion (see Exercise 5.4.8 of \cite{blue}) and note that we do not need to assume $f,g$ to be symmetric in \eqref{prod}.

Now let us introduce the Ornstein-Uhlenbeck operator $L$ that can be defined as follows.
We say that $F$ belongs to the $\text{Dom}(L)$  if $F\in\mathbb{D}^{1,2}$ and $DF\in\text{Dom}(\delta)$; in this case, we let $LF = -\delta DF$.  For $F\in L^2(\Omega)$ of the form \eqref{chaos_F},  $F\in\text{Dom}(L)$ if and only if
$
\sum_{n\geq 1} n^2 n! \| f_n \|_{\cH^{\otimes n}}^2 <\infty.
$
  In this case, we have $LF = \sum_{n\geq 1} -n I_n(f_n)$.  Using the chaos expansion, we can also define
  the Ornstein-Uhlenbeck semigroup $\{P_t = e^{tL}, t\in\bR_+\}$ and the pseudo-inverse $L^{-1}$ of the Ornstein-Uhlenbeck operator $L$ as follows. For $F\in L^2(\Omega)$ having the chaos expansion \eqref{chaos_F},
  \[
  P_t F := \sum_{n\geq 0} e^{-nt} I_n(f_n)  \quad {\rm and}\quad L^{-1} F = \sum_{n\geq 1} -\frac{1}{n} I_n(f_n).
  \]
 Observe that for any centered random variable $F\in L^2(\Omega, \sigma\{W\}, \mathbb{P})$, $LL^{-1}F = F$ and for any $G\in\text{Dom}(L)$, $L^{-1} LG = G - \bE[G].$  The above expression and the modified isometry property \eqref{miso} give us the contraction property of $P_t$ on $L^2(\Omega)$, that is, for $F\in L^2(\Omega, \sigma\{W\}, \mathbb{P})$, $\| P_t F\| _2 \leq \| F\|_2$. Moreover, $P_t$ is a contraction operator on $L^q(\Omega)$ for any $q\in[1,\infty)$; see \cite[Proposition 2.8.6]{blue}.

Finally, let us recall Nelson's \emph{hypercontractivity property} of the Ornstein-Uhlenbeck semigroup: For $F\in L^q(\Omega, \sigma\{W\}, \mathbb{P})$ with $q\in(1,\infty)$, it holds for each $t\geq 0$ that  $\| P_t F \| _{q_t} \leq \| F \| _q$ with $q_t = 1 + (q-1)e^{2t}$. In this paper, we need one of its consequences -- a moment inequality comparing $L^q(\Omega)$-norms on a fixed chaos:
 \begin{align}\label{hyper}
 \text{If $F\in\mathbb{C}^W_n$ and $p\in[2,\infty)$, then $  \| F\|_p \leq (p-1)^{n/2}\| F\|_2$; }
 \end{align}
  see \emph{e.g.} \cite[Corollary 2.8.14]{blue}.

\subsection{Inequalities} \label{sec22}

Let us first present a few inequalities, which will be used in Section \ref{sec3}.

\begin{lemma}\label{lem:embed}
Fix an integer $d\geq 1$. Suppose that either one of the following conditions hold:
\begin{center}
{\rm (a)} $\gamma \in L^{\ell}(\bR^d)$ for some $\ell\in(1,\infty)$ \quad  {\rm (b)} $\gamma(x)=|x|^{-\beta}$ for some $\beta \in (0,d)$. \end{center}
Define
\begin{align*}
q= \begin{cases}
\ell/(2\ell-1) & \text{in case  \rm (a)} \\
d/(2d-\beta)& \text{in case \rm (b).}
\end{cases}
\end{align*}
Then, for any $f,g \in L^{2q}(\R^{d})$,
$$\int_{\bR^d} \int_{\bR^d}f(x)g(y)\gamma(x-y)dxdy \leq C_\gamma \|f\|_{L^{2q}(\bR^d)}
\|g\|_{L^{2q}(\bR^d)},$$
where $C_\gamma=\|\gamma\|_{L^{\ell}(\bR^d)}$ in case {\rm(a)}, and $C_\gamma=C_{d,\beta}$ is the constant $($depending on $d,\beta)$ that appears in the Hardy-Littlewood-Sobolev inequality \eqref{HLS} below, in case {\rm(b)}.

\end{lemma}

\begin{proof}
In the case $d=2$, this result was essentially proved on page 15 of \cite{NZ20} in case (a), and on page 6 of \cite{BNZ20} in case (b). We reproduce the arguments here for the sake of completeness.

In case (a), we apply H\"older's inequality and \emph{Young's convolution inequality}:
\[
\int_{\bR^d}f(x)(g*\gamma)(x)dx \leq \|f\|_{L^{\frac{2\ell}{2\ell-1}}(\bR^d)}
\|g*\gamma\|_{L^{2\ell}(\bR^d)} \leq \|f\|_{L^{\frac{2\ell}{2\ell-1}}(\bR^d)} \|g\|_{L^{\frac{2\ell}{2\ell-1}}(\bR^d)} \|\gamma\|_{L^{\ell}(\bR^d)}.
\]
 In case (b), we apply H\"older's inequality and \emph{Hardy-Littlewood-Sobolev inequality}:
\begin{equation} \label{HLS}
\int_{\bR^d}f(x)(g*\gamma)(x)dx \leq \|f\|_{L^{\frac{2d}{2d-\beta}}(\bR^d)}
\|g*\gamma\|_{L^{2d/\beta}(\bR^d)} \leq C_{d,\beta}\|f\|_{L^{\frac{2d}{2d-\beta}}(\bR^d)} \|g\|_{L^{\frac{2d}{2d-\beta}}(\bR^d)}.
\end{equation}
This concludes the proof. \qedhere

\end{proof}

To deal with case (\texttt{c}) in  ${\bf (H1)}$, we need the following modification of Lemma \ref{lem:embed}.

\begin{lemma}\label{lem_mix}
Suppose  that  $\gamma(x_1,\ldots,x_d)=\prod_{i=1}^{d}\gamma_i(x_i)$, where for each $i\in\{1,\ldots,d\}$,
\[
\mbox{\rm (M1) $\gamma_i \in L^{\ell_i}(\bR)$ for some $\ell_i\in(1,\infty)$ \quad  or \quad (M2) $\gamma_i(x)=|x|^{-\beta_i}$ for some $\beta_i \in (0,1)$}.
\]
Let $q_i=\ell_i/(2\ell_i-1)$ in case {\rm (M1)} and $q_{i}=1/(2-\beta_i)$ in case {\rm(M2)}. Let $q=\max\{q_i: i=1,\dots,d\}$.

If $f, g \in L^{2q}(\bR^d)$ satisfy $f(x)=g(x)=0$ for $x \not \in \prod_{i=1}^d[a_i,b_i]$ for some real numbers $a_i<b_i$\footnote{We can apply this lemma to the function $y\in\bR^2\mapsto G_{t-s}(x-y)$ whose support is contained in $\{y \in \bR^2; |x-y|<t-s\}$,
 so we can choose $\Lambda = 2t-2s$.},
 then
 \begin{align}\label{ineq_mix}
  \int_{\bR^d} \int_{\bR^d}f(x)g(y)\gamma(x-y)dxdy \leq  \Lambda^{\nu} C_\gamma \|f\|_{L^{2q}(\bR^d)} \|g\|_{L^{2q}(\bR^d)},
  \end{align}
with $\Lambda=\max\{b_i-a_i;i=1,\ldots,d\}$,    $C_\gamma =  \prod_{i=1}^{d}C_{\gamma_i}$ and $\nu= \sum_{i=1}^{d} (q_i^{-1} - q^{-1})$. In particular, when $q_i=q$ for all $i\in\{1,\ldots,d\}$, we have
\[
\int_{\bR^d} \int_{\bR^d}f(x)g(y)\gamma(x-y)dxdy \leq C_\gamma \|f\|_{L^{2q}(\bR^d)} \|g\|_{L^{2q}(\bR^d)}.
\]
The constants $C_{\gamma_i}$ are defined as in Lemma  \ref{lem:embed}.
\end{lemma}

\begin{proof}
By Lemma \ref{lem:embed},  inequality \eqref{ineq_mix} holds for $d=1$ with $\nu=0$. Now let us consider $d\geq 2$ and prove  inequality  \eqref{ineq_mix} by induction. Suppose   \eqref{ineq_mix} holds for $d\leq k-1$ $(k\geq 2)$. We use the notation $x=(x_1,\ldots,x_k)=:\pmb{x_k}.$
 Without loss of any generality we assume $q_1\geq q_2\geq \cdots \geq q_k$,  so that $q=q_1$. Applying the initial step $(d=1)$ yields
\begin{align}
\nonumber
&\int_{\bR^{2k}} d\pmb{x_k} d\pmb{y_k}  f(\pmb{x_k} ) g(\pmb{y_k}) \prod_{i=1}^k \gamma_i(x_i-y_i)  \\
&\quad \leq  C_{\gamma_k}  \int_{\bR^{2(k-1)}} d\pmb{x_{k-1}}  d\pmb{y_{k-1}}  \big\| f(\pmb{x_{k-1}}, \bullet ) \big\|_{L^{2q_k}(\bR)}  \big\| g(\pmb{y_{k-1}}, \bullet ) \big\|_{L^{2q_k}(\bR)}   \prod_{i=1}^{k-1} \gamma_i(x_i-y_i).  \label{induction1}
\end{align}
By the induction hypothesis, we can bound the right-hand side of \eqref{induction1} by
\begin{align*}
\left(\prod_{i=1}^k C_{\gamma_i} \right) \Lambda^{\nu^\ast} \left(     \int_{\bR^{k-1}}  \big\| f(\pmb{x_{k-1}}, \bullet ) \big\|_{L^{2q_k}(\bR)}^{2q} d\pmb{x_{k-1}}   \right)^{\frac{1}{2q}}\left(     \int_{\bR^{k-1}}  \big\| g(\pmb{y_{k-1}}, \bullet ) \big\|_{L^{2q_k}(\bR)}^{2q} d\pmb{y_{k-1}}   \right)^{\frac{1}{2q}},
\end{align*}
with $\nu^\ast = \sum_{i=1}^{k-1}( q_i^{-1} - q^{-1})$. By H\"older's inequality,
\begin{align*}
 \left(     \int_{\bR^{k-1}}  \big\| f(\pmb{x_{k-1}}, \bullet ) \big\|_{L^{2q_k}(\bR)}^{2q} d\pmb{x_{k-1}}   \right)^{\frac{1}{2q}}  & =  \left(     \int_{\bR^{k-1}} \left[ \int_{a_k}^{b_k}   \big\vert f(\pmb{x_{k-1}}, x_k ) \big\vert^{2q_k} dx_k \right]^{\frac{2q}{2q_k}}   d\pmb{x_{k-1}}   \right)^{\frac{1}{2q}} \\
&\leq  \Lambda^{\frac{1}{2q_k} - \frac{1}{2q}}  \left(     \int_{\bR^{k-1}}   \int_{a_k}^{b_k}   \big\vert f(\pmb{x_{k-1}}, x_k ) \big\vert^{2q} dx_k    d\pmb{x_{k-1}}   \right)^{\frac{1}{2q}}.
 \end{align*}
 A similar inequality holds for $g$. Since $\nu^\ast  +  ( q_k^{-1} - q^{-1}) = \sum_{i=1}^{k} ( q_i^{-1} - q^{-1})$, inequality \eqref{ineq_mix} holds for $d=k$. 
\end{proof}

We will need the following generalization of
Lemma \ref{lem:embed} and Lemma \ref{lem_mix}.
\begin{lemma}\label{lem:1}

{\rm (1)} Under the conditions of Lemma \ref{lem:embed}, for any $f,g \in L^{2q}(\R^{md})$
  \begin{align}\label{ineq_lem1}
  \int_{\R^{2md}}  f(\pmb{x_m}) g(\pmb{y_m})
  \prod _{j=1}^m \gamma(x_j-y_j) d\pmb{x_m} d\pmb{y_m}
  \le C_\gamma^m \| f \|_{L^{2q}(\R^{md})}  \| g \|_{L^{2q}(\R^{md})},
  \end{align}
  where $C_\gamma$ is the same constant as in Lemma \ref{lem:embed}. Here $\pmb{x_m} = (x_1, \dots,  x_m)$ with $x_i\in\bR^d$.

  \medskip

  {\rm (2)}  Let $\gamma, C_{\gamma}$ and $q$   be given as in Lemma \ref{lem_mix}. If $f,g \in L^{2q}(\R^{md})$ satisfy $f(\pmb{x_{md}}) = g(\pmb{x_{md}})=0$ for $\pmb{x_{md}} \notin \prod_{i=1}^{md} [a_i, b_i]$  for some real numbers $a_i<b_i$, then   inequality \eqref{ineq_lem1} holds with  $C_\gamma$ replaced by $\Lambda^\nu C_\gamma$, where
  $\Lambda=\max\{ b_i -a_i : i=1,\dots, md\}$ and $\nu = \sum_{i=1}^d ( q_i^{-1} - q^{-1})$. Here  $\pmb{x_{md}} = (x_1, \dots, x_{md})$ with $x_i\in\bR$.
  \end{lemma}

\begin{proof}
  The proof will be done by induction on $m$ simultaneously for both cases (1) and (2).
  Let $C=C_{\gamma}$ in case (1) and $C=\Lambda^{\nu}C_{\gamma}$ in case (2). The results are true for $m=1$ by Lemma \ref{lem:embed} and Lemma \ref{lem_mix}.
   Assume that the results hold for $m-1$. Applying the inequality for $m=1$ yields
  \begin{align*}
&\quad   \int_{\R^{2dm}} f(\pmb{x_m}) g(\pmb{y_m})
\prod _{j=1}^m \gamma(x_j-y_j) d\pmb{x_m} d\pmb{y_m}\\
&   \leq C  \int_{\R^{2d(m-1)}} \|  f(\pmb{x_{m-1}},\bullet) \|
_{L^{2q} (\R^d)} \|g(\pmb{y_{m-1}},\bullet) \|_{L^{2q}
(\R^d)}   \prod _{j=1}^{m-1} \gamma(x_j-y_j) d\pmb{
x_{m-1}} d\pmb{y_{m-1}}.
  \end{align*}
  By the induction hypothesis, the latter term can be bounded by
  \begin{align*}
 & C^m      \left( \int_{\R^{d(m-1)}}   \|  f(\pmb{x_{m-1}},\bullet)\|^{2q}_{L^{2q} (\R^d)}
  d\pmb{x_{m-1}} \right)^{\frac 1{2q}}
   \left( \int_{\R^{d(m-1)}}   \|  g(\pmb{x_{m-1}},\bullet) \|^{2q} _{L^{2q} (\R^d)}
   d\pmb{x_{m-1}}\right)^{\frac 1{2q}},
  \end{align*}
which completes the proof.
  \end{proof}

Let us  return to the three cases of Hypothesis  ${\bf (H1)}$.  Lemma \ref{lem:embed} indicates  that $L^{2q}(\bR^{2})$ is continuously embedded into $\cP_{0}$,
with $q\in(1/2, 1)$  given by
\begin{align}
\label{def-q}
q=
\begin{cases}
\ell/(2\ell-1) & \mbox{in case  \rm (\texttt{a})}, \\
2/(4-\beta)& \mbox{in case \rm (\texttt{b})}.
\end{cases}
\end{align}
Recall that $\cP_{0}$ has been defined at the beginning of Section \ref{sec21}.
Moreover, for any $f, g\in L^{2q}(\bR^2)$,
\begin{equation}
\label{ineq-norms}
 \int_{\bR^4} \big\vert  f(x)g(x) \big\vert   \gamma(x-y)dxdy \leq D_{\gamma}
\|f\|_{L^{2q}(\bR^2)}\|g\|_{L^{2q}(\bR^2)},
\end{equation}
where
\begin{align}
\label{def-D}
D_{\gamma}=
\begin{cases}
 \|\gamma\|_{L^{\ell}(\bR^2)} & \mbox{in case \rm (\texttt{a})}, \\
C_{2,\beta}& \mbox{in case \rm (\texttt{b})}.
\end{cases}
\end{align}

 For case (\texttt{c}) of  Hypothesis  ${\bf (H1)}$, we consider three sub-cases:
\begin{align*}
  \begin{cases}
&{\rm(i)} ~  \gamma_i\in L^{\ell_i}(\bR)  ~\text{for some $\ell_i>1$, $i=1,2$;} \\
&{\rm(ii)}   ~\gamma_i(x_i) =|x_i|^{-\beta_i}~\text{for some $\beta_i\in(0,1)$, $i=1,2$;} \\
&{\rm(iii)}   ~\gamma_1\in L^{\ell}(\bR)  ~\text{for some $\ell\in(1,\infty)$ and $\gamma_2(x_2)=|x_2|^{-\beta}$ for some $\beta\in(0,1)$.}
\end{cases}
\end{align*}
Lemma \ref{lem_mix} implies that, for any $f, g\in L^{2q}(\bR^2)$ with
\begin{align} \label{def-qq}
q=
\begin{cases}
\max\{\ell_i/(2\ell_i-1) : i=1,2\}& \mbox{in case  \rm (i)} \\
\max\{1/(2-\beta_i): i=1,2\}& \mbox{in case \rm (ii)}\\
\max\{  \ell/(2\ell-1), 1/(2-\beta)\}& \mbox{in case \rm (iii)}
\end{cases},
\end{align}
 such that $f, g$ vanish outside a box with side lengths bounded by $\Lambda$, then  inequality \eqref{ineq-norms} still holds with
 \begin{align} \label{def-DD}
D_\gamma
=\begin{cases}
\| \gamma_1\|_{L^{\ell_1}(\bR)} \| \gamma_2\|_{L^{\ell_2}(\bR)}  \Lambda^{  |\frac 1 {\ell_1} - \frac 1 {\ell_2}|}& \mbox{in case  \rm (i)} \\
 C_{1,\beta_1}C_{1,\beta_2} \Lambda^{ | \beta_1-\beta_2|}   & \mbox{in case \rm (ii)}\\
 C_{1,\beta} \| \gamma_1\|_{L^\ell(\bR)}  \Lambda^{ | \frac 1\ell -\beta| }  & \mbox{in case \rm (iii)}
\end{cases},
\end{align}
where the constants $C_{1,\beta_i}$ are given as in  Lemma \ref{lem:embed}.

 From Lemma \ref{lem:1}, we deduce that in cases (\texttt{a}) and (\texttt{b}),
\begin{equation} \label{q-ineq}
\|f\|_{\cH_0^{\otimes n}}^2\leq D_{\gamma}^n \int_{[0,t]^n} \|f( \pmb{t_n},\bullet)\|_{L^{2q}(\bR^{2n})}^2 d\pmb{t_n},
\end{equation}
for any measurable  function $f:( \R_+ \times \bR^2)^n \to \bR$ such that   $f$ vanishes outside $([0,t] \times \R^2)^n$; in case (\texttt{c}), inequality \eqref{q-ineq} holds true for any measurable  function $f: (\bR_+ \times\bR^{2})^n \to \bR$ such that
\[
f(t_1,x_1, \dots, t_n, x_n) = f(\pmb{t_n}, \pmb{x_n}) = 0~\text{for $\pmb{t_n}\notin [0,t]^n$ and $\pmb{x_n}\notin \prod_{i=1}^{2n} [a_i, b_i]$}
\]
with $\Lambda :=\max\{ b_i-a_i : i=1,\dots, 2n\}<\infty$.

\medskip
Let us present a few facts  on the fundamental solution $G$.
  When $d=2$,
 \begin{equation}
\|G_t\|_{L^p(\bR^2)}=\left(\frac{(2\pi)^{1-p}}{2-p} \right)^{1/p}t^{\frac{2}{p}-1} \quad \mbox{for all}~  p \in (0,2),  \label{p-norm-G}
\end{equation}
\begin{equation}
G_t^{p}(x) \leq (2\pi t)^{q-p}  G_t^{q}(x) \quad \mbox{for all}   ~ p<q, \label{ineq-Gp}
\end{equation}
and
\begin{equation}
\mathbf{1}_{\{|x|<t\}} \leq 2\pi t G_t(x). \label{indicator}
 \end{equation}
 We will use also the following   estimate.

\begin{lemma}[Lemma 4.3 of \cite{BNZ20}]
\label{lem33BNZ}
For any $q \in (1/2,1)$ and $d=2$,
$$\int_r^t (G_{t-s}^{2q} * G_{s-r}^{2q})^{1/q}(z)ds \leq A_{q} (t-r)^{\frac{1}{q}-1} G_{t-r}^{2-\frac{1}{q}}(z),$$
where $A_{q}>0$ is a constant depending on $q$.
\end{lemma}

Finally, we record the expression of  the Fourier transform of $G_t$ for $d\in\{1,2\}$:
\begin{align}\label{FG}
 \mathcal{F} G_t(\xi) = \int_{\bR^d} e^{-i \xi \cdot x} G_t(x) dx =  \frac{\sin ( t | \xi | )}{|\xi |}=: \widehat{G}_t(\xi).
\end{align}
Note that (see e.g. (3.4) of \cite{BS17})
\begin{align}\label{bddFG}
\big\vert  \widehat{G} _t(\xi)\big\vert^2
\leq 2(t^2\vee 1) \frac{1}{1+ |\xi|^2}.
\end{align}

In Section \ref{sec4}, we need the  following two results.

\begin{lemma}\label{lem_ab}
For  $d\in\{ 1,2\}$,  let  $\gamma_0$ satisfy  the assumption {\rm (i)} on page \pageref{page2}  and  let $\mu_p$ be a  symmetric measure on $(\bR^{d})^p$, for some
integer $p\geq 1$. Then, with  $0< s\leq t$ and  $\Delta_p(t)= \{ \pmb{s_p}\in \bR_+^p: t =s_0 > s_1 > \cdots> s_p > 0  \}$,
\begin{align*}
&\quad\sum_{\sigma\in\mathfrak{S}_p}   \int_{\Delta_p(t)} d\pmb{s_p} \int_{[0,s]^p} d\pmb{\tilde{s}_p} \mathbf{1}_{\{ s > \tilde{s}_{\sigma(1)  > \cdots > \tilde{s}_{\sigma(p)} >0  }  \}} \left( \prod_{j=1}^p \gamma_0(s_j -   \tilde{s}_{j}  ) \right) \int_{\bR^{pd}}  \mu_p(d\pmb{\xi_p})  \notag    \\
& \qquad\qquad   \times      g(s_1, \xi_1, \dots, s_p, \xi_p)      g(\tilde{s}_{\sigma(1)}, \xi_{\sigma(1)}, \dots, \tilde{s}_{\sigma(p)}, \xi_{\sigma(p)})         \\
&\leq  \Gamma_t^p   \int_{\Delta_p(t)} d\pmb{s_p} \int_{\bR^{pd}}  \mu_p(d\pmb{\xi_p})    g(s_1, \xi_1, \dots, s_p, \xi_p)^2, \quad \text{with}~ \Gamma_t : = \int_{-t}^t \gamma_0(a)da,
\end{align*}
for any   measurable function $g: (\bR_+\times\bR^d)^p\to \bR_+$ for which the above integral is finite.
\end{lemma}

\begin{proof}
After applying $|ab|\leq \frac{a^2+b^2}{2}$ and  using the symmetry of $ \mu_p$, we have  that the  left-hand side quantity  is  bounded by
\begin{align}
&   \frac{1}{2} \sum_{\sigma\in\mathfrak{S}_p}   \int_{\Delta_p(t)} d\pmb{s_p} \int_{[0,s]^p} d\pmb{\tilde{s}_p} \mathbf{1}_{\{ s > \tilde{s}_{\sigma(1)  > \cdots > \tilde{s}_{\sigma(p)} >0  }  \}}  h(\pmb{s_p}) \prod_{j=1}^p \gamma_0(s_j -   \tilde{s}_{j}  ) \label{termA1} \\
&\quad + \frac{1}{2} \sum_{\sigma\in\mathfrak{S}_p}   \int_{\Delta_p(t)} d\pmb{s_p} \int_{[0,s]^p} d\pmb{\tilde{s}_p} \mathbf{1}_{\{ s > \tilde{s}_{\sigma(1)  > \cdots > \tilde{s}_{\sigma(p)} >0  }  \}}  h\big(  \tilde{s}_{\sigma(1)}, ..., \tilde{s}_{\sigma(p)} \big)   \prod_{j=1}^p \gamma_0(s_j -   \tilde{s}_{j}  )  \label{termA2}
\end{align}with
\begin{align*}
    h(s_1, \dots, s_p):= \begin{cases}
{\displaystyle \int_{\bR^{pd}}  \mu_p(d\pmb{\xi_p} )  g(s_1, \xi_1, \dots, s_p, \xi_p)^2, } \quad &\text{for $\pmb{s_p}\in\Delta_p(t)$} \\
0, & \text{otherwise.}
\end{cases}
\end{align*}
Putting $\mathcal{I}_s(s_1, \dots, s_p) := \mathbf{1}_{\{ s> s_1> \cdots> s_p>0 \}} $ and  letting $\widetilde{\mathcal{I}}_s(s_1, \dots, s_p)$  be its canonical symmetrization (so that  $\big\vert \widetilde{\mathcal{I}}_s\big\vert \leq (p!)^{-1}$), we can
rewrite the term in \eqref{termA1} as
\begin{align*}
\frac{p!}{2} \int_{\Delta_p(t)}\int_{[0,s]^p} d\pmb{s_p} d\pmb{\tilde{s}_p}    h(\pmb{s_p})  \widetilde{\mathcal{I}}_s(\pmb{\tilde{s}_p} )   \prod_{j=1}^p  \gamma_0(s_j-\tilde{s}_j) &\leq \frac{1}{2} \int_{\Delta_p(t)}\int_{[0,s]^p} d\pmb{s_p} d\pmb{\tilde{s}_p}    h(\pmb{s_p})     \prod_{j=1}^p  \gamma_0(s_j-\tilde{s}_j) \\
&\leq \frac{1}{2}\Gamma_t^p  \int_{\Delta_p(t)}   d\pmb{s_p}   h(\pmb{s_p}),
\end{align*}
using also the bound
$
\sup\{ \int_0^s \gamma_0(r-r') dr' :  r\in[0,t] \}\leq \Gamma_t
$.
For  the other term \eqref{termA2}, we argue in the same way: With $(\mathcal{I}_s \cdot h)(s_1, ... , s_p) =\mathcal{I}_s(s_1, \dots, s_p) h(s_1, ... , s_p)  $, we rewrite the term \eqref{termA2} as
\begin{align*}
&\quad \frac{p!}{2}   \int_{[0,t]^p} d\pmb{s_p} \int_{[0,s]^p} d\pmb{\tilde{s}_p}  \mathcal{I}_t(\pmb{s_p}) \times \widetilde{(\mathcal{I}_s \cdot h)}(\pmb{\widetilde{s}_p})   \prod_{j=1}^p \gamma_0(s_j -   \tilde{s}_{j}  ) = \frac{p!}{2}  \big\langle  \mathcal{I}_t,  \widetilde{\mathcal{I}_s \cdot h} \big\rangle_{\mathcal{H}^{\otimes p}}   \\
 &=\frac{p!}{2}  \big\langle  \widetilde{ \mathcal{I}_t},   \mathcal{I}_s \cdot h \big\rangle_{\mathcal{H}^{\otimes p}}  \leq \frac{1}{2} \int_{[0,t]^p} d\pmb{t_p} \int_{\Delta_p(s)} h( \pmb{\widetilde{s}_p})  \prod_{j=1}^p\gamma_0(s_j-\widetilde{s}_j) \leq  \frac{1}{2} \Gamma_t^p \int_{\Delta_p(s)}   d\pmb{s_p}   h(\pmb{s_p}),
 \end{align*}
 since $h\geq 0$ and $\big\vert \widetilde{\mathcal{I}}_t\big\vert \leq (p!)^{-1}$. This concludes the proof.
 \qedhere

 \end{proof}

\begin{lemma}\label{lem_4Qp}
For  $d\in\{ 1,2\}$  let  $\gamma, \mu$ satisfy the  assumption {\rm (ii)} on page \pageref{page2}. Then, for any \emph{nonnegative} function $h\in\cP_0\cap L^1(\R^d)$,
\[
\sup_{z\in\bR^d} \int_{\bR^d}\mu(d\xi) | \cF  h(\xi +z) |^2 \leq \int_{\bR^d}\mu(d\xi)  | \cF  h(\xi ) |^2.
\]
As a consequence,  for any integer $p\geq 1$ and $w_1, \dots , w_p\in [0,t]$,
\begin{equation} \label{ineq1}
\sup_{\pmb{w_p}\in[0,t]^p }\sup_{\pmb{z_p}\in\bR^{dp}}  \int_{\bR^{dp}} \mu(d\pmb{\xi_p}) \prod_{j=1}^{p} \big\vert \widehat{ G}_{w_j}(\xi_j + z_j ) \big\vert^2 \leq   \left( 2(t^2\vee 1)  \int_{\bR^d}   \frac{\mu(d\xi)}{1+ |\xi|^2} \right)^p.
\end{equation}

\end{lemma}

\begin{proof} Since $h \geq 0$, using the fact that
$\cF h(\xi+z) = \cF (e^{-iz \cdot} h)(\xi)$ together with    $|e^{-iz (x +y)}|=1$, we get
\begin{align*}
 \int_{\bR^d} \mu(d\xi) \big\vert \cF h (\xi + z ) \big\vert^2  =   \int_{\bR^{2d}} e^{-iz (x +y)} h(x) h(y) \gamma(x-y) dxdy \leq  \int_{\bR^{2d}}  h(x) h(y) \gamma(x-y) dxdy,
 \end{align*}
 which is exactly $ \int_{\bR^d} \mu(d\xi) \big\vert \cF h(\xi  ) \big\vert^2.$
In particular, by \eqref{bddFG},
\[
\sup_{z\in\bR^d}\int_{\bR^d} \mu(d\xi) \big\vert  \widehat{G}_s (\xi + z ) \big\vert^2 \leq \int_{\bR^d} \mu(d\xi) \big\vert \widehat{G}_s (\xi  ) \big\vert^2 \leq 2(s^2\vee 1) \int_{\bR^d}  \frac{\mu(d\xi)}{1+ |\xi|^2},
\]
which is finite due to Dalang's condition \eqref{DC}. Applying this inequality several times yields
\begin{align*}
&  \int_{\bR^{dp}} \mu(d\pmb{\xi_p}) \prod_{j=1}^{p} \big\vert  \widehat{G}_{w_j}(\xi_j + z_j ) \big\vert^2  \leq   \left( 2(t^2\vee 1)  \int_{\bR^d}   \frac{\mu(d\xi)}{1+ |\xi|^2} \right)^p,
\end{align*}
which is a uniform bound over $(\pmb{z_p}, \pmb{w_p})\in\bR^{dp}\times [0,t]^p$.
\end{proof}

\section{$L^p$ estimates for Malliavin derivatives} \label{sec3}

This section is mainly devoted to the proof of  Theorem \ref{MR1}.  The proof will be done in several steps organized in Sections \ref{sec31},  \ref{sec32},   \ref{sec33}, \ref{sec34} and \ref{sec35}. In Section \ref{sec36}, we record a few consequences of  Theorem  \ref{MR1} that will be used in the proof of Theorem \ref{MR3}   in Section \ref{sec5}.

\subsection{Step 1: Preliminaries} \label{sec31}

 Let us first introduce some handy notation. Recall that for
 $
 \pmb{t_n}:=(t_1,\ldots,t_n)
 $
 and
  $
  \pmb{x_n}:=(x_1,\ldots,x_n)
  $,
  we defined in \eqref{eq:3}
\begin{equation} \notag
f_{t,x,n}(\pmb{t_n},\pmb{x_n})=G_{t-t_{1}}(x-x_{1})
G_{t_1-t_2}(x_1-x_2) \cdots G_{t_{n-1}-t_n}(x_{n-1}-x_n),
\end{equation}
with the convention \eqref{rule1}, and  we denote by $\widetilde{f}_{t,x,n}$ the symmetrization of $f_{t,x,n}$; see \eqref{eq:3wt}.
We treat the time-space variables $(t_i,x_i)$ as one coordinate and
we write
\[
f_{t,x,n}(r,z;\pmb{t_{n-1}},\pmb{x_{n-1}}) := f_{t,x,n}(r,z,  t_1,x_1, \ldots, t_{n-1}, x_{n-1})
\]
as in Notation \textbf{A}-(3). Recall that the solution $u(t,x)$ has the Wiener chaos expansion
 \[
 u(t,x)=1+ \sum_{n=1}^ \infty   I_n(f_{t,x,n}),
 \]
where  the kernel   $f_{t,x,n}$ is not symmetric and in this case, by definition,     $I_n(f_{t,x,n})= I_n\big(\widetilde{f}_{t,x,n} \big)$.

\medskip

Our first goal is to  show that, for any fixed $(r,z) \in [0,t] \times \R^d$ and for any $p\in [2,\infty)$, the series
 \begin{align}
 \sum_{n\geq 1} n I_{n-1}\big( \widetilde{f}_{t,x,n}(r,z; \bullet)  \big) \label{series1}
 \end{align}
  converges in $L^p(\Omega)$, and the sum, denoted by $D_{r,z}u(t,x) $, satisfies the $L^p$ estimates  \eqref{goalz}.

  The first term of the series \eqref{series1} is $\widetilde{f}_{t,x,1}(r,z)=G_{t-r}(x-z)$. In general, for any $n\geq 1$,
\begin{equation}
\label{decomp-ftx}
\widetilde{f}_{t,x,n}(r,z;\bullet)  = \frac{1}{n} \sum_{j=1}^n
h^{(j)}_{t,x,n}(r,z;\bullet),
\end{equation}
where  $h^{(j)}_{t,x,n}(r,z;\bullet)$ is the symmetrization of the function $(\pmb{t_{n-1}},\pmb{x_{n-1}})\to f^{(j)}_{t,x,n}(r,z; \pmb{t_{n-1}},\pmb{x_{n-1}})$, which is obtained from $f_{t,x,n}$ by placing $r$ on position $j$ among the time instants, and $z$ on position $j$ among the space points: With the convention \eqref{rule1},
\begin{align} \nonumber
&f^{(j)}_{t,x,n}(r,z; \pmb{t_{n-1}},\pmb{x_{n-1}})\\
&\quad =G_{t-t_1}(x-x_1) \cdots G_{t_{j-1}-r}(x_{j-1}-z)G_{r-t_j}(z-x_j) \cdots G_{t_{n-2}-t_{n-1}}(x_{n-2}-x_{n-1}). \label{eq:4}
\end{align}
That is,
\begin{align}
f^{(j)}_{t,x,n}(r,z; \bullet) = f_{t,x,j}^{(j)}(r,z;\bullet)\otimes f_{r,z,n-j},\label{decomp-fj}
\end{align}
with $ f_{r,z,1}=1$.   For example, $f^{(1)}_{t,x,1}(r,z; \bullet) = G_{t-r}(x-z)$ and $f^{(1)}_{t,x,n}(r,z; \pmb{t_{n-1}}, \pmb{x_{n-1}} )= G_{t-r}(x-z) f_{r,z,n-1}( \pmb{t_{n-1}}, \pmb{x_{n-1}} )$.
By the definition of the symmetrization,
\begin{equation}
\label{def-h}
h^{(j)}_{t,x,n}(r,z;\pmb{t_{n-1}},\pmb{x_{n-1}})=\frac{1}{(n-1)!} \sum_{\sigma \in \mathfrak{S}_{n-1}}f_{t,x,n}^{(j)}(r,z;t_{\sigma(1)},x_{\sigma(1)}, \ldots,t_{\sigma(n-1)},x_{\sigma(n-1)}).
\end{equation}
 Similarly,
  for $\pmb{s_m}\in [0,t]^m$ and $\pmb{y_m}\in\bR^{dm}$,
  and
  for any $p\in [2,\infty)$,
   we will show that
\begin{align}
D^m_{\pmb{s_m}, \pmb{y_m} }u(t,x) := \sum_{n\geq m} \frac{n!}{(n-m)!} I_{n-m}\big( \widetilde{f}_{t,x,n}(\pmb{s_m}, \pmb{y_m}; \bullet)  \big) \label{seriesm}
 \end{align}
 converges in $L^p(\Omega)$. Note that if the series \eqref{seriesm} converges in $L^p(\Omega)$, we can see that almost surely,  the function
\[
(\pmb{s_m}, \pmb{y_m} ) \mapsto   D^m_{\pmb{s_m}, \pmb{y_m} }u(t,x)
\]
is \emph{symmetric}, meaning that for any $\sigma\in\mathfrak{S}_m$,
\[
D_{s_1, y_1} D_{s_2, y_2} \cdots D_{s_m, y_m} u(t,x) = D_{s_{\sigma(1)}, y_{\sigma(1)}} D_{s_{\sigma(2)}, y_{\sigma(2)}} \cdots D_{s_{\sigma(m)} , y_{\sigma(m)}} u(t,x).
\]
\emph{From now on}, we assume $t>s_1> ... > s_m>0$ without losing any generality.
Note that like    \eqref{decomp-ftx}, we can write
\begin{align}
 \frac{n!}{(n-m)!}   \widetilde{f}_{t,x,n}(\pmb{s_m}, \pmb{y_m}; \bullet) = \sum_{\pmb{i_m}\in \Delta_{n,m}} h^{(\pmb{i_m})}_{t,x,n}(\pmb{s_m}, \pmb{y_m}; \bullet), \label{eq40}
\end{align}
where $\pmb{i_m}\in \Delta_{n,m}$  means $1 \le i_1 < i_2 < \cdots < i_m \le n$  and $h^{(\pmb{i_m})}_{t,x,n}(\pmb{s_m}, \pmb{y_m}; \bullet)$ is the symmetrization of the function $f^{(\pmb{i_m})}_{t,x,n}(\pmb{s_m}, \pmb{y_m}; \bullet)$ that is defined by
\begin{align}
&f^{(\pmb{i_m})}_{t,x,n}(\pmb{s_m}, \pmb{y_m}; \bullet) \label{decomp-fjj} \\
&   = f^{(i_1)}_{t,x,i_1}(s_1, y_1; \bullet) \otimes  f^{(i_2-i_1)}_{s_1,y_1,i_2-i_1}(s_2, y_2; \bullet) \otimes \cdots \otimes  f^{(i_m-i_{m-1})}_{s_{m-1},y_{m-1}, i_m- i_{m-1}}(s_m, y_m; \bullet) \otimes   f_{s_m, y_m, n-i_m},  \notag
\end{align}
which is a generalization of \eqref{decomp-fj}.

\subsection{Step 2: Reduction to white noise in time} \label{sec32}

Let $\dot{\mathfrak{X}}$ denote the Gaussian noise that is white in time and has the same spatial correlation as $W$ and let $\{\mathfrak{X}(f): f\in\cH_0\}$ denote the resulting isonormal Gaussian process; see Section \ref{sec21}.

 For any $p\in[2,\infty)$, we deduce from  \eqref{seriesm} and  \eqref{eq40}  that
 \begin{align*}
 \big\| D^m_{\pmb{s_m}, \pmb{y_m} }u(t,x) \big\|_p &\leq   \sum_{n\geq m} \left\|  I_{n-m}\left(  \sum_{\pmb{i_m}\in \Delta_{n,m} } h^{(\pmb{i_m})}_{t,x,n}(\pmb{s_m}, \pmb{y_m}; \bullet)  \right) \right\|_p  \quad \text{by triangle inequality} \\
 &\leq   \sum_{n\geq m} (p-1)^{\frac{n-m}{2}} \left\|  I_{n-m}\left(  \sum_{\pmb{i_m}\in\Delta_{n,m}} h^{(\pmb{i_m})}_{t,x,n}(\pmb{s_m}, \pmb{y_m}; \bullet)  \right) \right\|_2  \quad \text{by \eqref{hyper}}.
 \end{align*}
The function $\sum_{\pmb{i_m}\in \Delta_{n,m}} h^{(\pmb{i_m})}_{t,x,n}(\pmb{s_m}, \pmb{y_m}; \bullet)$ vanishes outside $\big([0,t]\times\bR^d\big)^{n-m}$, thus we deduce from \eqref{white-ineq} that
\begin{align*}
&\quad \left\|  I_{n-m}\left(  \sum_{\pmb{i_m}\in \Delta_{n,m}} h^{(\pmb{i_m})}_{t,x,n}(\pmb{s_m}, \pmb{y_m}; \bullet)  \right) \right\|_2^2 = (n-m)! \left\|  \sum_{\pmb{i_m}\in \Delta_{n,m}} h^{(\pmb{i_m})}_{t,x,n}(\pmb{s_m}, \pmb{y_m}; \bullet)\right\|_{\cH^{\otimes( n-m)}}^2 \\
& \leq \Gamma_t^{n-m}  (n-m)! \left\|  \sum_{\pmb{i_m}\in \Delta_{n,m}} h^{(\pmb{i_m})}_{t,x,n}(\pmb{s_m}, \pmb{y_m}; \bullet)\right\|_{\cH_0^{\otimes (n-m)}}^2 = \Gamma_t^{n-m} \left\|  I^{\mathfrak{X}}_{n-m}\left(  \sum_{\pmb{i_m}\in \Delta_{n,m}} h^{(\pmb{i_m})}_{t,x,n}(\pmb{s_m}, \pmb{y_m}; \bullet)  \right) \right\|_2^2.
\end{align*}
Therefore, we get
 \begin{align} \label{Dm_Lp}
 \big\| D^m_{\pmb{s_m}, \pmb{y_m} }u(t,x) \big\|_p  &\leq   \sum_{n\geq m} \big[ (p-1) \Gamma_t\big]^{\frac{n-m}{2}} \left\|  \sum_{\pmb{i_m}\in\Delta_{n,m}} I^{\mathfrak{X}}_{n-m} \big(  f^{(\pmb{i_m})}_{t,x,n}(\pmb{s_m}, \pmb{y_m}; \bullet)\big)   \right\|_2.
 \end{align}
 This leads to
\begin{align}\label{b1}
 \big\| D^m_{\pmb{s_m}, \pmb{y_m} }u(t,x) \big\|_p   \leq   \sum_{n\geq m} \big[ (p-1) \Gamma_t\big]^{\frac{n-m}{2}} \sqrt{\mathcal{Q}_{m,n}},
\end{align}
with
\begin{align} \label{Qmn}
\mathcal{Q}_{m,n} :&=\bE\left[  \left( \sum_{\pmb{i_m}\in\Delta_{n,m}}I^{\mathfrak{X}}_{n-m} \big(  f^{(\pmb{i_m})}_{t,x,n}(\pmb{s_m}, \pmb{y_m}; \bullet)\big)   \right)^2 \right]  \le \binom{n}{m}   \sum_{\pmb{i_m}\in\Delta_{n,m}}  \bE\left( I^{\mathfrak{X}}_{n-m} \big(  f^{(\pmb{i_m})}_{t,x,n}(\pmb{s_m}, \pmb{y_m}; \bullet)\big)^2   \right).
\end{align}
The product formula \eqref{prod} and the decomposition \eqref{decomp-fjj} yield, with $(i_0, s_0, y_0)=(0, t,x)$,
\begin{align}
&\mathcal{Q}_{m,n} \leq \binom{n}{m}   \sum_{\pmb{i_m}\in\Delta_{n,m}}  \bE\left(  I^{\mathfrak{X}}_{n - i_{m}}\big( f_{s_m, y_m, n-i_m} \big)^2   \prod_{j=1}^m I^{\mathfrak{X}}_{i_j - i_{j-1}-1} \Big( f^{ (i_j- i_{j-1})  }_{s_{j-1}, y_{j-1}, i_j-i_{j-1}}(s_j, y_j;\bullet)  \Big)^2   \right) \notag \\
&= \binom{n}{m}   \sum_{\pmb{i_m}\in\Delta_{n,m}}  \big\|  I^{\mathfrak{X}}_{n - i_{m}}\big( f_{s_m, y_m, n-i_m} \big)\big\|^2_2  \times  \prod_{j=1}^m\Big\| I^{\mathfrak{X}}_{i_j - i_{j-1}-1} \Big( f^{ (i_j- i_{j-1})  }_{s_{j-1}, y_{j-1}, i_j-i_{j-1}}(s_j, y_j;\bullet) \Big) \Big\|^2_2,   \label{ineq_Qmn}
\end{align}
where the last equality is obtained by using the independence among the random variables inside the expectation. It remains to estimate two typical terms:
\begin{align}
\big\| I^{\mathfrak{X}}_j(f_{r,z,j}) \|_2^2 \quad{\rm and}\quad \Big\| I^{\mathfrak{X}}_{j-1}(f^{(j)}_{t,x,j}(r,z;\bullet ) \big) \Big\|_2^2 ~\text{for $1\leq j\leq n$ and $t>r$}. \label{types}
\end{align}
The first term in \eqref{types} can be estimated as follows. Using Fourier transform in space (see \eqref{FG}), we have, with $t_0=r$,
\begin{align}
\big\| I^{\mathfrak{X}}_j(f_{r,z,j}) \|_2^2  &= j! \big\| \widetilde{f}_{r,z,j} \big\|_{\cH_0^{\otimes j}}^2 = \int_{[0,r]^j} \big\| f_{r,z,j}(\pmb{t_j}, \bullet) \big\|_0^2 d\pmb{t_j}  \label{eq:-1} \\
&=  \int_{r>t_1 > \cdots> t_j >0}  \int_{\bR^{dj}}\big\vert \mathcal{F} f_{r,z,j}(\pmb{t_j}, \pmb{\xi_j})  \big\vert^2 \mu(d\pmb{\xi_j})    d\pmb{t_j}  \notag \\
&= \int_{r>t_1 > \cdots> t_j >0}  \left( \int_{\bR^{dj}} \prod_{k=0}^{j-1} \big\vert  \cF G_{t_{k} - t_{k+1}}(\xi_{k+1} +\cdots+ \xi_j ) \big\vert^2  \mu(d\xi_k)  \right)  d\pmb{t_j}.   \notag
\end{align}
By Lemma \ref{lem_4Qp},
\begin{align}
\big\| I^{\mathfrak{X}}_j(f_{r,z,j}) \|_2^2 &\leq \frac{C^j}{j!},  \label{EST_dj}
\end{align}
where $C= 2(t^2+1) \int_{\R^d}  ( 1+ |\xi |^2)^{-1}\mu(d\xi)$.

\begin{remark}\label{rem_Lp} {\rm By the arguments that lead to \eqref{Dm_Lp}, we can also get, for any $p\in[2,\infty)$,
\[
\big\| u(t,x) \big\|_p \leq 1 + \sum_{n\geq 1} \big\|  I_n(f_{t,x,n} )\big\|_p   \leq 1 + \sum_{n\geq 1} \big[ (p-1) \Gamma_t\big] ^{n/2}  \big\|  I^{\mathfrak{X}}_n(f_{t,x,n} )\big\|_2
\]
and then the estimate \eqref{EST_dj} implies $u(t,x)\in L^p(\Omega)$. Moreover,
\begin{align}\label{calsoRem31}
\sup_{(s,y)\in[0,t]\times\bR^d  } \| u(s,y) \|_p <+\infty ~ \text{for any $t\in\bR_+$.}
\end{align}
 This is done under the Dalang's condition \eqref{DC} only and the case $p=2$ provides another proof of \cite[Theorem  4.4]{BS17} when $d=1,2$.}
\end{remark}

In what follows, we estimate the second term in \eqref{types} separately for the cases $d=1$ and $d=2$. As usual,
we will use $C$ to denote an immaterial constant that may vary from line to line.

\subsubsection{Estimation of $\Big\| I^{\mathfrak{X}}_{j-1}(f^{(j)}_{t,x,j}(r,z;\bullet ) \big) \Big\|_2^2$ when $d=1$}
When $d=1$, $G_t(x) = \frac{1}{2} \mathbf{1}_{\{|x| <t \}}$.  For $j=1$, $I^{\mathfrak{X}}_{j-1}(f^{(j)}_{t,x,j}(r,z;\bullet ) \big)=G_{t-r}(x-z)$ with the convention \eqref{rule1}. For $j\geq 2$, it follows from  the (modified) isometry property \eqref{miso} that
\begin{align*}
\Big\| I^{\mathfrak{X}}_{j-1}(f^{(j)}_{t,x,j}(r,z;\bullet ) \big) \Big\|_2^2  = (j-1)! \Big\|   h^{(j)}_{t,x,j}(r,z;\bullet )   \Big\|_{\cH_0^{\otimes (j-1)}}^2 = \int_{[r,t]^{j-1}} \big\| f^{(j)}_{t,x,j}(r,z; \pmb{t_{j-1}}, \bullet ) \big\|_{0}^2 d\pmb{t_{j-1}},
\end{align*}
where we recall that $h^{(j)}_{t,x,j}(r,z;\bullet )  $ is the symmetrization of $f^{(j)}_{t,x,j}(r,z;\bullet )  $; see \eqref{def-h}.  Then, taking advantage of the simple form of $G_t(x)$ for $d=1$, we get
\[
0\leq  f^{(j)}_{t,x,j}(r,z; \pmb{t_{j-1}}, \bullet )  \leq  \frac{1}{2}\mathbf{1}_{\{ |x - z| < t-r\}}  f_{t,x,j-1}(\pmb{t_{j-1}}, \bullet ),
\]
from which we further get
\begin{align}
\Big\| I^{\mathfrak{X}}_{j-1}(f^{(j)}_{t,x,j}(r,z;\bullet ) \big) \Big\|_2^2  &\leq   G^2_{t-r}(x-z)  \int_{[r,t]^{j-1}} \big\| f_{t,x,j-1}(\pmb{t_{j-1}}, \bullet ) \big\|_0^2 d\pmb{t_{j-1}} \notag  \\
&\leq  \frac{C^{j-1}}{(j-1)!} G^2_{t-r}(x-z), \label{EST_d=1}
\end{align}
where the last inequality follows from \eqref{EST_dj} and \eqref{eq:-1}.

\subsubsection{Estimation of $\Big\| I^{\mathfrak{X}}_{j-1}(f^{(j)}_{t,x,j}(r,z;\bullet ) \big) \Big\|_2^2$ when $d=2$}
Let  $q$ be defined as in  \eqref{def-q} and  \eqref{def-qq} and we fix such a $q$ \emph{throughout this subsection}.   For $j=1$, $I^{\mathfrak{X}}_{j-1}(f^{(j)}_{t,x,j}(r,z;\bullet ) \big)=G_{t-r}(x-z)$ with the convention \eqref{rule1}. For $j\geq 2$, we begin with
\begin{align*}
\Big\| I^{\mathfrak{X}}_{j-1}(f^{(j)}_{t,x,j}(r,z;\bullet ) \big) \Big\|_2^2   &= \int_{[r,t]^{j-1}} \big\| f^{(j)}_{t,x,j}(r,z; \pmb{t_{j-1}}, \bullet ) \big\|_0^2 d\pmb{t_{j-1}}, \\
 &\leq C^{j-1} \int_{t> t_1 > \cdots >t_{j-1}>r} \big\| f^{(j)}_{t,x,j}(r,z; \pmb{t_{j-1}}, \bullet) \big\|_{L^{2q}(\bR^{2j-2}) }^2  d\pmb{t_{j-1}} =C^{j-1} \mathcal{T}_j,
\end{align*}
where we applied Lemma \ref{lem:1} for the inequality above\footnote{The function $\pmb{x_{j-1}} \to f_{t,x,j}^{(j)}(\pmb{t_{j-1}},\pmb{x_{j-1}})=G_{t-t_1}(x-x_1)G_{t_1-t_2}(x_1-x_2)\ldots
 G_{t_{j-1}-r}(x_{j-1}-z)$ has support contained in $\{\pmb{x_{j-1}}\in \bR^{2(j-1)};|x_i-x|<t-t_i, \ \mbox{for all} \ i=1,\ldots,j-1\}$.}  and we denote
\begin{align}\label{def_GJ}
 \mathcal{T}_j :=    \int_{ t>t_1>\cdots>t_{j-1}>r} d\pmb{t_{j-1}}    \left( \int_{\bR^{2(j-1)}} G^{2q}_{t-t_1}(x-x_1)     \cdots G^{2q}_{t_{j-1} -r}(x_{j-1} -z)  d \pmb{x_{j-1}} \right) ^{ 1/q}.
\end{align}
 Note that we can choose  $C$  to  depend only on $(t,\gamma, q)$ and be increasing in $t$.

  \medskip

  \noindent
  \textbf{Case $j=2$}. In this case, we deduce from Lemma  \ref{lem33BNZ} and \eqref{ineq-Gp} that
\begin{equation}
\label{eq-j2}
   \mathcal{T}_2 =   \int_r^t dt_1 ( G^{2q}_{t-t_1} \ast G^{2q}_{t_1-r} )^{1/q} (x-z) \leq CG_{t-r}^{2-\frac 1q}(x-z) \leq C  G^2_{t-r} (x-z).
\end{equation}

  \noindent
  \textbf{Case $j\ge 3$}. In this case,  we use Minkowski inequality with
  respect to the  norm in $L^{1/q}( [t_2,t] ,dt_1)$ in order to get
 \begin{align*}
 \mathcal{T}_{j}  &\leq
 \int_{ t>t_2>\cdots>t_{j-1}>r}
   \Bigg( \int_{\bR^{2(j-2)}}   \left[ \int_{t_2} ^t  \big( G_{t-t_1}^{2q} \ast G_{t_1-t_2}^{2q}\big)^{1/q}(x-x_2)    dt_1 \right] ^q\\
   & \qquad \times
 G^{2q}_{t_2-t_3}(x_2-x_3) \cdots G^{2q}_{t_{j-1} -r}(x_{j-1} -z)   dx_2 \cdots dx_{j-1}    \Bigg) ^{1/q} dt_2 \cdots dt_{j-1}.
 \end{align*}
 Applying Lemma \ref{lem33BNZ} yields
 \begin{align}
 \mathcal{T}_{j}  &\le     A_q
 \int_{t> t_2 > \cdots>  t_{j-1}>r} (t-t_2)^{\frac{1}{q}-1}
   \Bigg( \int_{\bR^{2(j-2)}}     G^{2q-1} _{t-t_2}(x-x_2)  \notag \\
   & \qquad \times
 G^{2q}_{t_2-t_3}(x_2-x_3) \cdots G^{2q}_{t_{j-1} -r}(x_{j-1} -z)   dx_2 \cdots dx_{j-1}    \Bigg) ^{1/q} dt_2 \cdots dt_{j-1}. \label{cont1}
 \end{align}
 If $j=3$, we have
 \begin{align*}
  \mathcal{T}_3   &\le    A_q
 \int_r^t (t-t_2)^{\frac{1}{q}-1} \Bigg( \int_{\bR^{2}}     G^{2q-1} _{t-t_2}(x-x_2)   G^{2q}_{t_2-r}(x_2-z)    dx_2    \Bigg) ^{ 1/q} dt_2.
 \end{align*}
 Owing to \eqref{ineq-Gp}, we can bound $G^{2q-1} _{t-t_2}(x-x_2) $
 by  $(2\pi) (t-t_2)G^{2q} _{t-t_2}(x-x_2) $, and then we apply again Lemma \ref{lem33BNZ} and \eqref{ineq-Gp} to conclude that
\begin{equation}
\label{eq-j3}
 \mathcal{T}_3 \leq   A_q^{2} (2\pi)^{\frac{1}{q}} (t-r)^{\frac{3}{q}-2} G_{t-r}^{2-\frac{1}{q}}(x-z)  \leq C G^2_{t-r}(x-z).
\end{equation}

  For $j\geq 4$,  we continue with the estimate \eqref{cont1}.
 We can first apply Minkowski inequality with   respect to the norm $L^{1/q}\big( [t_4, t_2], dt_3\big)$ and  then apply  Lemma \ref{lem33BNZ} to obtain
  \begin{align}
 \mathcal{T}_{j}  &\le       A_q^{2}
 \int_{t> t_2 >t_4 > \cdots > t_{j-1}> r} dt_2 dt_4  \cdots
 dt_{j-1}  (t-t_2)^{\frac{1}{q}-1}(t_2-t_4)^{\frac{1}{q}-1}
   \Bigg( \int_{\bR^{2(j-3)}}      G^{2q-1} _{t-t_2}(x-x_2)   \nonumber \\
   & \qquad \times G^{2q-1} _{t_2-t_4} (x_2-x_4)
 G^{2q}_{t_4-t_5}(x_4-x_5) \cdots G^{2q}_{t_{j-1} -r}(x_{j-1} -z)   dx_2  dx_4 \cdots dx_{j-1}    \Bigg) ^{1/q}.
 \label{eq:2}
 \end{align}
Note that
\[
G^{2q-1} _{t-t_2}(x-x_2) G^{2q-1} _{t_2-t_4}(x_2-x_4) \leq
\mathbf{1}_{\{ |x-x_4 | \le  t-t_4\}} G^{2q-1} _{t-t_2}(x-x_2)
G^{2q-1} _{t_2-t_4}(x_2-x_4).
\]
Then, by {Cauchy-Schwarz} inequality and \eqref{p-norm-G}, we
 can infer that
    \begin{align*}
 \int_{\R^2}   G^{2q-1} _{t-t_2}(x-x_2) G^{2q-1} _{t_2-t_4} (x_2-x_4)
dx_2
 & \le  \mathbf{1}_{\{ |x-x_4 | \le t-t_4\}}  \| G^{2q-1} _{t-t_2} \| _{L^2(\R^2)}   \| G^{2q-1} _{t_2-t_4} \| _{L^2(\R^2)}  \\
 & =c_1 (t-t_2)^{2-2q}(t_2-t_4)^{2-2q} \mathbf{1}_{\{ |x-x_4 | \le t-t_4\}},
  \end{align*}
 where $c_1= \frac{(2\pi)^{3-4q}}{4-4q}$.  Thus, substituting  this estimate into \eqref{eq:2},   we end up with
\begin{align*}
 \mathcal{T}_{j}  & \leq  A_q^{2} c_1^{1/q}  \int_{t> t_2 >t_4 > \cdots > t_{j-1}> r}
dt_2
dt_4  \cdots  dt_{j-1}  (t-t_2)^{\frac3q-3}(t_2-t_4) ^{\frac3q-3} \\
& \qquad \times \left( \int_{\R^{2(j-4)}} \mathbf{1}_{\{ |x-x_4 |
\le t-t_4\}}
 G^{2q}_{t_4-t_5}(x_4-x_5) \cdots G^{2q}_{t_{j-1} -r}(x_{j-1} -z)    dx_4
 \cdots dx_{j-1} \right)^{1/q}.
\end{align*}
Focusing on the indicators, the right-hand side of this estimate can  be
bounded by
\begin{align*}
 &   A_q^{2} c_1^{1/q} \mathbf{1}_{\{ |x-z | \le t-r\}} \int_{t> t_2 >t_4 > \cdots >
t_{j-1}> r} dt_2
dt_4  \cdots  dt_{j-1}   (t-t_2)^{\frac3q-3}(t_2-t_4) ^{\frac3q-3}\\
& \quad \times \left( \int_{\R^{2(j-4)}}
 G^{2q}_{t_4-t_5}(x_4-x_5) \cdots G^{2q}_{t_{j-1} -r}(x_{j-1} -z)   dx_4
 \cdots dx_{j-1} \right)^{1/q}.
\end{align*}

For $j=4$, using \eqref{indicator}, we have
\begin{equation}
\label{eq-j4}
 \mathcal{T}_{4} \leq  A_q^{2} c_1^{1/q} (t-r)^{\frac{6}{q}-6}   \mathbf{1}_{\{ |x-z | \le t-r\}}   \leq C G_{t-r}^{2}(x-z).
\end{equation}

Now for $j\geq 5$, we just integrate in each of the variables $x_4, \dots,
x_{j-1}$ (with this order) so that, thanks to \eqref{p-norm-G}, we
end up with
\begin{align*}
 \mathcal{T}_{j}  & \leq  A_q^{2} c_1^{1/q} c_2^{j-4} \mathbf{1}_{\{ |x-z | \le t-r\}} \int_{t>
t_2 >t_4 > \cdots
> t_{j-1}> r} dt_2
dt_4  \cdots  dt_{j-1}  \\
& \qquad \times  (t-t_2)^{\frac3q-3}(t_2-t_4) ^{\frac3q-3}
(t_4-t_5)^{\frac2q-2}\cdots (t_{j-1}-r)^{\frac2q-2} \quad\text{with $c_2=\left(\dfrac{(2\pi)^{1-2q}}{2-2q}\right)^{2}$} \\
&\leq    A_q^{2} c_1^{1/q} c_2^{j-4} \frac{ (t-r)^{j-3} }{(j-3)!} (t-r+1)^{j(\frac{2}{q}-2)   }   \mathbf{1}_{\{ |x-z | \le t-r\}},
\end{align*}
where we used the rough estimate $a^\nu \leq (b+1)^{\nu}$ for  $0<a\leq b$ and $\nu>0$.  Thus, using \eqref{indicator}  we obtain:
\begin{equation} \label{eq-j5}
 \mathcal{T}_{j} \leq   \frac{C^{j-3}}{(j-3)!}G^2_{t-r}(x-z) \quad \text{for any}~ j\geq 5.
\end{equation}
Hence, combining the estimates \eqref{eq-j2},   \eqref{eq-j3}, \eqref{eq-j4}  and \eqref{eq-j5} and  taking into account that $ I^{\mathfrak{X}}_{0}(f^{(1)}_{t,x,1}(r,z;\bullet ) \big)=G_{r-s}(z-y)$, we can write
\begin{align} \notag
\Big\| I^{\mathfrak{X}}_{j-1}(f^{(j)}_{t,x,j}(r,z;\bullet ) \big) \Big\|_2^2
\leq
 \begin{cases}
C G_{t-r}^2(x-z) & \text{for $j=1,2,3,4$}\\
\dfrac{C^{j} }{(j-3)!} G_{t-r}^2(x-z) & \text{for $j\geq 5$}
\end{cases},
\end{align}
where the constant $C > 1$ depends on $(t,\gamma, q)$ and is increasing in $t$. For $1\leq j \leq n$, we obtain the following bound
\begin{align} \label{EST_d=2}
 \Big\| I^{\mathfrak{X}}_{j-1}(f^{(j)}_{t,x,j}(r,z;\bullet ) \big) \Big\|_2^2  \leq  \  \frac{C^{j}}{j!}n^3G_{t-r}^2(x-z).
\end{align}

\subsection{Step 3: Proof of \eqref{goalz}} \label{sec33}

  Let us first consider the lower bound in \eqref{goalz} for $d\in\{1,2\}$. For $p\in[2,\infty)$, we deduce from the modified isometry \eqref{miso} that
\[
\big\| D^m_{\pmb{s_m}, \pmb{y_m}} u(t,x) \big\|_p \geq \big\| D^m_{\pmb{s_m}, \pmb{y_m}} u(t,x) \big\|_2 \geq m! \widetilde{f}_{t,x,m}( \pmb{s_m}, \pmb{y_m}).
\]
Now let us establish the upper bound in \eqref{goalz}. By symmetry, we can assume $t>s_1 > \cdots > s_m >0$.
First we consider
the case where $d=2$.   Recall  the definition of $\mathcal{Q}_{m,n}$ from \eqref{Qmn},   and then plugging the estimates \eqref{EST_dj} and \eqref{EST_d=2} into \eqref{ineq_Qmn} yields, with $(i_0, s_0, y_0) = (0, t,x)$,
\begin{align*}
\mathcal{Q}_{m,n} &\leq \binom{n}{m}  \sum_{\pmb{i_m}\in\Delta_{n,m}}  \frac{C^{n-i_m}}{(n-i_m)!} \times  \prod_{j=1}^m \frac{n^3C^{  i_j - i_{j-1}}  }{ (i_j - i_{j-1} )! }  G^2_{s_{j-1} - s_j}(y_{j-1} - y_j)  \\
&\leq (2C)^n n^{3m}  \left( \sum_{\pmb{i_m}\in\Delta_{n,m}} \frac{1}{ i_1! (i_2-i_1)! \cdots (i_m - i_{m-1})! (n- i_m)! } \right) f^2_{t,x,m}(\pmb{s_m}, \pmb{y_m}),
\end{align*}
where we used the rough bound $\binom{n}{m} \leq 2^n$. The sum  in the above display is equal to
\[
 \frac{1}{n!} \sum_{\substack{a_1 + ... + a_{m+1} =n \\ a_i\in\mathbb{N},\forall i}}   \binom{n}{a_1, ..., a_{m+1}} = \frac{(m+1)^n}{n!},
\]
by multinomial formula. That is, we can get
\[
\mathcal{Q}_{m,n} \leq   \frac{\big[C(m+1) \big]^n n^{3m}}{n!}  f^2_{t,x,m}(\pmb{s_m}, \pmb{y_m}),
\]
which, together with  the estimate \eqref{b1}, implies the upper bound in \eqref{goalz}, when $d=2$.

The case $d=1$ can be done in the same way by noticing that the bound in \eqref{EST_d=1} can be replaced by $n\frac{C^j}{j!} G_{t-r}^2(x-z)$ for $1\leq j\leq n$. Then,  like the estimate for $d=2$, we can get,  for $t>s_1> \cdots>s_m>0$,
\[
\mathcal{Q}_{m,n} \leq   \frac{\big[ C(m+1) \big]^n n^{m}}{n!}  f^2_{t,x,m}(\pmb{s_m}, \pmb{y_m}),
\]
which together with the estimate \eqref{b1} implies the upper bound in \eqref{goalz}, when $d=1$.
This completes the proof of  the estimate \eqref{goalz}.

Notice that the upper bound also shows the convergence in $L^p$ for any $p\in [2,\infty)$ of the series \eqref{seriesm}, for any \emph{fixed} $\pmb{s_m}\in [0,t]^m$ and $\pmb{y_m}\in\bR^{dm}$.

\subsection{Step 4: Existence of a measurable version}  \label{sec34}
We claim that there is a random field $Y$ such that $Y(\pmb{s_m}, \pmb{y_m}) = D^m_{\pmb{s_m}, \pmb{y_m}}u(t,x)$ almost surely for almost all $(\pmb{s_m}, \pmb{y_m})\in [0,t]^m\times  \bR^{md}$ and  the mapping
 \[
  (\omega, \pmb{s_m}, \pmb{y_m})\in \Omega\times   [0,t]^m\times  \bR^{md} \longmapsto Y(\omega, \pmb{s_m}, \pmb{y_m})\in\bR
  \]
is jointly measurable. This fact is rather standard and we will sketch the proof only in the case $d=2$.  From the explicit form of the kernels
$f_{t,x,n}$ given in \eqref{eq:3}, it follows that the mapping
\begin{equation} \label{eq19}
(\pmb{s_m}, \pmb{y_m}) \rightarrow   \widetilde{f}_{t,x,n}(\pmb{s_m}, \pmb{y_m}; \bullet)
\end{equation}
is measurable from $[0,t]^m\times \R^{2m}$ to $L^2([0,t]^{n-m} ; L^{2q}(\R^{2(n-m)}))$. Because
\begin{center}
       $L^2([0,t]^{n-m} ; L^{2q}(\R^{2(n-m)}))$ is continuously embedded into $\cH^{\otimes (n-m)}$  (see  \eqref{white-ineq} and \eqref{q-ineq}),
\end{center}
 we deduce
 that  the map
\eqref{eq19} is measurable from $[0,t]^m\times \R^{2m}$  into $\cH^{\otimes (n-m)}$.  This implies that the mapping
\begin{equation} \label{eq19a}
(\pmb{s_m}, \pmb{y_m}) \rightarrow   I_{n-m}( \widetilde{f}_{t,x,n}(\pmb{s_m}, \pmb{y_m}; \bullet) )
\end{equation}
is measurable from  $[0,t]^m\times \R^{2m}$ to $L^2(\Omega)$.  The upper bound in  \eqref{goalz}  implies that
the mapping \eqref{eq19a} belongs to the space
\[
L^{2q} ( [0,t]^m \times \R^{2m} ; L^2(\Omega)) \subset L^{2q} (  [0,t]^m \times \R^{2m} \times \Omega).
\]
From this, it follows that we can find a measurable modification of the process
\[\{ I_{n-m}( \widetilde{f}_{t,x,n}(\pmb{s_m}, \pmb{y_m}; \bullet) )(\omega): (\omega, \pmb{s_m}, \pmb{y_m})\in  \Omega \times [0,t]^m \times \R^{2m} \}.
\]
Finally, by standard arguments we deduce the existence of a measurable modification of the series  \eqref{seriesm}.

 \subsection{Step 5: Proof of $u(t,x)\in \mathbb{D}^{\infty}$}  \label{sec35}

We have already seen in Remark  \ref{rem_Lp} that $u(t,x)\in L^p(\Omega)$ for any $p\in[2,\infty)$. Then, it remains to show
that the function   $D^m_{\pmb{s_m}, \pmb{y_m}}u(t,x)$ defined as the limit of the series \eqref{seriesm} coincides with the $m$th Malliavin
derivative of $u(t,x)$. To do this, it suffices to show that
$\bE\big[ \| D^mu(t,x)\|_{\cH^{\otimes m} }^p \big] <\infty$ for any $m\geq 1$.  By Fubini' theorem and using the upper bound \eqref{goalz}, we write
\noindent
\begin{align*}
&\Big( \bE\big[ \| D^mu(t,x)\|_{\cH^{\otimes m} }^p \big] \Big)^{2/p}  \\
&= \left\|  \int_{[0,t]^{2m}\times\bR^{2md}} d\pmb{s_m}  d\pmb{s'_m}  d\pmb{y_m}  d\pmb{y'_m}   \big( D^m_{\pmb{s_m},\pmb{y_m}   }u(t,x)\big) \big( D^m_{\pmb{s'_m},\pmb{y'_m}   }u(t,x)\big)  \prod_{j=1}^m\gamma_0(s_j-s_j') \gamma(y_j-y_j')      \right\|_{p/2} \\
&\leq   \int_{[0,t]^{2m}\times\bR^{2md}} d\pmb{s_m}  d\pmb{s'_m}  d\pmb{y_m}  d\pmb{y'_m}   \big\| D^m_{\pmb{s_m},\pmb{y_m}   }u(t,x)\big\|_p \big\| D^m_{\pmb{s'_m},\pmb{y'_m}   }u(t,x)\big\|_p  \prod_{j=1}^m\gamma_0(s_j-s_j') \gamma(y_j-y_j')       \\
&\lesssim  \big\| \widetilde{f}_{t,x,m} \big\|^2_{\cH^{\otimes m}} <\infty.
\end{align*}
 This shows  $u(t,x)\in \mathbb{D}^{\infty}$ and completes the proof of Theorem   \ref{MR1}.
 \qedhere

 \begin{remark} \label{rem34} {\rm
 When $d=2,p=2,m=1$ and for the cases (\texttt{a}), (\texttt{b}) in   Hypothesis ${\bf (H1)}$, the upper bound in \eqref{goalz} can be proved in a  much simpler way for almost all  $(r,z)\in[0,t]\times\bR^2$. Let $v_{\lambda}$ be the solution to  the stochastic  wave equation
\[
\begin{cases}
{\displaystyle \frac{\partial^2 v_{\lambda}}{\partial t^2}=\Delta v_{\lambda}+ \lambda v_{\lambda} \dot{\mathfrak{X}} } \\
v_\lambda(0,\bullet) =1, \quad  \dfrac{\partial v_\lambda}{\partial t} (0, \bullet) = 0,
\end{cases}
\]
where $\lambda>0$ and $\dot{\mathfrak{X}}$ is given as before. This solution has the chaos expansion $v_{\lambda}(t,x)=\sum_{n\geq 0} \lambda^{n}I_{n}^{\mathfrak{X}}(f_{t,x,n})$ and its Malliavin derivative has the chaos expansion
\[
D_{r,z}v_{\lambda}(t,x)=\sum_{n\geq 1} \lambda^{n}  I_{n-1}^{\mathfrak{X}}\left(\sum_{j=1}^{n} h_{t,x,n}^{(j)}(r,z;\bullet)\right);
\]
 see \eqref{series1} and \eqref{decomp-ftx}.        From this, we infer that  for any $(\lambda, t, x)\in (0,\infty)^2\times\bR^2$ and for \emph{almost every}  $(r,z)\in [0,t]\times\bR^2$,
\begin{equation}
\label{white-D}
\big\| D_{r,z}v_{\lambda}(t,x) \big\|_2^2 =\sum_{n\geq 1} (n-1)! \, \lambda^{2n} \Big\|\sum_{j=1}^{n}h_{t,x,n}^{(j)}(r,z;\bullet)\Big\|_{\cH_0^{\otimes (n-1)}}^2 \leq C_{\lambda,t,\gamma}G_{t-r}^2(x-z),
\end{equation}
where $C_{\lambda,t,\gamma}>0$ is a constant depending on $(\lambda, t, \gamma)$ and is increasing  in $t$.  The inequality above is due to Theorem 1.3 of \cite{NZ20} for case (\texttt{a}), respectively Theorem 1.2 of \cite{BNZ20}  for  case (\texttt{b}). Therefore,
\begin{align*}
\big\| D_{r,z}u(t,x) \big\|_2^2 &=\sum_{n\geq 1} (n-1)! \, \big\|\sum_{j=1}^{n}h_{t,x,n}^{(j)}(r,z;\bullet)\big\|_{\cH^{\otimes (n-1)}}^2\\
&  \leq
\sum_{n\geq 1} (n-1)! \, \Gamma_t^{n-1} \big\|\sum_{j=1}^{n}h_{t,x,n}^{(j)}(r,z;\bullet)\big\|_{\cH_0^{\otimes (n-1)}}^2 ~\text{by  \eqref{white-ineq}}.
\end{align*}
Thus,  using  \eqref{white-D} with $\lambda=\sqrt{\Gamma_t}$, we get  $\big\| D_{r,z}u(t,x) \big\|_2^2  \leq C_{\Gamma_t,t,\gamma}G_{t-r}^2(x-z)$.
}
\end{remark}

 \subsection{Consequences of Theorem \ref{MR1}}\label{sec36}

We will  establish two estimates that will be useful in Section \ref{sec5}.

\begin{corollary}
\label{D-norm}
Let $d=1,2$. Then,  for any finite  $T>0$,
\begin{equation}
\label{sup-D}
\sup_{(t,x)\in [0,T]\times \bR^d} \, \sup_{r\in [0,t]} \bE \Big[ \big\| |D_{r,\bullet} u(t,x)| \big\|_{0}^2 \Big] <\infty.
\end{equation}
  In particular,   $D_{r,\bullet}u(t,x)(\omega) \in |\cP_{0}|$ for almost every $(\omega,r) \in \Omega \times [0,t]$, where $|\cP_0|$ is defined in
  \eqref{|cP_0|}.
\end{corollary}

\begin{proof}
We work with a version of $\{ D_{r,z}u(t,x): (r,z)\in[0,t]\times\bR^2\}$ that is jointly measurable. By Fubini's theorem and Cauchy-Schwarz inequality, we have
\begin{align*}
 \bE \Big[ \big\| |D_{r,\bullet} u(t,x)| \big\|_{0}^2 \Big] & \leq \bE \int_{\bR^{2d}} |D_{r,z}u(t,x)||D_{r,z'}u(t,x)|\gamma(z-z')dzdz'  \\
 & \leq \int_{\bR^{2d}} \| D_{r,z}u(t,x)\|_2 \| D_{r,z'}u(t,x)\|_2 \gamma(z-z')dzdz'  \\
 &\leq C  \int_{\bR^{2d}} G_{t-r}(x-z)G_{t-r}(x-z')\gamma(z-z')dzdz'   \quad\text{by Theorem \ref{MR1}}\\
 &=  C  \int_{\bR^d}   \mu(d\xi) \big\vert  \widehat{G}_{t-r}(\xi) \big\vert^2 \quad \text{using Fourier transform}\\
 &\leq 2C  (t^2\vee 1) \int_{\bR^d}\frac{\mu(d\xi)}{1+|\xi|^2}~\text{by  \eqref{ineq1}},
\end{align*}
where    $C$ is a constant   depending on $\gamma_0,\gamma,t$ and is increasing in $t$. The above (uniform) bound   implies   \eqref{sup-D}.  Hence, $D_{r,\bullet}u(t,x)(\omega) \in |\cP_{0}|$ for almost all $(\omega,r) \in \Omega \times [0,t]$.
\end{proof}

 The space  $ |\cH \otimes \cP_0|$ appearing in the next corollary is defined as the set of measurable functions  $h:\R_+\times \R^{2d} \to \R$ such that
 \[
 \int_{\R_+^2\times \R^{4d}}   |h(r, w, z) | | h(r', w', z')|  \gamma_0(r-r') \gamma(w-w') \gamma(z-z') dw dw' dz dz' dr dr'<\infty.
\]
Then,  $ |\cH \otimes \cP_0| \subset \cH \otimes \cP_0$.
\begin{corollary}
\label{D2-norm}
Let $d=1,2$.
For almost all $(\omega,r) \in \Omega \times [0,t]$, $D D_{r,\bullet} u(t,x)(\omega) \in  |\cH \otimes \cP_0|$ and
for any  finite $T>0$,
\begin{align}\label{sup-DD}
\sup_{(t,x) \in [0,T] \times \bR^d}\sup_{r \in [0,t]} \bE \left( \Big\|   \big\vert  D D_{r,\bullet}u(t,x) \big\vert \Big\|_{\cH \otimes \cP_0}^2 \right)< +\infty.
\end{align}
\end{corollary}

\begin{proof} Using Theorem \ref{MR1}, Cauchy-Schwarz inequality  and the estimate \eqref{goalz}, we can write
\begin{align*}
&\bE \left( \Big\|   \big\vert  D D_{r,\bullet} u(t,x) \big\vert \Big\|_{\cH \otimes \cP_0}^2 \right)  = \bE  \Bigg(\int_{[0,t]^2} \int_{\bR^{4d}}
|D_{(\theta,w),(r,z)}^2u(t,x)| |D_{(\theta',w'),(r,z')}^2u(t,x)|\\
& \qquad \qquad\qquad\qquad\qquad \qquad\qquad \times \gamma_0(\theta-\theta')\gamma(w-w') \gamma(z-z')dw dw'dzdz'd\theta d\theta' \Bigg) \\
&\leq \int_{[0,t]^2} \int_{\bR^{4d}}
\big\|D_{(\theta,w),(r,z)}^2u(t,x) \big\|_2  \big\|D_{(\theta',w'),(r,z')}^2u(t,x) \big\|_2\\
& \qquad\qquad\qquad\quad \quad \times \gamma_0(\theta-\theta')\gamma(w-w') \gamma(z-z')dw dw'dzdz'd\theta d\theta'  \\
&\leq  C\int_{[0,t]^2} \int_{\bR^{4d}}
\widetilde{f}_{t,x,2}(r,z,\theta, w) \widetilde{f}_{t,x,2}(r,z',\theta', w') \gamma_0(\theta-\theta')\gamma(w-w') \gamma(z-z')dw dw'dzdz'd\theta d\theta'.
\end{align*}
As a consequence,
\begin{align*}
\bE \left( \Big\|   \big\vert  D D_{r,\bullet} u(t,x) \big\vert \Big\|_{\cH \otimes \cP_0}^2 \right) &
\le C   \int_{\R^{2d} } \| \widetilde{f}_{t,x,2}(r,z;\bullet) \|_{\cH} \| \widetilde{f}_{t,x,2}(r,z';\bullet) \|_{\cH}\gamma(z-z') dzdz'.
\end{align*}
By the arguments used in the proof of Theorem \ref{MR1},     it follows that
\[
 \| \widetilde{f}_{t,x,2}(r,z;\bullet) \|_{\cH}  \le C G_{t-r}(x-z).
 \]
 Therefore,
\[
\bE \left( \Big\|   \big\vert  D D_{r,\bullet} u(t,x) \big\vert \Big\|_{\cH \otimes \cP_0}^2 \right) \leq C \int_{\bR^{2d}} \gamma(z-z') G_{t-r}(x-z)G_{t-r}(x-z')dzdz'
\]
and the same argument as in the proof of Corollary \ref{D-norm} ends our proof.
\end{proof}

\begin{remark}
{\rm Note that for any finite $T>0$,   $\bE \big(\big\| \vert D^2 u(t,x)\vert \big\|_{\cH^{\otimes 2}}^2\big) < \infty$ for any $(t,x) \in [0,T] \times \bR^d$.}
\end{remark}

\section{Gaussian fluctuation: Proof of Theorem \ref{MR2}} \label{sec4}

Recall  that
\[
F_R(t)= \int_{B_R} \big[ u(t,x) -1 \big] dx
\]
and $\sigma_R(t)  =  \sqrt{\text{Var}\big( F_R(t) \big)  }$. First, we need to obtain the limiting covariance structure, which is the content of Proposition \ref{PROP_COV}. It will give us the growth order of  $\sigma_R(t)$. Then, in Section \ref{sec42}, we apply the second-order Gaussian Poincar\'e inequality to establish the quantitative CLT for $F_R(t)/\sigma_R(t)$. Finally, we will prove the functional CLT by showing the  convergence of the  finite-dimensional distributions and the tightness.

\subsection{Limiting covariance}  \label{sec41}

\begin{proposition} \label{PROP_COV}
{\rm
Let $u$ denote the solution to the hyperbolic Anderson model \eqref{wave} and assume that the non-degeneracy condition \eqref{NDC} holds.  Then, the following results hold true:

(1) Suppose $d \in\{1,2\}$ and  $\gamma(\bR^d) \in(0, \infty)$. Then, for any $t,s\in(0,\infty)$,
\begin{align}\label{COV(1)}
\lim_{R\to\infty} R^{-d} \bE\big[ F_R(t) F_R(s) \big] = \omega_d \sum_{p\geq 1} p! \int_{\bR^d} \big\langle \widetilde{f}_{t,x,p}, \widetilde{f}_{s,0,p} \big\rangle_{\cH^{\otimes p}}dx,
\end{align}
see also \eqref{COV:G}. In particular, $\sigma_R(t) \sim R^{d/2}$.

\medskip

(2) Suppose $d \in \{1,2\}$ and $\gamma(x) = |x|^{-\beta}$ for some $\beta\in(0, 2\wedge d)$. Then, for any $t,s\in(0,\infty)$,
\begin{align}\label{COV(2)}
\lim_{R\to\infty} R^{\beta-2d} \bE\big[ F_R(t) F_R(s) \big] = \kappa_{\beta, d} \int_0^t dr\int_0^s dr' \gamma_0(r-r') (t-r)(s-r'),
\end{align}
where $\kappa_{\beta, d} = \int_{B_1^2} dxdy | x- y |^{-\beta}$ is introduced in \eqref{KBD}.  In particular, $\sigma_R(t) \sim R^{d- \frac{\beta}{2}}$.

\medskip

(3) Suppose $d=2$ and $\gamma(x_1,x_2) =\gamma_1(x_1)\gamma_2(x_2)$ satisfies one of the following conditions:
\begin{align}\label{cases413}
\begin{cases}
(c_1)   & \gamma_i(x_i) = |x_i|^{-\beta_i}~\text{for some $\beta_i\in(0,1)$, $i=1,2$;} \\
\quad\\
(c_2)  & \gamma_1\in L^1(\bR) ~{\rm and}~ \gamma_2(x) = |x|^{-\beta}~\text{for some $\beta\in(0,1)$}
\end{cases}\,\, .
\end{align}
For any $s,t\in(0,\infty)$, the following results hold true:
 \begin{enumerate}
 \item[$(r_1)$]   In $(c_1)$, we have
\begin{align}
\lim_{R\to \infty} R^{\beta_1-\beta_2-4} \bE\big[ F_R(t) F_R(s) \big] &=  K_{\beta_1, \beta_2} \int_0^t dr\int_0^s dr' \gamma_0(r-r') (t-r)(s-r'),  \label{COV(c2)}
\end{align}
where $K_{\beta_1, \beta_2}$ is defined in \eqref{Kbeta12}.

\item[$(r_2)$] In $(c_2)$, we have
\begin{align}\label{COV(c3)}
\lim_{R\to \infty} R^{\beta-3}  \bE\big[ F_R(t) F_R(s) \big]  = \gamma_1(\bR) \mathcal{L}_\beta   \int_0^t dr\int_0^s dr' \gamma_0(r-r') (t-r)(s-r'),
\end{align}
where $\mathcal{L}_\beta$ is defined in \eqref{def_L1B}.

\end{enumerate}

}
\end{proposition}

 \subsubsection{Proof of part (1) in Proposition \ref{PROP_COV}}       \label{sec411}

\paragraph{Preparation.}
In the following, we will denote by $\varphi$ the density of $\mu$.
 For $0< s\leq t < \infty$ and $x,y\in\bR^d$, we have
\begin{align*}
\bE\big[ u(t,x) u(s,y)\big]-1 &=\sum_{p\geq 1} p! \big\langle  \widetilde{f}_{t,x,p},   \widetilde{f}_{s,y,p}\big\rangle_{\mathcal{H}^{\otimes p}} \\
&=: \sum_{p\geq 1} \frac{1}{p!} \Phi_p(t,s; x-y),
\end{align*}
where $ \widetilde{f}_{t,x,p}\in\mathcal{H}^{\otimes p}$ is defined as in \eqref{eq:3}-\eqref{eq:3wt} and $\Phi_p(t,s; x-y)$, defined in the obvious manner, depends only on the difference $x-y$. To see this dependency and to prepare for the future computations, we rewrite $\Phi_p(t,s; x-y)$ using Fourier transform in space:
\begin{align}
& \Phi_p(t,s; x-y) = (p!)^2  \big\langle  f_{t,x,p},   \widetilde{f}_{s,y,p}\big\rangle_{\mathcal{H}^{\otimes p}}   \notag \\
&= p! \sum_{\sigma\in\mathfrak{S}_p}   \int_{\Delta_p(t)} d\pmb{s_p} \int_{[0,s]^p} d\pmb{\tilde{s}_p} \left( \prod_{j=1}^p \gamma_0(s_j -  \tilde{s}_j  ) \right) \int_{\bR^{2pd}} d\pmb{y_p} d\pmb{\tilde{y}_p} \left(\prod_{j=1}^p\gamma(y_j - \tilde{y}_j)  \right)  \notag\\
&\qquad \times  \left( \prod_{j=0}^{p-1} G_{s_{j} - s_{j+1}}(y_j - y_{j+1}) \right) \left( \prod_{j=0}^{p-1} G_{\tilde{s}_{\sigma(j)} - \tilde{s}_{\sigma(j+1)}}( \widetilde{y}_{\sigma(j)} - \widetilde{y}_{\sigma(j+1)}) \right)  \label{neg}  \\
&= p! \sum_{\sigma\in\mathfrak{S}_p}   \int_{\Delta_p(t)} d\pmb{s_p} \int_{[0,s]^p} d\pmb{\tilde{s}_p} \left( \prod_{j=1}^p \gamma_0(s_j -  \tilde{s}_{j}  ) \right) \int_{\bR^{pd}}  d  \pmb{\xi_p}  \left( \prod_{j=1}^p \varphi(\xi_j) \right)   e^{-i (x-y)\cdot (\xi_1+\cdots+ \xi_p)} \notag\\
&\qquad \times  \left( \prod_{j=0}^{p-1} \widehat{G}_{s_{j} - s_{j+1}}( \xi_p+\cdots+ \xi_{j+1} ) \right) \left( \prod_{j=0}^{p-1} \widehat{G}_{\tilde{s}_{\sigma(j)} - \tilde{s}_{\sigma(j+1)}}(\xi_{\sigma(p)}+ \cdots+  \xi_{\sigma(j+1)}) \right),  \label{FT-Phi}
\end{align}
where $\Delta_p(t) =\{ \pmb{s_p}: t> s_1> \cdots > s_p>0\}$, $(s_0, y_0, \tilde{s}_{\sigma(0)}, \tilde{y}_{\sigma(0)}) = (t,x,s,y)$, $\widehat{G}_t(\xi) = \frac{\sin(t |\xi| )}{| \xi|}$ is introduced in \eqref{FG} and we have used again the convention $G_t(z)=0$ for $t\leq 0$.

Relation \eqref{neg} shows that $ \Phi_p(t,s; x-y)$ is always nonnegative and  equality \eqref{FT-Phi} indicates that  $\Phi_p(t,s; x-y)$ indeed depends only on the difference $x-y$, so that we can write
\begin{align}\label{PHI-p}
\Phi_p(t,s; z) = (p!)^2 \big\langle  \widetilde{f}_{t,z,p},   \widetilde{f}_{s,0,p}\big\rangle_{\mathcal{H}^{\otimes p}}.
\end{align}
Note that  $\Phi_p(t,t; 0)$ coincides with $\alpha_p(t)$
given in \cite[Equation (4.11)]{BS17}.
Moreover,  applying Lemma \ref{lem_ab} with     $\mu_p(d\pmb {\xi_p}) = \varphi(\xi_1) \cdots \varphi(\xi_p) d\xi_1 \cdots  d\xi_p$ and
$
g(s_1,\xi_1, \dots , s_p,\xi_p) =  \prod_{j=0}^{p-1}  \vert \widehat{G}_{s_{j} - s_{j+1}}( \xi_p+\cdots+ \xi_{j+1} )  \vert,
$
we get (with $s\leq t$)
\begin{align}
&\Phi_p(t,s; z)   \leq \Gamma_t^p  p! \int_{\Delta_p(t)} d\pmb{s_p} \int_{\bR^{pd}}  \mu(d\pmb{\xi_p} )     \prod_{j=0}^{p-1} \Big\vert \widehat{G}_{s_{j} - s_{j+1}}( \xi_p+\cdots+ \xi_{j+1} ) \Big\vert^2, \label{fini:Phi}
\end{align}
  where we recall that $\Gamma_t = \int_{-t}^t \gamma_0(a)da$ and point out that the right-hand side of  \eqref{fini:Phi} is finite by applying Lemma \ref{lem_4Qp} with $z_j = \xi_{j+1}+\cdots +\xi_p$ and $z_p=0$.

 \bigskip

Now we are ready to show \eqref{COV(1)}.

\begin{proof}[Proof of   \eqref{COV(1)}]     Let us begin with
\begin{align*}
\frac{\bE\big[ F_R(t) F_R(s)\big]}{R^d} & = \int_{B_R^2} dx dy \frac{\bE\big[ u(t,x) u(s,y)\big]-1}{R^d} = \sum_{p\geq 1} \frac{\omega_d }{p!} \int_{\bR^d} \frac{\text{Leb}\big( B_R\cap B_R(-z) \big)   }{\text{Leb}( B_R )} \Phi_p(t,s;z)dz,
\end{align*}
where $\omega_1=2$, $\omega_2=\pi$ and $\text{Leb}(A)$ stands for the Lebesgue measure of $A\subset \bR^d$. We claim that
\begin{align}\label{claim1:sec41}
\sum_{p\geq 1} \frac{1}{p!} \int_{\bR^d} \Phi_p(t,s;z)dz < \infty,
\end{align}
from which and the dominated convergence theorem we can deduce that
\begin{align}\label{COV:C1C4}
\lim_{R\to\infty} R^{-d}\bE\big[ F_R(t) F_R(s)\big] = \omega_d \sum_{p\geq 1} \frac{1}{p!} \int_{\bR^d}  \Phi_p(t,s;z)dz.
\end{align}
We remark that, by the monotone convergence theorem and the fact that $\Phi_p(t,s;z)\geq 0$ for all $z\in\bR^d$,
the claim \eqref{claim1:sec41} is equivalent to
\begin{align}\label{claim2:sec41}
\sup_{\e>0} \sum_{p\geq 1} \frac{1}{p!} \int_{\bR^d} \Phi_p(t,s;z) e^{-\frac{\e}{2} |z|^2}dz < \infty.
\end{align}
Let us show the claim \eqref{claim2:sec41}.
\medskip

For $p=1$, by direct computations, we can perform integration with respect to $z, y, \tilde{y} $ (one by one in this order) to obtain
\begin{align}
 \int_{\bR^d}  \Phi_1(t,s;z)dz & =  \int_{\bR^d} \left(  \int_0^t dr \int_0^s d\tilde{r} \gamma_0(r-\tilde{r}) \int_{\bR^{2d}} dyd\tilde{y} G_{t-r}(y-z)G_{s-\tilde{r}}(\tilde{y}) \gamma(y - \tilde{y} )  \right)    dz  \notag \\
  &=\gamma(\bR^d) \int_0^t  \int_0^s    \gamma_0(r-\tilde{r})  (t-r)   (s-\tilde{r} )d\tilde{r} dr  \leq \gamma(\bR^d) t^3 \Gamma_t, \label{p=1}
\end{align}
where $ \int_{\bR^d}  \Phi_1(t,s;z)dz >0$    due to the non-degeneracy assumption \eqref{NDC} on   $\gamma_0$. This implies in particular that $\sigma_R(t) > 0$ for large enough $R$.
\medskip

Next we consider $p\geq 2$.   Using the expression \eqref{FT-Phi} and applying Fubini's theorem with the dominance condition \eqref{fini:Phi}, we can write
\begin{align}
&\mathcal{T}_{p,\e}:=(2\pi)^{-d} \int_{\bR^d} \Phi_p(t,s;z) e^{-\frac{\e}{2} |z|^2}dz =  p! \sum_{\sigma\in\mathfrak{S}_p}  \int_{\Delta_p(t)} d\pmb{s_p} \int_{[0,s]^p} d\pmb{\tilde{s}_p}  \prod_{j=1}^p \gamma_0(s_j -  \tilde{s}_{j}  )  \int_{\bR^{pd}}  d\pmb{\xi_p}    \notag  \\
&\quad  \times    p_\e(\xi_1+\cdots + \xi_p)   \prod_{j=0}^{p-1} \varphi(\xi_{j+1}) \widehat{G}_{s_{j} - s_{j+1}}( \xi_p+\cdots+ \xi_{j+1} )   \widehat{G}_{\tilde{s}_{\sigma(j)} - \tilde{s}_{\sigma(j+1)}}(\xi_{\sigma(p)}+ \cdots+  \xi_{\sigma(j+1)})   \notag \\
&\leq \Gamma_t^p p!  \int_{\Delta_p(t)} d\pmb{s_p} \int_{\bR^{pd}}  d\pmb{\xi_p} \left(\prod_{j=1}^p\varphi(\xi_j) \right)  p_\e\left(\sum_{j=1}^p\xi_j\right)    \prod_{j=0}^{p-1} \Big\vert \widehat{G}_{s_{j} - s_{j+1}}( \xi_p+\cdots+ \xi_{j+1} ) \Big\vert^2,  \label{CE1}
\end{align}
where $p_\e(\xi) =(2\pi \e)^{-d/2} e^{-|\xi|^2/(2\e)} $ for $\xi\in\bR^d$  and we applied Lemma \ref{lem_ab} with $ \mu_p(d\pmb{\xi_p}) =\varphi(\xi_1)\cdots \varphi(\xi_p) p_\e(\xi_1+\cdots+\xi_p) d\xi_1 \cdots d\xi_p$.

 Next, we make the  change of variables
\[
\eta_j= \xi_p +\cdots + \xi_j~\text{with the convention $\eta_{p+1}=0$},
\]
and the bound \eqref{CE1} becomes
\begin{align}
\mathcal{T}_{p,\e} &\leq  \Gamma_t^p p!  \int_{\Delta_p(t)} d\pmb{s_p} \int_{\bR^{pd}}  d\pmb{\eta_p} \left(\prod_{j=1}^p\varphi(\eta_j -\eta_{j+1}) \right)  p_\e(\eta_1)    \prod_{j=0}^{p-1} \Big\vert \widehat{G}_{s_{j} - s_{j+1}}( \eta_{j+1} ) \Big\vert^2 \notag \\
&\leq  \Gamma_t^p p!  \| \varphi\|_\infty t^2 \int_{\bR^d} d\eta_1 p_\e(\eta_1)  \int_{\Delta_p(t)} d\pmb{s_p} \int_{\bR^{pd-d}}  d\eta_2 \cdots d\eta_p\left(\prod_{j=2}^p\varphi(\eta_j -\eta_{j+1}) \right) \notag   \\
& \quad\times \Big\vert \widehat{G}_{s_{1} - s_2}( \eta_{2} )\widehat{G}_{s_{2} - s_3}( \eta_{3} )\cdots \widehat{G}_{s_{p-1} - s_p}( \eta_{p} ) \Big\vert^2=  \Gamma_t^p p!  \| \varphi\|_\infty t^2 \int_{\bR^d} d\eta_1 p_\e(\eta_1)  Q_{p-1},  \label{CE2}
\end{align}
where we used $| \widehat{G}_{t-s_1} (\xi) | \leq t$, and $\varphi(\eta_1-\eta_2)\leq \| \varphi\|_\infty$
(which is finite because $\gamma(\bR^d)<\infty$) to obtain \eqref{CE2}, and
\begin{align}\label{def-Qp}
Q_{p-1}:= \int_{\Delta_{p}(t)} d\pmb{s_{p}} \int_{\bR^{pd-d}}      \prod_{j=2}^p\varphi(\eta_j -\eta_{j+1}) \big\vert \widehat{G}_{s_{j-1} -s_j }( \eta_{j} ) \big\vert^2 d\eta_j.
\end{align}
 Observe that $Q_{p-1}$ does not depend on $\eta_1$, thus for any $p\geq 2$
\begin{align}
\mathcal{T}_{p,\e} \leq   \Gamma_t^p p!  \| \varphi\|_\infty t^2 Q_{p-1}. \label{bddTpe}
\end{align}
By Lemma \ref{lem_4Qp}, we have for any $p\geq 2$
\[
Q_{p-1} \leq   \left( 2(t^2\vee 1)  \int_{\bR^d}   \frac{\mu(d\xi)}{1+ |\xi|^2} \right)^{p-1} \frac{t^p}{p!} \leq \frac{C^p}{p!}.
\]
Now, plugging the above estimate and  \eqref{bddTpe}   into  \eqref{claim2:sec41}, and
using \eqref{p=1} for $p=1$, we have
\[
\sup_{\e>0} \sum_{p\geq 1} \frac{1}{p!} \int_{\bR^d} \Phi_p(t,s;z) e^{-\frac{\e}{2} |z|^2}dz \leq \gamma(\bR^d) t^3 \Gamma_t+(2\pi)^d   \| \varphi\|_\infty t^2     \sum_{p\geq 2} \frac{\Gamma_t^p C^{p}}{p!} < +\infty.
\]
 This shows the claim \eqref{claim2:sec41} and the claim  \eqref{claim1:sec41}, which 
  confirm the limiting covariance structure \eqref{COV:C1C4}. Hence the proof of \eqref{COV(1)} is completed.      \qedhere

 \end{proof}

 \subsubsection{Proof of part (2) in Proposition \ref{PROP_COV}}      \label{sec412}

In this case,  the corresponding spectral density is given by
$
\varphi(\xi) = c_{d,\beta} | \xi |^{\beta - d}
$, for some constant $c_{d,\beta}$ that only depends on $d$ and $\beta$.

Now, let us recall the chaos expansion \eqref{WCE} of $u(t,x)$, from which we can obtain the following chaos expansion of $F_R(t)$:
\[
F_R(t) = \sum_{p\geq 1}  \mathbf{J}_{p,R}(t),
\]
where  $\mathbf{J}_{p,R}(t):= I_p\left(  \int_{|x| \leq R} \widetilde{f}_{t,x,p}  dx \right)$ is the projection of $F_R(t)$ onto the $p$th Wiener chaos, with $ \widetilde{f}_{t,x,p} $ given as in \eqref{eq:3wt}.

\medskip

Using the orthogonality of Wiener chaoses with different order, we have
\[
\sigma^2_R(t)= \text{Var}\big( F_R(t) \big) = \sum_{p\geq 1} \text{Var}\big(    \mathbf{J}_{p,R}(t) \big).
\]
Let us first consider the variance of $   \mathbf{J}_{1,R}(t)$. With $B_R=\{ x\in\bR^d:  |x | \leq R\}$, we can write
\begin{align}
&\quad \text{Var}\big(    \mathbf{J}_{1,R}(t) \big)  = \int_{B_R^2} dxdx' \langle G_{t-\bullet}(x-\ast),  G_{t-\bullet}(x'-\ast) \rangle_{\mathcal{H}} \notag \\
 & = \int_{B_R^2} dxdx'  \int_{[0,t]^2} dsds' \gamma_0(s-s') \int_{\bR^d} d\xi \varphi(\xi) e^{-i (x-x') \cdot \xi} \widehat{G}_{t-s}(\xi)  \widehat{G}_{t-s'}(\xi).  \label{oneBeta}
\end{align}
Then, making the change of variables  $(x, x', \xi)\to (Rx, Rx', \xi/R)$, we get
\begin{align*}
& \text{Var}\big(    \mathbf{J}_{1,R}(t) \big)  =R^{2d-\beta} \int_{[0,t]^2} dsds' \gamma_0(s-s') \int_{B_1^2} dxdx'  \int_{\bR^d} d\xi \varphi(\xi) e^{-i (x-x') \cdot \xi} \widehat{G}_{t-s}(\xi/R)  \widehat{G}_{t-s'}(\xi/R).
\end{align*}
Note that $\widehat{G}_{t}(\xi/R)$ is uniformly bounded and convergent to $t$ as $R\to\infty$; observe also that
\begin{equation}
\label{LR}
\ell_{R}(\xi):=\int_{B_R^2} dxdx' e^{-i (x-x') \cdot \xi}  =  \big\vert\mathcal{F}\mathbf{1}_{B_R}\big\vert^2(\xi) \in[0,\infty).
\end{equation}
Thus     we deduce from the dominated convergence theorem that, with $\kappa_{\beta,d} :=\int_{B_1^2} dxdx'  | x- x' |^{-\beta}$,
\begin{align}
\frac{\text{Var}\big(    \mathbf{J}_{1,R}(t) \big)}{R^{2d-\beta}} \xrightarrow{R\to\infty} &\int_{[0,t]^2} dsds' \gamma_0(s-s') (t-s) (t-s')    \int_{\bR^d} d\xi \varphi(\xi)   \big\vert\mathcal{F}\mathbf{1}_{B_1}\big\vert^2(\xi)     \notag  \\
&= \kappa_{\beta,d} \int_{[0,t]^2} dsds' \gamma_0(s-s') s s'  \label{Riesz:p=1}.
\end{align}
In the same way, we can get
\begin{align}
\frac{ \bE\big[  \mathbf{J}_{1,R}(t) \mathbf{J}_{1,R}(s)  \big]}{R^{2d-\beta}} \xrightarrow{R\to\infty} &\kappa_{\beta,d}\int_{0}^t  dr \int_0^s dr' \gamma_0(r-r')  (t-r)(s-r')    \label{Riesz:p=11}
\end{align}
In what follows, we will show that as $R\to\infty$,
\begin{align}\label{smallo}
\sum_{p\geq 2} \text{Var}\big(    \mathbf{J}_{p,R}(t) \big) = o(R^{2d-\beta}).
\end{align}
In view of the orthogonality again, the above claim \eqref{smallo} and the results  \eqref{Riesz:p=1}-\eqref{Riesz:p=11} imply that the first chaos of $F_R(t)$ is dominant and
\[
\frac{ \bE\big[  F_R(t) F_{R}(s)  \big] }{R^{2d-\beta}} \xrightarrow{R\to\infty} \kappa_{\beta,d}\int_{0}^t dr \int_0^s dr' \gamma_0(r-r') (t-r)(s-r'),
\]
which gives us the desired limiting covariance structure. Moreover, we obtain immediately that the process $\big\{R^{-d+\frac{\beta}{2} } F_R(t): t\in\bR_+\big\}$ converges in finite-dimensional distributions to the centered Gaussian process $\mathcal{G}_\beta$, whose covariance structure is given by  \eqref{COV:Gbeta}.

The rest of    Section \ref{sec412} is then devoted to proving   \eqref{smallo}. We point out that the strategy in Section \ref{sec411} can not be directly used, because  $\varphi$ is not uniformly bounded here.

\begin{proof}[Proof of Claim \eqref{smallo}] We begin by writing (with  $s_0 =\tilde{s}_{\sigma(0)} = t$ and $B_R=\{ x: |x| \leq R\}$)
\begin{align*}
&\quad \text{Var}\big(    \mathbf{J}_{p,R}(t) \big) = p! \int_{B_R^2}dxdx' \big\langle \widetilde{f}_{t,x,p},    \widetilde{f}_{t,x',p}  \big\rangle_{\mathcal{H}^{\otimes p}}=p! \int_{B_R^2}dxdx' \big\langle f_{t,x,p},    \widetilde{f}_{t,x',p}  \big\rangle_{\mathcal{H}^{\otimes p}} \\
&= c_{d,\beta}^p \sum_{\sigma\in\mathfrak{S}_p}  \int_{B_R^2}dxdx'  \int_{[0,t]^{2p}} d\pmb{s_p}d\pmb{\tilde{s}_p} \prod_{k=1}^p \gamma_0(s_k-  \tilde{s}_{k} ) \int_{\bR^{pd}} \left(\prod_{j=1}^p d\xi_j | \xi_j|^{\beta-d} \right) \\
&\quad\times e^{-i (x-x')\cdot (\xi_p+\cdots+\xi_1)} \prod_{j=0}^{p-1} \widehat{G}_{s_j - s_{j+1}}(\xi_p + \cdots+ \xi_{j+1} )  \widehat{G}_{\tilde{s}_{\sigma(j)} - \tilde{s}_{\sigma(j+1)}}(\xi_{\sigma(p)} + \cdots + \xi_{\sigma(j+1)} ),
\end{align*}
where we recall the convention that $G_t(z)=0$ for $t\leq 0$.
 Then, recalling definition \eqref{LR} of $\ell_{R}(\xi)$,
we can apply Lemma \ref{lem_ab} with
\[
 \mu (d\pmb{\xi_p} ) = \varphi(\xi_1) \cdots \varphi(\xi_p) \ell_R(\xi_1 + \cdots + \xi_p) d\xi_1 \cdots d\xi_p
\]
to get $ \text{Var}\big(    \mathbf{J}_{p,R}(t) \big)$ bounded by
\begin{align}\label{lead2}
 c_{d,\beta}^p \Gamma_t^p  \int_{\Delta_p(t)} d\pmb{s_p}\int_{\bR^{pd}} \left(\prod_{j=1}^p d\xi_j | \xi_j|^{\beta-d} \right) \ell_R(\xi_1+\cdots+\xi_p)\prod_{j=0}^{p-1} \Big\vert \widehat{G}_{s_j - s_{j+1}}(\xi_p + \cdots + \xi_{j+1} ) \Big\vert^2.
\end{align}
Making change of variables
\[
{\rm(i)} ~\eta_j = \xi_p+\cdots + \xi_j~\text{with $\eta_{p+1}=0$  \quad (ii)}~ (x, x', \eta_1)\to (Rx, Rx', \eta_1 R^{-1}),
\]
we  obtain
\begin{align*}
  \text{Var}\big(    \mathbf{J}_{p,R}(t) \big)& \leq c_{d,\beta}^p \Gamma_t^p   \int_{\Delta_p(t)} d\pmb{s_p}\int_{\bR^{pd}} \left(\prod_{j=1}^p d\eta_j | \eta_j - \eta_{j+1}|^{\beta-d} \right) \\
& \qquad \times  \left( \int_{B_R^2}dxdx'  e^{-i (x-x')\cdot\eta_1}\right) \prod_{j=0}^{p-1} \Big\vert \widehat{G}_{s_j - s_{j+1}}(\eta_{j+1} ) \Big\vert^2\\
&=  c_{d,\beta}^p \Gamma_t^p R^{2d-\beta}   \int_{\Delta_p(t)} d\pmb{s_p}\int_{\bR^{pd}} d\eta_1 | \eta_1 - \eta_2 R|^{\beta-d} \left(\prod_{j=2}^p d\eta_j | \eta_j - \eta_{j+1}|^{\beta-d} \right) \\
& \qquad \times  \left( \int_{B_1^2}dxdx'  e^{-i (x-x')\cdot\eta_1}\right) \Big\vert \widehat{G}_{t-s_1}( \eta_1/R ) \Big\vert^2   \prod_{j=1}^{p-1} \Big\vert \widehat{G}_{s_j - s_{j+1}}(\eta_{j+1} ) \Big\vert^2 \\
&\leq t^2 c_{d,\beta}^{p-1} \Gamma_t^p R^{2d-\beta}   \int_{\Delta_p(t)} d\pmb{s_p}\int_{\bR^{pd-d}}   \left(\prod_{j=2}^p d\eta_j | \eta_j - \eta_{j+1}|^{\beta-d} \right) \\
& \qquad \times  \left( \int_{B_1^2}dxdx'  |x-x'|^{-\beta} e^{-i (x-x')\cdot\eta_2 R}\right)     \prod_{j=1}^{p-1} \Big\vert \widehat{G}_{s_j - s_{j+1}}(\eta_{j+1} ) \Big\vert^2,
\end{align*}
where in the last inequality we used $| \widehat{G}_t | \leq t$ and  the following Fourier transform:
\begin{align*}  
 \int_{B_1^2}dxdx'  c_{d,\beta} \int_{\bR^d} d\eta_1 | \eta_1 - \eta_2 R |^{\beta-d}  e^{-i (x-x')\cdot\eta_1}
 &=c_{d,\beta} \int_{\bR^d} d\eta_1 | \eta_1 - \eta_2 R |^{\beta-d}  \big\vert \cF \mathbf{1}_{B_1} \big\vert^2(\eta_1)    \\
 &=\int_{B_1^2}dxdx'   | x- x' |^{-\beta} e^{-i (x-x')\cdot \eta_2 R}.
\end{align*}
Note that the integral $\int_{B_1^2}dxdx'  |x-x'|^{-\beta} e^{-i (x-x')\cdot\eta_2 R}$ is uniformly bounded by $\kappa_{\beta,d}$ and  it converges to zero
as $R\to\infty$ for $\eta_2\neq 0$.  This convergence is a consequence of the Riemann-Lebesgue's lemma.
 Taking into account
  the definition
   \eqref{def-Qp} of $Q_{p-1}$, then we have
   \[
  R^{\beta-2d}  \text{Var}\big(    \mathbf{J}_{p,R}(t) \big) \leq t^2   \kappa_{\beta,d} \Gamma_t^p Q_{p-1},
    \]
    which is summable over $p\geq 2$ by the arguments in the previous section. Hence by the dominated convergence theorem, we get
\[
R^{\beta-2d} \sum_{p\geq 2}  \text{Var}\big(    \mathbf{J}_{p,R}(t) \big) \xrightarrow{R\to\infty} 0.
\]
This proves the claim \eqref{smallo}. \qedhere

\end{proof}

 \subsubsection{Proof of part (3) in Proposition \ref{PROP_COV}}      \label{sec413}
Recall the two cases from \eqref{cases413}:
\begin{align*}
\begin{cases}
(c_1) & \gamma_i(x_i) = |x_i|^{-\beta_i}~\text{for some $\beta_i\in(0,1)$, $i=1,2$,} \\
(c_2) & \gamma_1\in L^1(\bR) ~{\rm and}~ \gamma_2(x) = |x|^{-\beta}~\text{for some $\beta\in(0,1)$.}
\end{cases}.
\end{align*}

\medskip

 In $(c_1)$,  the spectral density  is $\varphi(\xi_1, \xi_2) =c_{1,\beta_1}c_{1,\beta_2} | \xi_1|^{\beta_1-1} | \xi_2|^{\beta_2-1} $
 for $(\xi_1,\xi_2)\in\bR^2$, where $c_{1,\beta}$ is a constant that only depends on $\beta$. Now, using the notation from Section \ref{sec412}, we write
 \begin{align*}
&\text{Var}\big( \mathbf{J}_{1,R}(t) \big) =  \int_{B_R^2} dxdx'  \int_{[0,t]^2} dsds' \gamma_0(s-s') \int_{\bR^d} d\xi \varphi(\xi) e^{-i (x-x') \cdot \xi} \widehat{G}_{t-s}(\xi)  \widehat{G}_{t-s'}(\xi) \quad   \text{see \eqref{oneBeta}} \\
&=R^{4-\beta_1 -\beta_2}   \int_{[0,t]^2} dsds' \gamma_0(s-s') \int_{\bR^d} d\xi \varphi(\xi_1, \xi_2)  \int_{B_1^2} dxdx'  e^{-i (x-x') \cdot \xi} \widehat{G}_{t-s}(\xi/R)  \widehat{G}_{t-s'}(\xi/R),
    \end{align*}
where the last equality is obtained by  the change of variables  $(x,x', \xi_1, \xi_2)$ to $(Rx,Rx', \xi_1/R, \xi_2/R)$. Thus, by the exactly same arguments that lead to  \eqref{Riesz:p=1}, we can get
\[
\frac{\text{Var}\big( \mathbf{J}_{1,R}(t) \big)}{R^{4-\beta_1 - \beta_2}} \xrightarrow{R\to\infty}  K_{\beta_1, \beta_2}\int_{[0,t]^2} dsds' \gamma_0(s-s') ss',
\]
with  $K_{\beta_1, \beta_2} $ introduced in   \eqref{Kbeta12}. Similar to   \eqref{Riesz:p=11}, we also have
\begin{align}
\frac{\bE\big[  \mathbf{J}_{1,R}(t)  \mathbf{J}_{1,R}(s) \big]}{R^{4-\beta_1 - \beta_2}} \xrightarrow{R\to\infty} K_{\beta_1, \beta_2}\int_{0}^t dr \int_0^s dr' \gamma_0(r-r')  (t-r)(s-r'). \label{r2_1st}
\end{align}
   To obtain the result $(r_1)$, it remains to show
   \begin{align}
 \sum_{p\geq 2}  \text{Var}\big( \mathbf{J}_{p,R}(t) \big) = o\big( R^{4-\beta_1 - \beta_2} \big). \label{smallo1}
   \end{align}
  Its proof can be done \emph{verbatim} as for the result \eqref{smallo}, so  we omit the details here.

  \medskip

   Finally, let us look at the more interesting   case $(c_2)$ where $\gamma_1\in L^1(\bR)$ and $\gamma_2(x) =|x|^{-\beta}$ for   some fixed $\beta\in(0,1)$. In this case, the corresponding spectral density  is $\varphi(\xi_1,\xi_2) = \varphi_1(\xi_1) \varphi_2(\xi_2)$, where
   \begin{align}  \label{i-ii}
  \begin{cases}
  {\rm (i)}  & \text{ $\gamma_1 = \cF\varphi_1$ and $\varphi_1$ is uniformly continuous and bounded,     } \\
 {\rm (ii)}  & \text{ $\varphi_2(\xi_2) =c_{1,\beta} |\xi_2|^{\beta-1} $ for some constant $c_{1,\beta}$ that only depends on $\beta$.}
  \end{cases}
  \end{align}
Let us begin with \eqref{oneBeta} and make the usual change of variables $(x,x', \xi)\to (Rx,Rx', \xi/R)$ to obtain
  \begin{align*}
 & \text{Var}\big( \mathbf{J}_{1,R}(t) \big) =  \int_{B_R^2} dxdx'  \int_{[0,t]^2} dsds' \gamma_0(s-s') \int_{\bR^2} d\xi \varphi_1(\xi_1)\varphi_2(\xi_2) e^{-i (x-x') \cdot \xi} \widehat{G}_{t-s}(\xi)  \widehat{G}_{t-s'}(\xi) \\
  &=R^{3-\beta}  \int_{[0,t]^2} dsds' \gamma_0(s-s') \int_{\bR^2} d\xi \varphi_1(\xi_1/R)\varphi_2(\xi_2) \left( \int_{B_1^2} dxdx'  e^{-i (x-x') \cdot \xi} \right) \widehat{G}_{t-s}(\xi/R)  \widehat{G}_{t-s'}(\xi/R)\\
  &=R^{3-\beta}  \int_{[0,t]^2} dsds' \gamma_0(s-s') \int_{\bR^2} d\xi \varphi_1(\xi_1/R)\varphi_2(\xi_2) \big\vert \cF \mathbf{1}_{B_1} \big\vert^2(\xi) \widehat{G}_{t-s}(\xi/R)  \widehat{G}_{t-s'}(\xi/R).
  \end{align*}
Recall that $\varphi_1 $, $  \widehat{G}_{t-s} $ and $ \widehat{G}_{t-s'} $ are uniformly bounded and continuous. Note that,
applying Plancherel's theorem and the Parseval-type relation \eqref{parseval}, we have
\begin{align*}
\int_{\bR^2} d\xi  \varphi_2(\xi_2) \big\vert \cF \mathbf{1}_{B_1} \big\vert^2(\xi)  &=
2\pi \int_{\R^2}  dx_1 d\xi_2  \varphi_2(\xi_2)  \left| \cF\mathbf{1}_{B_1} (x_1, \bullet)(\xi_2) \right|^2  \\
&= 2\pi \int_{\bR^3} dx_1dx_2 dx_3 \mathbf{1}_{\{  x_1^2 + x_2^2\leq 1 \}} \mathbf{1}_{\{  x_1^2 + x_3^2\leq 1 \}} |x_2-x_3|^{-\beta}<\infty.
\end{align*}
  Therefore,
 by the dominated convergence theorem and  the fact that $\varphi_1(0) = \frac{1}{2\pi}\gamma_1(\bR)$, we get
 \begin{align*}
  \frac{ \text{Var}\big( \mathbf{J}_{1,R}(t) \big) }{R^{3-\beta}}\xrightarrow{R\to\infty}  & \varphi_1(0) \int_{[0,t]^2} dsds' \gamma_0(s-s') (t-s) (t-s') \int_{\bR^2} d\xi \varphi_2(\xi_2) \big\vert \cF \mathbf{1}_{B_1} \big\vert^2(\xi) \\
  = & \gamma_1(\bR) \mathcal{L}_\beta \int_{[0,t]^2} dsds' \gamma_0(s-s')ss',
  \end{align*}
  where $\mathcal{L}_\beta$ is defined in \eqref{def_L1B}.
 In the same way, we get for $s,t\in(0,\infty)$,
  \begin{align} \label{r3_1st}
  \frac{ \bE\big[ \mathbf{J}_{1,R}(t)  \mathbf{J}_{1,R}(s) \big] }{R^{3-\beta}}\xrightarrow{R\to \infty}  \gamma_1(\bR) \mathcal{L}_\beta   \int_0^t dr \int_0^s dr' \gamma_0(r-r')  (t-r)(s-r').
\end{align}
  Now we claim that the other chaoses are negligible, that is, as $R\to\infty$,
  \begin{align}\label{smallo2}
   \sum_{p\geq 2}  \text{Var}\big( \mathbf{J}_{p,R}(t) \big) = o(R^{3-\beta}).
  \end{align}
Note that the desired limiting covariance structure follows from \eqref{r3_1st} and the above claim \eqref{smallo2}.   The rest of this section is devoted to proving claim \eqref{smallo2}.

\begin{proof}[Proof of Claim \eqref{smallo2}] By the same  arguments that lead to the estimate \eqref{lead2}, we can obtain
\begin{align*}
 \text{Var}\big( \mathbf{J}_{p,R}(t) )\leq   \Gamma_t^p  \int_{\Delta_p(t)} d\pmb{s_p}\int_{\bR^{2p}} d\pmb{\xi_p} \varphi_p(\pmb{\xi_p})  \prod_{j=0}^{p-1} \Big\vert \widehat{G}_{s_j - s_{j+1}}(\xi_p + \cdots+ \xi_{j+1} ) \Big\vert^2 ~\text{with $s_0=t$},
\end{align*}
where $\varphi_p(\pmb{\xi_p}) = \varphi(\xi_1)\cdots \varphi(\xi_p) \ell_R(\xi_1+ \cdots+ \xi_p)$ for $\xi_j = (\xi_{j}^{(1)}, \xi_{j}^{(2)})\in\bR^2$, $j=1,\dots,p$ and $\ell_R$ is defined in \eqref{LR}. Recall that in the current case, $\varphi(\xi) = \varphi_1(\xi^{(1)}) \varphi_2( \xi^{(2)} )$ for $\xi=(\xi^{(1)},\xi^{(2)} )\in\bR^2$ and  $\varphi_1, \varphi_2$ satisfy the conditions in \eqref{i-ii}. Then, the following change of variables
\begin{center}
$\eta_j = \xi_j + \xi_{j+1} + \cdots+ \xi_p$ with $\eta_{p+1}=0$
\end{center}
yields
\begin{align*}
 \text{Var}\big( \mathbf{J}_{p,R}(t) ) & \leq   \Gamma_t^p  \int_{\Delta_p(t)} d\pmb{s_p}\int_{\bR^{2p}} d\pmb{\eta_p} \ell_R(\eta_1)  \prod_{j=0}^{p-1} \varphi(\eta_{j+1} - \eta_{j+2} ) \Big\vert \widehat{G}_{s_j - s_{j+1}}(\eta_{j+1} ) \Big\vert^2.
\end{align*}
In view of \eqref{LR}, we have $\ell_R(\eta_1/R) = R^4 \ell_1(\eta_1)$. Thus, by changing $\eta_1$ to $\eta_1/R$, we write
\begin{align*}
 \text{Var}\big( \mathbf{J}_{p,R}(t) ) & \leq  R^2 \Gamma_t^p  \int_{\Delta_p(t)} d\pmb{s_p}\int_{\bR^{2p}} d\pmb{\eta_p} \ell_1(\eta_1)   \varphi(\eta_1R^{-1} - \eta_2) \Big\vert \widehat{G}_{t - s_{1}}(\eta_{1}/R ) \Big\vert^2\\
 &\qquad\times      \prod_{j=1}^{p-1} \varphi(\eta_{j+1} - \eta_{j+2} ) \Big\vert \widehat{G}_{s_j - s_{j+1}}(\eta_{j+1} ) \Big\vert^2 \\
 &\leq R^{3-\beta} \Gamma_t^p  \| \varphi_1 \|_\infty t^2  \int_{\Delta_p(t)} d\pmb{s_p}\int_{\bR^{2p-2}} d\eta_2 ... d\eta_p  \left( \int_{\bR^2} d\eta_1 \ell_1(\eta_1)    c_{1,\beta}  \big\vert  \eta^{(2)}_1 - \eta_2^{(2)}R \big\vert^{\beta-1} \right)\\
 &\qquad\times      \prod_{j=1}^{p-1} \varphi(\eta_{j+1} - \eta_{j+2} ) \Big\vert \widehat{G}_{s_j - s_{j+1}}(\eta_{j+1} ) \Big\vert^2,
\end{align*}
where we used $ \vert \widehat{G}_{t - s_{1}}(\eta_{1}/R )  \vert^2 \leq t^2$. Observe that with $\eta= (\eta^{(1)}, \eta^{(2)})$, we deduce from  the fact  $\ell_1(\eta) = \big\vert \cF \mathbf{1}_{B_1}\big\vert^2(\eta^{(1)}, \eta^{(2)})$ that
\begin{align*}
\int_{\bR^2} d\eta \ell_1(\eta) \varphi_2(\eta^{(2)} - x R) &= \int_{\bR^2} d\eta^{(1)}d \eta^{(2)}  \big\vert \cF \mathbf{1}_{B_1}\big\vert^2(\eta^{(1)}, \eta^{(2)}+xR) \varphi_2(\eta^{(2)} ) \\
&= 2\pi \int_{\bR^3} \mathbf{1}_{\{ x_1^2+x_2^2\leq 1  \}}\mathbf{1}_{\{ x_1^2+x_3^2\leq 1  \}} e^{-i(x_2-x_3) xR} |x_2- x_3|^{-\beta}dx_1dx_2dx_3,
\end{align*}
by inverting the Fourier transform. The above quantity is uniformly bounded by $2\pi\mathcal{L}_\beta$ with $\mathcal{L}_\beta$ given in \eqref{def_L1B} and convergent to zero as $R\to\infty$ for every $x\neq 0$ in view of the Riemann-Lebesgue lemma.  Thus, $R^{\beta-3}  \text{Var}\big( \mathbf{J}_{p,R}(t) )$ is uniformly bounded by $2\pi \mathcal{L}_\beta \Gamma_t^p  \| \varphi_1 \|_\infty t^2 Q_{p-1}$, with $Q_{p-1}$ given by  \eqref{def-Qp} and it converges  to zero  as $R\to\infty$. Since $Q_{p} \leq C^p/p!$, we have
\[
\sum_{p\geq 2} \Gamma_t^p   Q_{p-1} <\infty,
\]
and the dominated convergence theorem implies \eqref{smallo2}.
\end{proof}

\begin{remark} Under the assumptions of Proposition  \ref{PROP_COV}, we point out that $\sigma_R(t) > 0$ for large enough $R$ so that the renormalized random variable $F_R(t)/ \sigma_R(t)$ is well-defined for large $R$.

\end{remark}

\subsection{Quantitative central limit theorems (QCLT) and f.d.d. convergence}  \label{sec42}

In this section, we prove the quantitative CLTs that are stated in Theorem \ref{MR2} and, as an easy consequence, we are also able to show the   convergence of finite-dimensional distributions in Theorem \ref{MR2}.
We consider first the part (1) and later we treat parts (2) and (3).

\subsubsection{Part (1)}
We will first show the estimate
\begin{align}\label{OPT1}
d_{\rm TV}\big( F_R(t)/\sigma_R(t), Z\big) \lesssim R^{-d/2},
\end{align}
where $Z \sim N(0,1)$.   By Proposition \ref{2nd-tool} applied to  $\frac 1{ \sigma_R(t)} F_R(t)$, we have
\begin{equation} \label{OPT1a}
d_{\rm TV}\big( F_R(t)/\sigma_R(t), Z\big)
\le  \frac 4  { \sigma^2_R(t)} \sqrt{\mathcal{A}_R},
\end{equation}
where
\begin{align*}
\mathcal{A}_R&= \int_{\bR_+^6\times\bR^{6d}} drdr' dsds' d\theta d\theta' dzdz' dydy' dwdw' \gamma_0(\theta - \theta') \gamma_0(s-s') \gamma_0(r-r')\gamma(z-z') \gamma(w-w')   \\
&\quad\times     \gamma(y-y')  \| D_{r,z}D_{\theta,w}F_R(t) \|_4  \| D_{s,y}D_{\theta',w'}F_R(t) \|_4  \| D_{r',z'}F_R(t)\|_4 \| D_{s', y' }F_R(t) \|_4.
\end{align*}
Recall from Section \ref{sec411} that    $\sigma^2_R(t)   \sim R^d$. Therefore,  in order to show \eqref{OPT1}  it suffices to prove the estimate
\begin{equation} \label{EST1}
\mathcal{A}_R \lesssim R^d.
\end{equation}
  Using   Minkowski's inequality, we can write
\[
\| D_{r,z}D_{\theta,w}F_R(t) \| _4 = \left\| \int_{B_R}  D_{r,z}D_{\theta,w} u(t,x) dx \right\|_4 \leq  \int_{B_R}  \big\| D_{r,z}D_{\theta,w} u(t,x) \big\|_4  dx.
\]
Then, it follows from our fundamental estimates  in  Theorem \ref{MR1}
 that
\begin{align}\label{USE1}
\| D_{r,z}D_{\theta,w}F_R(t) \| _4 \lesssim  \int_{B_R}  \widetilde{f}_{t,x,2}(r,z, \theta, w) dx,
 \end{align}
 with
 \[
 \widetilde{f}_{t,x,2}(r,z,\theta,w) = \frac{1}{2} \left[ G_{t-r}(x-z) G_{r-\theta}(z-w)\mathbf{1}_{\{ r > \theta\}} + G_{t-\theta}(x-w) G_{\theta-r}(z-w)\mathbf{1}_{\{ r < \theta\}}\right];
 \]
and,  in the same way, we have
\begin{align}\label{USE2}
\| D_{r,z}F_R(t) \| _4 \lesssim  \int_{B_R}  G_{t-r}(x-z)dx,
 \end{align}
 where the implicit constants in \eqref{USE1}-\eqref{USE2} do not depend on $(R, r,z,\theta,w)$ and are increasing in $t$.  Now, plugging \eqref{USE1}-\eqref{USE2} into the expression of $\mathcal{A}_R$, we get
 \begin{align*}
&\mathcal{A}_R\lesssim  \int_{[0,t]^6\times\bR^{6d}} drdr'dsds' d\theta d\theta' dzdz' dydy' dwdw' \gamma_0(r-r') \gamma_0(s-s') \gamma_0(\theta - \theta')\gamma(z-z')  \gamma(w-w')      \\
& \quad \times  \gamma(y-y') \int_{B_R^4}   \widetilde{f}_{t,x_1,2}(r,z, \theta, w)  \widetilde{f}_{t,x_2,2}(s,y, \theta', w') G_{t-r'}(x_3- z') G_{t-s'}(x_4 - y') d\pmb{x_4} =: \sum_{j=1}^4\mathcal{A}_{R,j}.
  \end{align*}
The four terms $ \mathcal{A}_{R,1},  \dots, \mathcal{A}_{R,4}$ are defined according to  whether $r>\theta$ or $r<\theta$, and  whether $ s>\theta'$ or $s<\theta'$.  For example,  the term $\mathcal{A}_{R,1}$ corresponds to $r>\theta$ and $s>\theta'$:
\begin{align}
\mathcal{A}_{R,1}&= \frac 14 \int_{[0,t]^6\times\bR^{6d}} drdr'dsds' d\theta d\theta' dzdz' dydy' dwdw' \gamma_0(r-r') \gamma_0(s-s') \gamma_0(\theta - \theta')   \notag \\  \notag
& \quad   \times    \gamma(w-w')  \gamma(y-y')   \gamma(z-z')G_{r-\theta}(z-w)G_{s-\theta'}(y-w')   \\
& \quad   \times     \int_{B_R^4} d\pmb{x_4} G_{t-r}(x_1-z)   G_{t-s}(x_2-y)    G_{t-r'}(x_3- z') G_{t-s'}(x_4 - y'). \label{asB}
\end{align}
The term $\mathcal{A}_{R,2}$ corresponds to $r>\theta$ and $s<\theta'$, the term $\mathcal{A}_{R,3}$ corresponds to $r<\theta$ and $s>\theta'$ and  the term $\mathcal{A}_{R,4}$  corresponds to $r< \theta$ and $s<\theta'$.
In the following, we estimate $\mathcal{A}_{R,j}$  for $j=1,2,3,4$ by a constant times $R^{d}$, which yields \eqref{EST1}.

  To get the bound for $\mathcal{A}_{R,1}$,  it suffices to perform the integration with respect to $dx_1, dx_2, dx_4$, $dy', dy, dw', dw$, $dz, dz', dx_3$ one by one, by taking into account the following facts:
 \[
\sup_{z\in\bR^d} \int_{B_R} G_{t-r}(x-z)dx \leq t \quad {\rm and}\quad \sup_{y'\in\bR^d}\int_{\bR^d} \gamma(y-y') dy  = \| \gamma\|_{L^1(\bR^d)}.
 \]
To get the bound for $\mathcal{A}_{R,2}$,  it suffices to perform the integration with respect to $dx_1, dx_3,dz', dz$, $dx_2, dw, dw', dy, dy', dx_4$.
  To get the bound for $\mathcal{A}_{R,3}$,  it suffices to perform the integration with respect to $dx_4, dy', dx_2, dy, dw', dx_1, dw, dz, dz', dx_3$ one by one.
   To get the bound for $\mathcal{A}_{R,4}$,  it suffices to perform the integration with respect to $dx_1, dx_3, dx_2, dz', dz, dw, dw', dy, dy', dx_4$ one by one.
This completes the proof of \eqref{OPT1}.

\medskip
In the second part of this subsection, we show the  f.d.d. convergence in Theorem \ref{MR2}-(1).

 Fix an integer $m\geq 1$ and choose $t_1, \dots, t_m\in(0,\infty)$. Put $\mathbf{F}_R = \big(  F_R(t_1), \dots , F_R(t_m)   \big)$. Then, by the result on limiting covariance structure from Section \ref{sec411}, we have that the covariance matrix of $R^{-d/2}\mathbf{F}_R$, denoted by $\mathcal{C}_R$,  converges to the matrix $\mathcal{C} = (\mathcal{C}_{ij}: 1\leq i,j \leq m)$, with
\[
\mathcal{C}_{ij}= \omega_d \sum_{p\geq 1} p! \int_{\bR^d} \big\langle \widetilde{f}_{t_i,x,p},   \widetilde{f}_{t_j,0,p} \big\rangle_{\mathcal{H}^{\otimes p}}dx.
\]
Since $F_R(t)=\delta(-DL^{-1}F_R(t))$, according to \cite[Theorem 6.1.2]{blue}\footnote{Note that there is a typo in Theorem 6.1.2 of \cite{blue}: In (6.1.3) of \cite{blue}, one has $d/2$ instead of $1/2$.}, for any twice differentiable function $h: \bR^m \to \bR$ with bounded second partial derivatives,
 \begin{align}
&\Big\vert  \bE\big[ h(R^{-d/2}\mathbf{F}_R) - h(\mathbf{Z}) \big] \Big\vert \leq \Big\vert  \bE\big[ h(R^{-d/2}\mathbf{F}_R) - h(\mathbf{Z}_R) \big] \Big\vert + \Big\vert  \bE\big[ h(\mathbf{Z}) - h(\mathbf{Z}_R) \big] \Big\vert  \notag \\
& \leq \frac{m}{2R^d} \| h''\|_\infty \sqrt{ \sum_{i,j=1}^m  {\rm Var}\Big( \big\langle DF_R(t_i), - DL^{-1}F_R(t_j) \big\rangle_\mathcal{H} \Big)      } + \Big\vert  \bE\big[ h(\mathbf{Z}) - h(\mathbf{Z}_R) \big] \Big\vert  , \label{ZZR}
 \end{align}
 with $\mathbf{Z}_R\sim N\big(0, \mathcal{C}_R \big)$,  $\mathbf{Z}\sim N\big(0, \mathcal{C} \big)$  and $\| h'' \|_\infty = \sup\big\{  \big\vert \frac{\partial^2}{\partial x_i \partial x_j} h(x)  \big\vert :  x\in\mathbb{R}^m,               i,j=1, \dots, m\big\}$. It is clear that the second term in \eqref{ZZR} tends to zero as $R\to\infty$. For the variance term in \eqref{ZZR},
 taking advantage of Proposition \ref{propAV} applied to $F=F_R(t_i)$ and $G=F_R(t_j)$ and using arguments analogous to those  employed to derive \eqref{EST1},
we obtain
 \[
  {\rm Var}\Big( \big\langle DF_R(t_i), - DL^{-1}F_R(t_j) \big\rangle_\mathcal{H} \Big)   \lesssim R^d.
 \]
Thus, the first term in \eqref{ZZR} is $O(R^{-d/2})$, implying that $ \bE\big[ h(R^{-d/2}\mathbf{F}_R) - h(\mathbf{Z}) \big]$ converges to zero as $R\to\infty$. This shows the convergence of the finite-dimensional distributions of $\{ R^{-d/2} F_R(t): t\in\bR_+ \}$  to those of the centered Gaussian process $\mathcal{G}$, whose covariance structure is given by
 \[
 \bE\big[ \mathcal{G}(t) \mathcal{G}(s) \big] = \omega_d \sum_{p\geq 1} p! \int_{\bR^d} \big\langle \widetilde{f}_{t,x,p},   \widetilde{f}_{s,0,p} \big\rangle_{\mathcal{H}^{\otimes p}}dx, \; \text{for $s,t\in[0,\infty)$}.
 \]
 This concludes the proof  of part (1) in Theorem \ref{MR2}. \hfill $\square$

\subsubsection{Proofs in parts (2) and (3)}

 \label{sec422}
In part (2), in view of the dominance of the first chaos, we have already obtained in Section \ref{sec412} that the
finite-dimensional distributions of the
process $\big\{R^{-d+\frac{\beta}{2} } F_R(t): t\in\bR_+\big\}$ converge  to those of a centered Gaussian process $\{\mathcal{G}_\beta(t) \}_{ t\in\bR_+}$, whose covariance structure is given by  \eqref{COV:Gbeta}.  By the same reason, the  convergence of the finite-dimensional distributions  in part (3) follows from   \eqref{r2_1st},  \eqref{smallo1},  \eqref{r3_1st}  and \eqref{smallo2}.

In this section, we   show that:
\begin{align}\label{OPT2}
d_{\rm TV}\big( F_R(t)/\sigma_R(t), Z \big)\lesssim \begin{cases}
 R^{-\beta/2} & \text{in part (2)}, \\
 R^{-\frac{1}{2}(\beta_1+\beta_2)} & \text{in  part (3) case $(a')$},\\
  R^{-(1+\beta)/2}  & \text{in   part (3) case $(b')$,}
  \end{cases}
\end{align}
where $Z\sim N(0,1)$.
Taking into account \eqref{OPT1a} and
the variance estimates in Section \ref{sec412} and Section \ref{sec413},  in order to get \eqref{OPT2}  it suffices to show that, for $j\in\{1,2,3,4\}$ and for $R\geq t$,
 \begin{align}\label{need2S}
\mathcal{A}_{R,j}\lesssim
\begin{cases}
R^{4d-3\beta}  & \text{in part (2)}, \\
R^{8- 3(\beta_1+\beta_2)}  & \text{in  case $(a')$ of  part (3),}\\
R^{5-3\beta}  & \text{in   case $(b')$ of part (3).}
\end{cases}
\end{align}
    Since the total-variation distance is always bounded by one, the bound \eqref{OPT2} still holds for $R<t$ by choosing the implicit constant large enough.

\bigskip

The rest of this section is then devoted to proving  \eqref{need2S} for $R\geq t$ and for $j\in\{1,2,3,4\}$.

  \begin{proof}[Proof of \eqref{need2S}] Let us first consider the term $\mathcal{A}_{R,1}$, which can be expressed as
     \begin{align*}
\mathcal{A}_{R,1}&=  \int_{[0,t]^6} drdr'dsds' d\theta d\theta'   \gamma_0(r-r') \gamma_0(s-s') \gamma_0(\theta - \theta') \mathbf{S}_{1,R}.
  \end{align*}
with
\begin{align*}
\mathbf{S}_{1,R}:&=  \int_{\bR^{6d}}   dzdz' dydy' dwdw'    \gamma(w-w')    \gamma(y-y')  \gamma(z-z') \int_{B_R^4} d\pmb{x_4}   G_{t-r}(x_1-z)  \\
& \quad   \times G_{r-\theta}(z-w)   G_{t-s}(x_2-y) G_{s-\theta'}(y-w')  G_{t-r'}(x_3- z') G_{t-s'}(x_4 - y').
\end{align*}
From now on, when $d=2$, we  write $(w, w', y, y', z, z') =(w_1, w_2, w'_1, w'_2, y_1, y_2,  y'_1, y'_2,z_1, z_2,z'_1, z'_2) $ and then $dy = dy_1 dy_2$; note also that $x_1,\dots , x_4$ denote the dummy variables in $\bR^d$.
By making the following change of variables
\begin{align}\label{cova}
(z,z', y, y', w,w', x_1,x_2,x_3,x_4 ) \to R(z,z', y, y', w,w', x_1,x_2,x_3,x_4 )
\end{align}
and using the scaling property $G_{t}(Rz) =  R^{1-d} G_{tR^{-1}}(z)$ for $d\in\{1,2\}$, we get
\begin{align}
& \mathbf{S}_{1,R}=R^{6+4d}  \int_{[-2,2]^{6d}}   dzdz' dydy' dwdw'    \gamma(Rw-Rw')    \gamma(Ry-Ry')  \gamma(Rz-Rz') \int_{B_1^4} d\pmb{x_4}    \notag  \\
&  \quad   \times G_{\frac{t-r}{R}}(x_1-z)   G_{\frac{r-\theta}{R}}(z-w)   G_{\frac{t-s}{R}}(x_2-y) G_{\frac{s-\theta'}{R}}(y-w')  G_{\frac{t-r'}{R}}(x_3- z') G_{\frac{t-s'}{R}}(x_4 - y'). \label{SPI}
 \end{align}
Note that we  have  replaced the integral domain $\bR^{6d}$ by $[-2,2]^{6d}$ in \eqref{SPI} without changing the value of $\mathbf{S}_{1,R}$, because, for example, $x_1\in B_1$ and $|x_1-z| \leq (t-r)/R$ implies  $|z|\leq 1 + tR^{-1}\leq 2$ while
$|z-w| \leq (r-\theta)/R$ and $|x_1-z| \leq  (t-r)/R$ imply $|w|\leq (t-\theta)R^{-1}+1 \leq 2$.

In view of the expression of $\gamma$ in part (2) and part (3), we write, for $z\in\bR^d$ ($z=(z_1,z_2)\in\bR^2$ when $d=2$),
\begin{align*}
\gamma(Rz) =
\begin{cases}
R^{-\beta} \gamma(z) & \text{in part (2)} ,\\
R^{-\beta_1-\beta_2} \gamma(z)  & \text{in case $(a')$ of part (3)}, \\
R^{-\beta} \gamma_1(Rz_1)\gamma_2(z_2)  & \text{in case $(b')$ of part (3)},
\end{cases}
\end{align*}
and it is easy to see that
\begin{align*}
\sup_{z'\in [-2,2]^d} \int_{[-2,2]^d}\gamma(Rz-Rz')dz \leq  \begin{cases}
{\displaystyle R^{-\beta}\int_{[-4,4]^d}\gamma(z)dz <  \infty}  & \text{in part (2)},  \\
\quad\\
  {\displaystyle R^{-\beta_1-\beta_2}\int_{[-4,4]^d}\gamma(z)dz <  \infty } & \text{in case $(a')$ of part (3)},   \\
  \quad\\
{\displaystyle  R^{-\beta-1} \gamma_1(\bR)\int_{-4}^4\gamma_2(s)ds <  \infty}  & \text{in case $(b')$ of part (3)}.
 \end{cases}
\end{align*}
To ease the notation, we just rewrite the above estimates as
\begin{align}\label{note2}
\sup_{z'\in [-2,2]^d} \int_{[-2,2]^d}\gamma(Rz-Rz')dz  \lesssim R^{-\alpha}
\end{align}
with $\alpha= \beta$ in part (2), $\alpha=\beta_1+\beta_2 $ in case $(a')$ of part (3),  and $\alpha=1+\beta$ in case $(b')$ of part (3).

To estimate $\mathcal{A}_{R,1}$,  we can use   \eqref{note2} to perform  integration with respect to $dx_1, dx_2, dx_4$, $dy', dy, dw', dw$, $dz, dz', dx_3$ successively. More precisely,  performing the integration  with respect to $dx_1, dx_2, dx_4$ and using the fact
\begin{align}\label{useineq1}
\sup_{ (s,z')\in[0,t]\times\bR^d}  \int_{\bR^d}G_{s/R}(z-z') dz = t/R
\end{align}
  gives us
\begin{align*}
  \mathbf{S}_{1,R}&\leq R^{3+4d} t^3   \int_{[-2,2]^{6d}}   dzdz' dydy' dwdw'    \gamma(Rw-Rw')    \gamma(Ry-Ry')  \gamma(Rz-Rz') \int_{B_1} dx_3       \\
&  \qquad   \times    G_{\frac{r-\theta}{R}}(z-w)    G_{\frac{s-\theta'}{R}}(y-w')  G_{\frac{t-r'}{R}}(x_3- z') \\
&\lesssim R^{3+4d}   R^{-\alpha}   \int_{[-2,2]^{5d}}   dzdz' dy dwdw'    \gamma(Rw-Rw')       \gamma(Rz-Rz') \int_{B_1} dx_3       \\
&  \quad   \times    G_{\frac{r-\theta}{R}}(z-w)    G_{\frac{s-\theta'}{R}}(y-w')  G_{\frac{t-r'}{R}}(x_3- z')  \quad \text{by integrating out $dy'$ and using \eqref{note2}} \\
&\lesssim R^{2+4d -\alpha }     \int_{[-2,2]^{4d}}   dzdz'  dwdw'    \gamma(Rw-Rw')       \gamma(Rz-Rz') \int_{B_1} dx_3       \\
&  \quad   \times    G_{\frac{r-\theta}{R}}(z-w)       G_{\frac{t-r'}{R}}(x_3- z')  \quad \text{by integrating out $dy$ and using \eqref{useineq1} } \\
&\lesssim R^{2+4d -2\alpha }     \int_{[-2,2]^{3d}}   dzdz'   dw     \gamma(Rz-Rz') \int_{B_1} dx_3      G_{\frac{r-\theta}{R}}(z-w)       G_{\frac{t-r'}{R}}(x_3- z')
\intertext{by integrating out $dw'$ and using \eqref{note2}; then, using \eqref{useineq1} to integrate out $dw$ }
&\lesssim R^{1+4d -2\alpha }     \int_{[-2,2]^{2d}}   dzdz'     \gamma(Rz-Rz') \int_{B_1} dx_3           G_{\frac{t-r'}{R}}(x_3- z')  \lesssim R^{4d -3\alpha }
 \end{align*}
where the last inequality is obtained by integrating out $dz, dz'$, $dx_3$ one by one and using \eqref{note2} and \eqref{useineq1}. The bound
\[
\mathbf{S}_{1,R} \lesssim R^{4d -3\alpha }  =
 \begin{cases}
R^{4d-3\beta} & \text{in part (2)},  \\
 R^{8-3\beta_1-3\beta_2} & \text{in cae $(a')$ of part (3)},   \\
  R^{5-3\beta} & \text{in cae $(b')$ of part (3)}
 \end{cases}
\] is uniform over $(r,r',s,s' ,\theta,\theta')\in[0,t]^6$, and hence we obtain \eqref{need2S} for $j=1$. For the other terms $\mathcal{A}_{R,2}, \mathcal{A}_{R,3}$ and $\mathcal{A}_{R,4}$, the arguments are the same: We  first go through the same change of variables \eqref{cova} to obtain terms $\mathbf{S}_{j, R}$ similar to $\mathbf{S}_{1, R}$ in \eqref{SPI},  and then use the facts \eqref{note2} and \eqref{useineq1} to perform one-by-one integration with respect to  the variables
\[
\begin{cases}
dx_1, dx_3,dz', dz, dx_2, dw, dw', dy, dy', dx_4 \quad\text{for estimating $\mathcal{A}_{R,2}$} \\
dx_4, dy', dx_2, dy, dw', dx_1, dw, dz, dz', dx_3\quad\text{for estimating $\mathcal{A}_{R,3}$} \\
dx_1, dx_3, dx_2, dz', dz, dw, dw', dy, dy', dx_4\quad\text{for estimating $\mathcal{A}_{R,4}$}
\end{cases}.
\]
This concludes the proof of \eqref{need2S} and hence completes the proof of \eqref{OPT2}. \qedhere

\end{proof}

\subsection{Tightness } \label{sec43}

This section is devoted to establishing the tightness in Theorem \ref{MR2}. This, together with the results in Section \ref{sec41} and Section \ref{sec42} will conclude the proof of Theorem \ref{MR2}. To get the tightness, we appeal to the criterion of   Kolmogorov-Chentsov (see \emph{e.g.} \cite[Corollary 16.9]{OK}). Put
\begin{align}
\sigma_R =
\begin{cases}
R^{d/2} &\text{in part (1) of Theorem \ref{MR2}} \\
R^{d-\frac{\beta}{2}} &\text{in part (2) of Theorem \ref{MR2}}  \\
R^{2-  \frac{1}{2}(\beta_1+\beta_2) } &\text{in part (3)-$(a')$ of Theorem \ref{MR2}}  \\
R^{(3-  \beta)/2} &\text{in part (3)-$(b')$ of Theorem \ref{MR2}}
\end{cases} \label{cases123}
\end{align}
and we will   show, for any fixed $T>0$, that the following inequality holds for any integer $k\geq 2$ and any $0 < s < t \leq T\leq R$:
\begin{align}\label{goal_tight}
\big\| F_R(t)- F_R(s) \big\|_k \lesssim  (t-s) \sigma_R,
\end{align}
where the implicit constant does not depend on $R, s$ or $t$. This moment estimate \eqref{goal_tight} ensures the tightness of $\big\{ \sigma_R^{-1} F_R(t): t\in[0,T]\big\}$ for any fixed $T>0$ and, therefore, the desired tightness on $\bR_+$
holds.

To show the above moment estimate \eqref{goal_tight} for  the increment $F_R(t)- F_R(s)$, we begin with the chaos expansion
\[
F_R(t) - F_R(s) = \sum_{n\geq 1} I_n\left(  \int_{B_R}dx [  f_{t,x,n} -  f_{s,x,n}]  \right) =  \sum_{n\geq 1} I_n\big( g_{n,R}\big),
\]
where $s, t$ are fixed, so we leave  them out of  the subscript of the kernel $g_{n,R}$ and
\begin{align}\label{gnR}
 g_{n,R}(\pmb{s_n} , \pmb{y_n}) = \Big[ \varphi_{t,R}(s_1,y_1)  -  \varphi_{s,R}(s_1,y_1) \Big] \prod_{j=1}^{n-1}G_{s_j- s_{j+1}}(y_j-y_{j+1})
\end{align}
with $\prod_{j=1}^0 =1$ and $\varphi_{t,R}(r,y) := \int_{B_R} G_{t-r}(x-y) dx$.
The rest of this section is then devoted to proving \eqref{goal_tight}.

\begin{proof}[Proof of \eqref{goal_tight}]
By the triangle inequality and using the moment estimate \eqref{hyper}, we get, for any $k\in [2,\infty)$,
\begin{align*}
\big\| F_R(t) - F_R(s) \big\|_k \leq  \sum_{n\geq 1} (k-1)^{n/2} \left\|  I_n\left( g_{n,R}  \right) \right\|_2.
\end{align*}
Note that  the kernel  $g_{n,R} =0$ outside $[0,t]^n \times\bR^{dn}$. Then, using  \eqref{miso} and \eqref{white-ineq}, we can write
\begin{align*}
\big\| F_R(t) - F_R(s) \big\|_k \leq  \sum_{n\geq 1}  \big[ \Gamma_t (k-1) \big]^{n/2}    \Big( n!  \|   \widetilde{g}_{n,R}  \|_{\cH_0^{\otimes n}}^2  \Big)^{1/2},
\end{align*}
where $ \widetilde{g}_{n,R}$ is the canonical symmetrization of $g_{n,R}$:
\[
 \widetilde{g}_{n,R}( \pmb{s_n} , \pmb{y_n}  ) = \frac{1}{n!} \sum_{\sigma\in\mathfrak{S}_n}   \Big[ \varphi_{t,R}(s_{\sigma(1)} ,y_{\sigma(1)})  -  \varphi_{s,R}(s_{\sigma(1)} ,y_{\sigma(1)}) \Big] \prod_{j=1}^{n-1}G_{s_{\sigma(j)}- s_{\sigma(j+1)}}(y_{\sigma(j)}-y_{\sigma(j+1)}).
\]
With the convention \eqref{rule1} in mind, we can write
\begin{align*}
 &n!  \|   \widetilde{g}_{n,R}  \|_{\cH_0^{\otimes n}}^2 = \int_{t>s_1>\cdots>s_n >0} d\pmb{s_n} \int_{\bR^{2nd}}  \Big[ \varphi_{t,R}(s_1,y_1)  -  \varphi_{s,R}(s_1,y_1) \Big] \left( \prod_{j=1}^{n-1}G_{s_j- s_{j+1}}(y_j-y_{j+1}) \right) \\
  &\qquad\qquad\qquad \times  \Big[ \varphi_{t,R}(s_1,y'_1)  -  \varphi_{s,R}(s_1,y'_1) \Big] \left( \prod_{j=1}^{n-1}G_{s_j- s_{j+1}}(y'_j-y'_{j+1}) \right) \prod_{j=1}^n \gamma(y_j - y_j') dy_j dy_j'.
 \end{align*}
Then,  using Fourier transform, we can rewrite  $n!  \|   \widetilde{g}_{n,R}  \|_{\cH_0^{\otimes n}}^2$ as follows:
\begin{align}
& n!  \|   \widetilde{g}_{n,R}  \|_{\cH_0^{\otimes n}}^2 =  \int_{t>s_1>\cdots>s_n >0} d\pmb{s_n} \int_{\bR^{nd}} \mu(d\pmb{\xi_p} )  \big\vert \cF \mathbf{1}_{B_R}\big\vert^2(\xi_1+\cdots+\xi_p)  \notag \\
& \qquad\qquad  \times  \big\vert  \widehat{G}_{t-t_1}(\xi_1 + \cdots + \xi_p)- \widehat{G}_{s-t_1}(\xi_1 + \cdots + \xi_p)  \big\vert^2  \prod_{j=1}^{n-1} \big\vert \widehat{G}_{s_j- s_{j+1}} \big\vert^2(\xi_{j+1} + \cdots + \xi_p)  . \label{lipFG}
 \end{align}
Recall the expression \eqref{FG}  $\widehat{G}_t(\xi)= \frac{\sin(t|\xi|)}{|\xi|}$ and note that it is a $1$-Lipschitz function in the variable $t$, uniformly over $\xi\in\bR^d$. Then
\[
\big\vert  \widehat{G}_{t-t_1}(\xi_1 + \cdots + \xi_p)- \widehat{G}_{s-t_1}(\xi_1 + \cdots + \xi_p)  \big\vert^2  \leq (t-s)^2.
\]
Therefore, plugging this inequality into \eqref{lipFG} and then applying Lemma \ref{lem_4Qp} yields
\begin{align*}
 n!  \|   \widetilde{g}_{n,R}  \|_{\cH_0^{\otimes n}}^2 &\leq (t-s)^2  \int_{t>s_1>\cdots>s_n >0} d\pmb{s_n}  \left(\int_{\bR^{d}} \mu(d\xi )  \big\vert \cF \mathbf{1}_{B_R}\big\vert^2(\xi)\right)      \prod_{j=1}^{n-1} \int_{\bR^d}\mu(d\xi_j) \big\vert \widehat{G}_{s_j- s_{j+1}} \big\vert^2(\xi_j)   \\
&  \leq (t-s)^2  \frac{t^n}{n!} \left(  2(t^2\vee 1)\int_{\bR^{d}}  \frac{\mu(d\xi )}{1+ |\xi|^2}   \right)^{n-1} \int_{\bR^{d}} \mu(d\xi )  \big\vert \cF \mathbf{1}_{B_R}\big\vert^2(\xi),
 \end{align*}
 which is finite since $ \mathbf{1}_{B_R}\in \cP_0 $. Using Fourier transform, we can write
 \begin{align*}
  \int_{\bR^{d}} \mu(d\xi )  \big\vert \cF \mathbf{1}_{B_R}\big\vert^2(\xi) = \int_{\bR^{2d}}  \mathbf{1}_{B_R}(x)  \mathbf{1}_{B_R}(y) \gamma(x-y)dxdy.
 \end{align*}
 Now let us consider the cases in \eqref{cases123}.

\medskip

\noindent In part (1) where $\gamma\in L^1(\bR^d)$,
\[
 \int_{\bR^{2d}}  \mathbf{1}_{B_R}(x)  \mathbf{1}_{B_R}(y) \gamma(x-y)dxdy \leq \gamma(\bR^d) \omega_d R^d \lesssim \sigma_R^2.
 \]

 \noindent In the other cases,
we can  make the change of variables $(x,y)\to R(x,y)$ to obtain
\begin{align*}
 \int_{\bR^{2d}}  \mathbf{1}_{B_R}(x)  \mathbf{1}_{B_R}(y) \gamma(x-y)dxdy   &=  R^{2d} \int_{\bR^{2d}}  \mathbf{1}_{B_1}(x)  \mathbf{1}_{B_1}(y) \gamma(Rx-Ry)dxdy \\
 &\lesssim  R^{2d-\alpha} = \sigma_R^2,
 \end{align*}
using  \eqref{note2}
with $\alpha= \beta$ in part (2), $\alpha=\beta_1+\beta_2 $ in case $(a')$,  and $\alpha=1+\beta$ in case $(b')$.

As a consequence, we get
\[
 n!  \|   \widetilde{g}_{n,R}  \|_{\cH_0^{\otimes n}}^2 \leq  \frac{C^n}{n!} \sigma_R^2 (t-s)^2,
 \]
 and therefore,
 \[
 \big\| F_R(t) - F_R(s) \big\|_k \leq  |t - s| \sigma_R \sum_{n\geq 1} \big[C \Gamma_t (k-1) \big]^{n/2}   \frac{1}{\sqrt{n!} },
 \]
 which leads to \eqref{goal_tight}.
\end{proof}

\section{Proof of Theorem \ref{MR3}} \label{sec5}

We argue as in the proof of Theorem 1.2 of \cite{BQS}.
As we explained in the introduction, it suffices to show
 that for each $m\ge 1$,
$$
\|D u(t,x)\|_{\cH}>0 \quad \mbox{a.s. on} \ \Omega_m,
$$
where $\Omega_m =\{ |u(t,x) | \ge   1/m\}$.

We claim that, almost surely, the function $(s,y) \mapsto D_{s,y}u(t,x)$ satisfies the assumptions of Lemma \ref{pos-norm}.  Indeed,
for $d=2$,  by Minkowski's inequality and the estimate \eqref{goalz}, we have
\begin{align*}
 \bE  \left(\int_0^t ds     \left(  \int_{\R^2}  | D_{s,y} u(t,x) |^{2q} dy \right)^{ 1/q}    \right)
& \leq  \int_0^t ds      \left(  \int_{\R^2} \Big|  \bE\big[  | D_{s,y} u(t,x) |^{2} \big] \Big|^q dy \right)^{ 1/q}   \\
& \leq C \int_0^t ds      \left(  \int_{\R^2}   G^{2q}_{t-s} (x-y)dy \right)^{ 1/q}  <\infty.
\end{align*}
For $d=1$, again by  the estimate \eqref{goalz},
\[
 \bE  \left(\int_0^t ds     \left(  \int_{\R}  | D_{s,y} u(t,x) |^{2} dy \right)    \right) \le
 C \int_0^t ds       \int_{\R}  G^{2}_{t-s} (x-y)dy  <\infty.
\]
Moreover, $(s,y) \mapsto D_{s,y}u(t,x)$  has compact support on $[0,t]\times B_M$ for some $M>0$.  As a consequence,  by Lemma \ref{pos-norm}, it suffices to prove that
\begin{equation}
\label{int2}
\int_0^t \|D_{r,\bullet}u(t,x)\|_0^2 dr = \int_0^t \int_{\bR^{2d}}D_{r,z}u(t,x) D_{r,z'}u(t,x)\gamma(z-z')dzdz'dr>0 ~ \mbox{a.s. on $\Omega_m$}.
\end{equation}

  As in the proof of Lemma 5.1 of \cite{BQS}, Corollaries \ref{D-norm} and \ref{D2-norm} allow us to infer that the $\cH\otimes \cP_0$-valued process $K^{(r)}$  defined by
\[
K^{(r)}(s,y,z)=G_{t-s}(x-y)D_{r,z}u(s,y)
\]
belongs to the space $\bD^{1,2}(\cH \otimes \cP_0)$. This is because, using Corollary \ref{D-norm},  we can write
\begin{align*}
 \bE\big(  \|K^{(r)}\|_{\cH \otimes \cP_0}^2 \big)    &
= \int_{[r,t]^2} \int_{\bR^{2d}} G_{t-s}(x-y)G_{t-s'}(x-y') \bE \Big(  \big\langle D_{r,\bullet}u(s,y), D_{r,\bullet} u(s',y') \big\rangle_0 \Big)\\
& \qquad \times  \gamma_0(s-s')  \gamma(y-y') dydy' dsds'\\
& \le C\int_{[r,t]^2} \int_{\bR^{2d}} G_{t-s}(x-y)G_{t-s'}(x-y')
  \gamma_0(s-s')  \gamma(y-y') dydy' dsds' <\infty,
\end{align*}
and in the same way, using Corollary \ref{D2-norm}  we can show that   $\bE\big(  \| DK^{(r)}\|_{ \cH \otimes \cH \otimes \cP_0}^2 \big) <\infty$.
 Therefore,  the process $K^{(r)}$ belongs to the domain of the  $\cP_0$-valued Skorokhod integral, denoted by $\overline{\delta}$.
Then,  using the same arguments as in the proof of Proposition 5.2 of \cite{BQS}, replacing $L^2(\bR)$ by $\cP_0$, we can show that
for any $r \in [0,t]$, the following equation holds in $L^2(\Omega;\cP_0)$:
\begin{equation}
\label{eq-Malliavin-P}
D_{r,\bullet}u(t,x)=G_{t-r}(x-\bullet)u(r,\bullet)+\int_r^t \int_{\bR^d}G_{t-s}(x-y)D_{r,\bullet}u(s,y)W(\overline{\delta} s, \overline{\delta}y).
\end{equation}

Let $\delta \in (0,t \wedge 1)$ be arbitrary.  Due to relation   \eqref{eq-Malliavin-P}
we have, almost surely,
\begin{align}
 \int_{0}^t \|D_{r, \bullet} u(t,x)\|^2_{0}\, dr & \geq \int_{t-\delta}^t \|D_{r,\bullet} u(t,x)\|^2_{0}\, dr   \geq  \frac{1}{2} \int_{t-\delta}^t \|G_{t-r}(x-\bullet) u(r,\bullet)\|^2_{0} \, dr - I(\delta),
 \label{eq:33}
\end{align} where
\begin{align*}
 I(\delta) & = \int_{t-\delta}^t \left\| \int_{r}^t \int_{\bR^d} G_{t-s}(x-y) D_{r, \bullet} u(s,y)
 W(\overline{\delta} s, \overline{\delta} y) \right\|^2_{0} \, dr\\
 & = \int_{t-\delta}^t \left\| \int_{t-\delta}^t \int_{\bR^d} G_{t-s}(x-y) D_{r, \bullet} u(s,y)
 W(\overline{\delta} s, \overline{\delta} y)\right\|^2_{0} \, dr.
\end{align*}
On the event $\Omega_m=\{ | u(t,x) | \geq 1/m\}$, we have
\begin{align*}
&   \int_{t-\delta}^t  \|G_{t-r}(x-\bullet) u(r,\bullet)\|_0^2 dr =
\int_{t-\delta}^t  \int_{\bR^{2d}} G_{t-r}(x-z)G_{t-r}(x-z') u(r,z) u(r,z')\gamma(z-z')dzdz' dr \\
&   = \int_{t-\delta}^t  \int_{\bR^{2d}}  G_{t-r}(x-z)G_{t-r}(x-z') u(t,x)^2 \gamma(z-z')dzdz' dr \\
&   \quad  -\int_{t-\delta}^t  \int_{\bR^{2d}}
G_{t-r}(x-z)G_{t-r}(x-z') \big[  u(t,x)^2-u(r,z)u(r,z')\big]
\gamma(z-z')dzdz' dr \\
&   \geq \frac{1}{m^2} \psi_0(\delta) - J(\delta),
\end{align*}
where
\begin{align*}
\psi_0(\delta)& :=\int_{t-\delta}^t \int_{\bR^{2d}}
G_{t-r}(x-z)G_{t-r}(x-z')\gamma(z-z') dzdz' dr\\
& = \int_0^\delta \int_{\bR^{2d}} G_r(z)G_r(z')\gamma(z-z')dz dz'dr
\end{align*}
and
\begin{align*}
 J(\delta) & := \int_{t-\delta}^t \int_{\bR^{2d}}   G_{t-r}(x-z)G_{t-r}(x-z')\gamma(z-z') \Big( u(t,x)^2
 - u(r,z)u(r,z')\Big) dz dz' dr.
\end{align*}
Coming back to \eqref{eq:33}, we can write
\begin{equation}
\label{LB-D}
\int_0^t \|D_{r,\bullet}u(t,x)\|_0^2  dr \geq \frac{1}{2 m^2}\psi_0(\delta)- \frac12
J(\delta) - I(\delta) \quad \mbox{on} \quad \Omega_m.
\end{equation}

We now give upper bounds for the first moments of $J(\delta)$ and $I(\delta)$.
We will use the following facts, which were proved in \cite{BS17}:
\begin{align*}
  C_t^*&:=\sup_{(s,y) \in[0,t] \times \bR^d}\|u(s,y)\|_2<\infty \qquad (\text{see also \eqref{calsoRem31} in Remark \ref{rem_Lp}} ) \\
g_{t,x}(\delta)& := \sup_{|t-s|<\delta} \sup_{|x-y|<\delta} \|u(t,x)-u(s,y)\|_2 \to 0 \quad \mbox{as} \ \delta \to 0.
\end{align*}
We first treat $J(\delta)$.
By Cauchy-Schwarz inequality, for any $r\in [0,t]$ and $z,z' \in \bR^2$,
\begin{align*}
\bE\big[ |u(t,x)^2-u(r,z)u(r,z')| \big] & \leq \|u(t,x)\|_2 \|u(t,x)-u(r,z)\|_2+ \|u(r,z)\|_2 \|u(t,x)-u(r,z')\|_2 \\
&\leq C_t^* \Big( \|u(t,x)-u(r,z)\|_2+ \|u(t,x)-u(r,z')\|_2  \Big).
\end{align*}
Since $G_{t-r}(x-z)$ contains the indicator of the set
$\{|x-z|<t-r\}$, we obtain:
\begin{align*}
\bE(|J(\delta)|) &\leq 2C_t^*\int_{t-\delta}^t \int_{\bR^{2d}}  G_{t-r}(x-z)G_{t-r}(x-z')\gamma(z-z')\|u(t,x)-u(r,z)\|_2dz dz' dr \\
&\leq 2C_t^* \int_{t-\delta}^t \int_{\bR^{2d}}
G_{t-r}(x-z)G_{t-r}(x-z')\gamma(z-z')  \sup_{ \substack{ t-\delta<s<t \\ |x-y|<\delta }}\|u(t,x)-u(s,y)\|_2dz dz' dr.
\end{align*}
 It follows that
\begin{equation}
\label{estimate-J}
\bE(|J(\delta)|)\leq 2 C_t^* g_{t,x}(\delta) \psi_0(\delta).
\end{equation}
Next, we treat $I(\delta)$. Applying  Proposition 6.2 of \cite{B12} to the $\cP_0$-valued process
\[
U(s,y)=\mathbf{1}_{[t-\delta,t]}(s)G_{t-s}(x-y)D_{r,\bullet}u(s,y)
\]
we obtain
\[
\bE(\|\overline{\delta}(U)\|_0^2) \leq \bE(\|U\|_{\cH \otimes \cP_0}^2)+\bE(\|DU\|_{\cH \otimes \cH \otimes \cP_0}^2).
\]
We have,
\begin{align*}
\bE(\|U\|_{\cH \otimes \cP_0}^2)& =\bE\Bigg( \int_{[t-\delta,t]^2}
\int_{\bR^{2d}}
 G_{t-s}(x-y)G_{t-s'}(x-y') \gamma_0(s-s')\gamma(y-y') \\
 & \qquad\qquad  \times  \big\langle D_{r,\bullet}u(s,y),D_{r,\bullet}u(s',y') \big\rangle_0 dydy' dsds' \Bigg)
 \end{align*}
 and
 \begin{align*}
&\bE(\|DU\|_{\cH \otimes \cH \otimes \cP_0}^2)=\bE \Bigg(\int_{[t-\delta,t]^2}
\int_{[0,r]^2} \int_{\bR^{4d}}G_{t-s}(x-y) G_{t-s'}(x-y') \gamma_0(s-s')\gamma(y-y') \\
& \qquad  \times  \big\langle D^2_{(\theta,w),(r,\bullet)}u(s,y),D_{(\theta',w'),(r,\bullet)}u(s',y')\big\rangle_0\,
\gamma_0(\theta-\theta')\gamma(w-w') dwdw' dydy' d\theta d\theta' dsds' \Bigg)\\
& \quad =\bE  \Bigg(\int_{[t-\delta,t]^2} \int_{\bR^{2d}}G_{t-s}(x-y) G_{t-s'}(x-y') \gamma_0(s-s')\gamma(y-y') \\
& \qquad \qquad    \times \big\langle   DD_{r,\bullet}u(s,y), DD_{r,\bullet}u(s',y')\big\rangle_{\cH \otimes \cP_0} dydy' ds ds'\Bigg).
\end{align*}
Hence,
$\bE(I(\delta)) \leq I_1(\delta)+I_2(\delta)$,  where
\begin{align*}
I_1(\delta)&:=\bE \Bigg( \int_{[t-\delta,t]^3} \int_{\bR^{2d}}
 G_{t-s}(x-y)G_{t-s'}(x-y') \gamma_0(s-s')\gamma(y-y') \\
 & \qquad  \times  \big\langle   D_{r,\bullet}u(s,y),D_{r,\bullet}u(s',y') \big\rangle_0 dydy' dsds'dr \Bigg)
 \end{align*}
 and
 \begin{align*}
I_2(\delta) &:=\bE\Bigg( \int_{[t-\delta,t]^3} \int_{\bR^{2d}}G_{t-s}(x-y) G_{t-s'}(x-y') \gamma_0(s-s')\gamma(y-y') \\
& \qquad \times  \langle  DD_{r,\bullet}  u(s,y), DD_{r,\bullet}u(s',y')\rangle_{\cH \otimes \cP_0} dydy' ds ds'dr \Bigg).
\end{align*}
Using  Cauchy-Schwarz inequality and Corollaries \ref{D-norm} and \ref{D2-norm}, we obtain:
\[
\bE\Big(   \big|\langle D_{r,\bullet}u(s,y),D_{r,\bullet}u(s',y')\rangle_0 \big|\Big)\leq C_t  \quad \mbox{and} \quad \bE\Big( \big|  \langle DD_{r,\bullet}u(s,y), DD_{r,\bullet}u(s',y')\rangle_{\cH \otimes \cP_0}\big| \Big)\leq C_t''.
\]
Hence,
\begin{equation}
\label{estimate-I}
\bE[I(\delta)] \leq (C_t+C_t'')\delta \phi(\delta),
\end{equation}
where
\begin{align}
\phi(\delta) :& =\int_{[t-\delta,t]^2}
\int_{\bR^{2d}}G_{t-s}(x-y)G_{t-s'}(x-y')\gamma_0(s-s')\gamma(y-y')dydy'ds
ds' \notag \\
& = \int_{[0,\delta]^2}\int_{\bR^{2d}}
G_s(y)G_{s'}(y')\gamma_0(s-s')\gamma(y-y')dydy'ds ds'. \label{eq:phi_del}
\end{align}

Using \eqref{LB-D}, \eqref{estimate-J} and \eqref{estimate-I}, we
conclude the proof as follows. For any $n\geq 1$,
\begin{align*}
& \quad \bP\left( \left\{\int_{0}^t \|D_{r,\bullet} u(t,x)\|_0^2 \, dr <\frac1n\right\} \cap
 \Omega_m \right) \leq
 \bP\left(I(\delta) + \frac{1}{2} J(\delta) > \frac{1}{2m^2}\psi_0(\delta) - \frac1n \right)\\
 &  \leq \left( \frac{1}{2m^2} \psi_0(\delta) - \frac1n\right)^{-1} \Big(\bE[I(\delta)] + \frac{1}{2} \bE[|J(\delta)|]\Big)
    \leq
\frac{  (C_t+C_t'') \delta \phi(\delta) + C_t^* g_{t,x}(\delta)\psi_0(\delta) }{ \frac{1}{2m^2} \psi_0(\delta) - \frac1n }.
\end{align*}
Letting $n\to \infty$, we obtain:
\[
  \bP\left( \left\{\int_{0}^t \|D_{r,\bullet} u(t,x)\|_0^2 dr =0 \right\} \cap \Omega_m \right) \leq 2m^2 \Big((C_t+C_t'') \delta \frac{\phi(\delta)}{\psi_0(\delta)} + C_t^* g_{t,x}(\delta)\Big).
\]
Note that using Fourier transform and  the expression \eqref{FG}, we can rewrite \eqref{eq:phi_del} as
\begin{align*}
\phi(\delta) &=  \int_{[0,\delta]^2}\int_{\bR^{d}} \widehat{G}_s(\xi) \widehat{G}_{s'}(\xi)\gamma_0(s-s') \mu(d\xi) ds ds' \\
&\leq   \int_{[0,\delta]^2}\int_{\bR^{d}} \frac{1}{2} \Big[  \widehat{G}_s(\xi)^2 + \widehat{G}_{s'}(\xi)^2 \Big]\gamma_0(s-s') \mu(d\xi) ds ds'  \leq \Gamma_\delta  \int_{[0,\delta]}\int_{\bR^{d}}   \widehat{G}_s(\xi)^2   \mu(d\xi) ds,
\end{align*}
where $\Gamma_{\delta}=2\int_0^{\delta}\gamma_0(s)ds$. That is, we have
$\phi(\delta) \leq
\Gamma_{\delta} \psi_0(\delta)$. Finally taking $\delta\to 0$ proves \eqref{int2}, since
$g_{t,x}(\delta)\to 0$ and $\delta \frac{\phi(\delta)}{\psi_0(\delta)} \leq \delta \Gamma_\delta\to 0$ as $\delta\to 0$.
\qed

\appendix

\section{Appendix}\label{appA}
\subsection{Auxiliary Results}

Let $d=2$ and assume Hypothesis ${\bf (H1)}$.
Suppose that $S: \R_+\times \R^2 \to \R$ is a measurable function such that $S\in L^{2}( \R_+; L^{2q} (\R^2))$, where $q$ is given in
\eqref{def-q} in cases (\texttt{a}) and (\texttt{b}) and it is given in \eqref{def-qq} in case (\texttt{c}).
We assume also that $S$   has  support in $[0,T]\times B_M$  for some $M>0$. We claim that $S$ belongs to $\cH$ and the following estimates hold true:
\[
 \| S \|_{\cH}  \le  \sqrt{\Gamma_T}   \| S \|_{\cH_0}
  \le  \sqrt{\Gamma_T D_\gamma } \| S\|_{ L^{2}( \R_+; L^{2q} (\R^2))}.
\]
 Indeed, the first inequality is due to \eqref{white-ineq} and the second one follows from   \eqref{q-ineq}.

For $d=1$, if     $S\in L^{2}( \R_+ \times \R)$    has  support in $[0,T]\times B_M$  for some $M>0$, then $S\in \cH$ and the following estimates hold true:
\[
 \| S \|_{\cH}  \le  \sqrt{\Gamma_T}   \| S \|_{\cH_0}
  \le  \sqrt{\Gamma_T  \| \gamma \mathbf{1}_{B_{2M}} \|_{L^1(\bR)} } \| S\|_{ L^{2}( \R_+\times \R)}.
\]
 Indeed, the first inequality is due to \eqref{white-ineq} and the second one follows from
 \begin{align*}
 \| S \|_{\cH_0}^2 & =  \int_0^T \int_{\R^2}  S(t,y) S(t,y') \gamma(y-y') dydy'
 dt \leq   \int_0^T \int_{\R^2}  \frac{ S^2(t,y)  + S^2(t,y') }{2} \gamma(y-y') dydy' dt
 \end{align*}
and
\[
\sup_{y'\in B_M}\int_{B_M} \gamma(y-y') dy \leq \int_{B_{2M}} \gamma(y) dy.
\]

\medskip

  Let us recall the Hypothesis  ${\bf (H2)}$: The measures  $\mu_0$ and $\mu$ such that $\gamma_0 =\mathcal{F}\mu_0$ and $\gamma = \mathcal{F} \mu$ are absolutely continuous with respect to the Lebesgue measures with strictly positive densities.

\begin{lemma}\label{pos-norm} Fix $d\in\{1,2\}$ and assume that the  Hypothesis    ${\bf (H2)}$  holds.   Let  the Hypothesis    ${\bf (H1)}$ hold  if in addition $d=2$.   Suppose that  the function $S: \bR_{+} \times \bR^d\to\bR$  has  support in $[0,T]\times B_M$  for some $M>0$ and $S\in L^{2}\big( \bR_+;  L^{2q}(\bR^d) \big)$,  where
\begin{align*}
\begin{cases}
\text{$q$ is given by \eqref{def-q} in cases {\rm (\texttt{a})} and {\rm(\texttt{b})} and  by \eqref{def-qq} in case {\rm(\texttt{c})} if $d=2$, } \\
\text{$q=1$ if $d=1$.}
\end{cases}
\end{align*}
 If
\begin{equation}
I:=\int_0^{T}\int_{\bR^d}\int_{\bR^d} S(t,x) S(t,y)\gamma(x-y)dxdydt>0,
\label{eq:1}
\end{equation}
then $\|S\|_{\cH}>0$.
\end{lemma}

\begin{proof}
 Suppose that $\|S\|_{\cH}=0$.  There exists a sequence of  smooth functions $(\psi_k)_{k\ge 1}$ in $C^\infty ( \bR_+\times \bR^d)$, with  support in    $[0,T]\times B_M$, which converges to $S$ in   $L^2(\bR_+; L^{2q}(\bR^d))$.  Then,
 \[
 0=\| S \|_{\cH}^2 = \lim_{k\rightarrow \infty} \| \psi_k \|^2_{\cH}
 =\lim_{k\rightarrow \infty}   \int_{\bR_+\times \bR^d} | \mathcal{F} \psi_k(\tau, \xi) |^2 \mu_0(d\tau) \mu(d\xi),
 \]
 where $\gamma_0 =\mathcal{F}\mu_0$, $\gamma = \mathcal{F} \mu$ and $\mathcal{F} \psi_k$ stands for the Fourier transform of $\psi_k$ in space-time variables  \emph{in this proof}.
 By choosing a subsequence $(k_j)_{j\geq 1}$ we have that
 \[
\lim_{j\rightarrow \infty} \mathcal{F} \psi_{k_j}(\tau, \xi) =0
\]
for   $\mu_0 \otimes \mu$-almost all $(\tau, \xi)$. On the other hand, keeping in mind that the  supports of  $S,\psi_k$ are contained in $[0,T]\times B_M$,
we have
\[
\big\| \psi_k - S\big\| _{ L^1(\bR_+\times\bR^2)} \leq (\pi M^2 T)^{1-\frac{1}{2q}}  \big\| \psi_k - S\big\| _{ L^{2q}(\bR_+\times\bR^2)}  \leq    (\pi M^2)^{1-\frac{1}{2q}}  T^\frac{1}{2} \big\| \psi_k - S\big\| _{ L^2(\bR_+; L^{2q}(\bR^2))},
\]
from which  we deduce that $(\psi_k)_{k\ge 1}$ converge  in $L^1([0,T] \times B_M)$ to $S$.  Thus   $\mathcal{F} \psi_k(\tau, \xi)$ converges to
$\mathcal{F} S(\tau,\xi)$ for all $(\tau,\xi)$  and the convergence is uniform. As a consequence,
  $\cF S(\tau,\xi)=0$ for $\mu_0 \otimes \mu$-almost all $(\tau,\xi) \in \bR_+ \times \bR^d$ and by Hypothesis ${\bf (H2)}$, we obtain
   $\cF S(\tau,\xi)=0$ for  almost all $(\tau,\xi) \in \bR_+ \times \bR^d$ with respect to the Lebesgue measure.

  Hence $S(t,x)=0$ for almost all $t>0$ and $x \in \bR^d$, i.e. there exists a Borel set $N \subset \bR_{+} \times \bR^d$ with $\lambda_{d+1}(N)=0$ such that $S(t,x)=0$ for all $(t,x) \not \in N$. Here $\lambda_{k}$ denotes the Lebesgue measure on $\bR^{k}$. Therefore,
$$I=\int_0^{\infty}\int_{\bR^d}\int_{\bR^d} 1_{A}(t,x,y) S(t,x) S(t,y)\gamma(x-y)dxdydt,$$
where $A:=\{(t,x,y) \in \bR_{+} \times \bR^d \times \bR^d; (t,x) \in N,(t,y) \in N\}$.

Let $N_t=\{x \in \bR^d; (t,x) \in N\}$ be the section of the set $N$ at point $t>0$. By Fubini's theorem,
$\lambda_{d+1}(N)=\int_{0}^{\infty}\lambda_d(N_t)dt$. Since $\lambda_{d+1}(N)=0$, we infer that $\lambda_d(N_t)=0$ for almost all $t>0$. Note that the section of the set $A$ at point $t$ is $A_t=\{(x,y) \in \bR^d \times \bR^d; (t,x,y) \in A\}=N_t \times N_t$, and its Lebesque measure is $\lambda_{2d}(A_t)=\lambda_{d}^2(N_t)=0$ for almost all $t>0$. By applying Fubini again, we infer that $\lambda_{2d+1}(A)=\int_0^{\infty}\lambda_{2d}(A_t)dt=0$.
This shows   $I=0$, which contradicts \eqref{eq:1}.
\end{proof}


\subsection{Proof of Proposition \ref{propAV}} \label{pfAV}

In this section, we only sketch the proof of Proposition \ref{propAV} as the main body of the proof is almost identical to that in \cite[Proposition 3.2]{Anna}.

\begin{proof}[Proof of \eqref{propA}]
Using the duality relation \eqref{IbP} and the identity $L = -\delta D$, we have
 \[
 \bE\big[  \langle DF, - DL^{-1}G \rangle_\mathcal{H} \big] =  \bE\big[  F (- \delta D)L^{-1}G  \big] = \bE[ F LL^{-1}G] =\bE[FG]= \text{Cov}(F,G),
 \]
  which shows the equality in \eqref{propA}.
  Then,   applying the Gaussian Poincar\'e inequality \eqref{GPI} and using Lemma 3.2  of \cite{NPR09}, we can bound the  variance appearing in the left-hand side of  \eqref{propA} by
\begin{align*}
   \bE\Big[ \| D \langle DF, - DL^{-1}G \rangle_\mathcal{H} \|_\mathcal{H}^2 \Big]  \leq 2 \bE\Big[ \|  \langle D^2F, - DL^{-1}G \rangle_\mathcal{H} \|_\mathcal{H}^2 \Big] + 2 \bE\Big[ \|  \langle DF, - D^2L^{-1}G \rangle_\mathcal{H} \|_\mathcal{H}^2 \Big].
\end{align*}
We will show that the first expectation-term is bounded by $A_1$ and
the other one can be estimated in the same way and bounded by $A_2$.
Using the representation (see \emph{e.g.} \cite[Proposition
2.9.3]{blue})
\[
-DL^{-1} G = \int_0^\infty dt e^{-t} P_t DG,
\]
with   $\{P_t, t\geq 0\}$   the Ornstein-Uhlenbeck semigroup,   we
can write
\begin{align}\label{dedfrom}
 \langle D^2F, - DL^{-1}G \rangle_\mathcal{H} = \int_0^\infty dt e^{-t} \langle D^2F, P_t DG \rangle_\mathcal{H}.
\end{align}
Note that if $(\mathcal{M}, \mathfrak{M}, \nu)$ is a probability space on which $s\in\mathcal{M}\longmapsto V_s\in |\cH|$ is $\mathfrak{M}$-measurable such that $\int_{\mathcal{M}} \big\| |V_s|\big\|_{\cH}^2 \nu(ds)<\infty$, then by Fubini's theorem and Cauchy-Schwarz inequality,
\begin{align*}
\left\| \int_{\mathcal{M}} V_s \nu(ds) \right\|_{\cH}^2 & =  \int_{\mathcal{M}^2}  \langle V_s, V_{s'}\rangle_{\cH} \nu(ds) \nu(ds')  \\
& \leq \int_{\mathcal{M}^2}   \frac{\| V_s\|^2_{\cH} +\| V_{s'}\|^2_{\cH} }{2}  \nu(ds) \nu(ds') =\int_{\mathcal{M}} \| V_s\|_{\cH}^2 \nu(ds).
\end{align*}
Using the above inequality   on $(\bR_+, e^{-t}dt)$, we deduce from \eqref{dedfrom} that
\begin{align*}
\big\|  \langle D^2F, - DL^{-1}G \rangle_\mathcal{H} \big\|_{\cH}^2 \leq  \int_0^\infty dt e^{-t} \big\| \langle D^2F, P_t DG \rangle_\mathcal{H} \big\|^2_{\cH}.
\end{align*}
Observe  that $  \langle D^2F, P_t DG \rangle_\mathcal{H}$ is nothing else but the one-contraction $D^2F\otimes_1 P_tDG$, so that
\begin{align*}
 \big\| \langle D^2F, P_t DG \rangle_\mathcal{H} \big\|^2_{\cH} &= \langle D^2F\otimes_1 P_tDG,  D^2F\otimes_1 P_tDG      \rangle_{\cH} \\
 &= \big\langle D^2F\otimes_1 D^2F,  (P_tDG)\otimes (P_tDG)     \big\rangle_{\cH^{\otimes 2}},
 \end{align*}
 where the last equality follows from the definition of contractions.
Therefore,  we have
 \begin{align*}
 &  \bE[ \|  \langle D^2F, - DL^{-1}G \rangle_\mathcal{H} \|_\mathcal{H}^2  ] \\
 &\quad \leq   \bE \int_0^\infty dt ~e^{-t}  \int_{\bR_+^6\times\bR^{6d}} drdr' dsds' d\theta d\theta' dzdz' dydy' dwdw' \gamma_0(\theta - \theta') \gamma_0(s-s') \gamma_0(r-r') \\
& \qquad\quad \times  \gamma(z-z') \gamma(w-w') \gamma(y-y') \times  \big[ D_{r,z}D_{\theta,w}F \big]  \big[  D_{s,y}D_{\theta',w'}F \big]   P_t (D_{r',z'}G)  P_t(D_{s', y' }G)
\end{align*}
and thus we   end our estimation of $ \bE[ \|  \langle D^2F, - DL^{-1}G \rangle_\mathcal{H} \|_\mathcal{H}^2  ]$ by using H\"older inequality and the contraction property of $P_t$ on $L^4(\Omega)$, that is,  using $\| P_t  (D_{r',z'}G)  \| _4 \leq \| D_{r',z'}G \|_4$.

To estimate the other expectation-term $ \bE[ \|  \langle DF, - D^2L^{-1}G \rangle_\mathcal{H} \|_\mathcal{H}^2  ]$, one can begin with
\[
 - D^2L^{-1}G = \int_0^\infty dt e^{-2t} P_t D^2G
\]
and then follow the same arguments.
\qedhere

\end{proof}



\begin{thebibliography}{99}


\bibitem{B12} Balan, R. M. (2012): The stochastic wave equation with multiplicative fractional noise: a Malliavin calculus approach. {\em Potential Anal.} {\bf 36}, 1-34.



\bibitem{BQS} Balan, R.M., Quer-Sardanyons, L. and Song, J. (2019):
Existence of density for the stochastic wave equation with space-time homogeneous
Gaussian noise. {\em Electron. J. Probab.} {\bf 24}, no. 106, 1-43.


\bibitem{BS17} Balan, R. M. and Song, J. (2017): Hyperbolic Anderson Model with space-time homogeneous Gaussian noise.
{\em ALEA Lat. Am. J. Probab. Math. Stat.} {\bf 14}, 799-849.




\bibitem{BNZ20} Bola\~nos-Guerrero, R., Nualart, D. and Zheng, G. (2021): Averaging 2D stochastic wave equation. \emph{Electron. J. Probab.} \textbf{26} (102): 1-32. 


\bibitem{BH86}

Bouleau N. and  Hirsch, F. (1986):  Propri\'et\'e d'absolue continuit\'e dans les espaces de Dirichlet et applications aux \'equations diff\'erentielles stochastiques. \emph{S\'eminaire de Probabilit\'es} \textbf{XX}: 12, 131-161, LNM 1204.


\bibitem{BM83}

  Breuer P.  and  Major P. (1983) : Central limit theorems for non-linear functionals of Gaussian fields. \emph{J.
Multivariate  Anal.} \textbf{13}, 425-441.

\bibitem{CN} Carmona, R. and Nualart, D. (1988): Random non-linear wave equations: smoothness of the
solutions. \emph{Probab. Theory Related Fields} {\bf 79}, 469-508.

\bibitem{CHA09}
 Chatterjee C.  (2009): Fluctuation of eigenvalues and second order Poincar\'e inequalities. \emph{Probab. Theory Related Fields} \textbf{143}, 1-40.



 


\bibitem{CKNP19-2}

Chen L. ,  Khoshnevisan D.,    Nualart D.  and  Pu F. (2022): Poincar\'e inequality, and central limit theorems for parabolic stochastic partial differential equations.  To appear in: \emph{Ann.  Inst.  Henri Poincar\'e  Probab.  Stat.} arXiv: 1912.01482


\bibitem{CKNP20}
Chen L. ,  Khoshnevisan D.,    Nualart D.  and  Pu F. (2021):   Central limit theorems for spatial averages of the stochastic heat equation via Malliavin-Stein's method.  
 {\it    Stoch. Partial Differ. Equ.  Anal. Comput.}


 \bibitem{Dalang99}
 R. C. Dalang (1999): Extending the martingale measure stochastic integral with applications to spatially homogeneous S.P.D.E.'s.  \emph{Electron. J. Probab.} \textbf{4},  no. 6,   1-29


\bibitem{DNZ20}
 Delgado-Vences, F.,   Nualart, D. and Zheng G. (2020):
 A central limit theorem for the stochastic wave equation with fractional noise. \emph{Ann. Inst. Henri Poincar\'e Probab. Stat.}, \textbf{56},  4, 3020-3042.


\bibitem{DGLZ20}
Dunlap, A., Gu, Y.,    Ryzhik, L. and  Zeitouni, O. (2020):   Fluctuations of the solutions to the KPZ equation in dimensions three and higher.   \emph{Probab.Theory  Related Fields}  {\bf 176}.


\bibitem{HP95}
Houdr\'e C.  and  P\'erez-Abreu V. (1995): Covariance identities and inequalities for functionals on Wiener and Poisson spaces.  \emph{Ann. Probab.} \textbf{23}, 400-419.



\bibitem{HNV20}
Huang J.,    Nualart D. and Viitasaari L. (2020):  A central limit theorem for the stochastic heat equation. \emph{Stochastic Process. Appl.} \textbf{130}, no. 12,   7170-7184.


\bibitem{HNVZ19}

Huang J.,    Nualart D.,  Viitasaari L. and  Zheng G. (2020):   Gaussian fluctuations for the stochastic heat equation with colored noise.   \emph{Stoch. Partial Differ. Equ.  Anal. Comput.}  \textbf{8}, 402-421.



\bibitem{OK}

  Kallenberg O. (2002):  \emph{Foundations of Modern Probability}. Second edition.  Probability and Its Applications, Springer.


\bibitem{KZ99}



 Karkzewska A. and    Zabczyk J.  (1999): \emph{Stochastic PDE's with function-valued solutions.}
In: \emph{Infinite-dimensional stochastic analysis} (Cl\'ement Ph., den Hollander F., van
Neerven J. \& de Pagter B., eds), pp. 197--216, Proceedings of the Colloquium
of the Royal Netherlands Academy of Arts and Sciences, Amsterdam.



\bibitem{KNP20}

  Khoshnevisan D.,    Nualart D.  and  Pu F. (2021): Spatial stationarity, ergodicity and CLT for parabolic Anderson model with delta initial condition in dimension $d\geq 1$. \emph{SIAM J. Math. Anal.} \textbf{53} no. 2,  2084-2133.


 \bibitem{KY20}

 Kim K. and Yi, J. (2022): Limit theorems for time-dependent averages of nonlinear stochastic heat equations. \emph{Bernoulli}   \textbf{28} (1): 214-238.
 

\bibitem{M76}

Malliavin, P. (1978):  Stochastic calculus of variations and hypoelliptic operators. \emph{Proceedings of the International Symposium on Stochastic Differential Equations} (Res. Inst. Math. Sci., Kyoto Univ., Kyoto, 1976). New York: Wiley. pp. 195-263.

\bibitem{MMS} M\'arquez-Carreras, D., Mellouk, M., Sarr\`a, M. (2001): On stochastic partial differential
equations with spatially correlated noise: smoothness of the law.
\emph{Stochastic Process. Appl.} {\bf 93}, 269-284.

\bibitem{MS}  Millet, A. and Sanz-Sol\'e, M. (1999): A stochastic wave equation in two space dimension:
smoothness of the law. \emph{ Ann. Probab.} {\bf 27}, 803-844.

\bibitem{NP09}

 Nourdin I.  and  Peccati G. (2009): Stein's method on Wiener chaos. \emph{Probab. Theory Related Fields}  \textbf{145}, no. 1, 75-118.



\bibitem{blue}
 Nourdin I.  and Peccati G. (2012):   \emph{Normal approximations with Malliavin calculus: from Stein's method to universality.} Cambridge Tracts in Mathematics \textbf{192}, Cambridge University Press.

\bibitem{NPR09}

 Nourdin I.,  Peccati G. and Reinert G. (2009): Second order Poincar\'e inequalities and CLTs on Wiener space. \emph{J. Funct. Anal.} \textbf{257}, 593-609.


\bibitem{Nualart06}
Nualart D. (2006): \emph{The Malliavin Calculus and Related Topics}, second edition. Probability and Its Applications,
Springer-Verlag Berlin Heidelberg.



\bibitem{NOL08}
Nualart D. and Ortiz-Latorre S. (2008):  Central limit theorems for multiple stochastic integrals and Malliavin calculus, \emph{Stochastic Process. Appl.}   \textbf{118} no 4,  614-628.


\bibitem{NP88}
 Nualart D.  and  Pardoux \'E. (1988): Stochastic calculus with anticipating integrands. \emph{Probab. Theory Related Fields} \textbf{78}, 535-581.

\bibitem{FMT05}

Nualart D. and  Peccati G. (2005):  Central limit theorems for sequences of multiple stochastic
integrals. \emph{Ann. Probab.} \textbf{33} no. 1, 177-193.

\bibitem{NQ} Nualart, D. and Quer-Sardanyons, L. (2007): Existence and smoothness of the density for
spatially homogeneous SPDEs. \emph{Potential Anal.} {\bf 27},
281-299.

\bibitem{NSZ20}

Nualart, D., Song X. M. and Zheng, G. (2021): Spatial averages for the Parabolic Anderson model driven by rough noise. \emph{ALEA Lat. Am. J. Probab. Math. Stat.} \textbf{18}, 907-943.

\bibitem{NZ19BM}

Nualart, D. and Zheng, G. (2020):  Averaging Gaussian functionals. {\em Electron. J. Probab.} \textbf{25}, no. 48, 1-54.


\bibitem{NXZ21}

Nualart, D., Xia, P. and Zheng, G. (2021):  Quantitative central limit theorems for the parabolic Anderson model driven by colored noises.   arXiv:2109.03875


\bibitem{NZ20}
 Nualart, D. and Zheng, G. (2021):
 Central limit theorems for stochastic wave equations in dimensions one and two.   \emph{Stoch.  Partial Differ. Equ.  Anal. Comput.}

\bibitem{NZ20erg}
Nualart, D. and Zheng, G. (2020):  Spatial ergodicity of stochastic wave equations in dimensions 1, 2 and 3. \emph{Electron. Commun. Probab.} \textbf{25}, no. 80, pages 1-11.





\bibitem{PT05}

Peccati G.  and Tudor C. A. (2005): Gaussian limits for vector-valued multiple stochastic integrals.
\emph{S\'eminaire de Probabilit\'es} \textbf{XXXVIII}. pp 247-262.


\bibitem{Pu20}
Pu F. (2021):  Gaussian fluctuation for spatial average of parabolic Anderson model with Neumann/Dirichlet/periodic boundary conditions.    \textit{Trans. Amer. Math. Soc.} 

\bibitem{QS1} Quer-Sardanyons, L., Sanz-Sol\'e, M. (2004): Absolute continuity of the law of the solution to
the 3-dimensional stochastic wave equation. \emph{J. Funct. Anal.}
{\bf 206}, no. 1, 1-32.

\bibitem{QS2}
Quer-Sardanyons, L., Sanz-Sol\'e, M. (2004): A stochastic wave
equation in dimension 3: Smoothness of the law. \emph{Bernoulli}
{\bf 10}, no. 1, 165-186.

\bibitem{SS}
Sanz-Sol\'e, M. and S\"uss, A. (2013): The stochastic wave equation
in high dimensions: Malliavin differentiability and absolute
continuity. \emph{Electron. J. Probab.} {\bf 18}, no. 64, 1-28

\bibitem{Anna}
Vidotto A. (2020): An improved second-order Poincar\'e inequality
for functionals of Gaussian fields.
\emph{J. Theoret. Probab.} \textbf{33}, 396-427.

\bibitem{Walsh}


   Walsh J.B.  (1986):
    \newblock {\it An Introduction to Stochastic Partial Differential Equations}.
    \newblock In: \'Ecole d'\'et\'e de probabilit\'es de
    Saint-Flour, XIV---1984, 265--439.
    Lecture Notes in Math.\ 1180, Springer, Berlin.



\bibitem{GZ18}
 Zheng  G. (2018): \emph{Recent developments around the Malliavin-Stein approach --- fourth moment phenomena via
exchangeable pairs.} Ph.D thesis, Universit\'e du Luxembourg. Available at  http://hdl.handle.net/10993/35536

 \end{thebibliography}
\end{document}